\documentclass[10pt,reqno]{amsart}

\usepackage{amssymb}
\usepackage{graphicx}
\usepackage{amsfonts}
\usepackage{amsmath}
\usepackage{amsthm}
\usepackage{fancyhdr}
\usepackage[left=3.2cm, right=3.2cm, bottom=3.2cm]{geometry}
\usepackage{indentfirst}
\usepackage[backref=page, colorlinks = true, linkcolor = black, citecolor = blue]{hyperref}
\usepackage{comment}
\usepackage{color}
\usepackage{subcaption}
\usepackage{epsfig}
\usepackage{pst-grad} 
\usepackage{pst-plot} 
\usepackage[space]{grffile} 
\usepackage{etoolbox} 
\usepackage{mathrsfs} 
\usepackage{mathtools} 
\mathtoolsset{showonlyrefs=true}
\usepackage{musicography}
\usepackage{enumerate}
\usepackage{dsfont}

\newtheorem{theorem}{Theorem}[section]
\newtheorem{lemma}[theorem]{Lemma}

\newtheorem{proposition}[theorem]{Proposition}
\newtheorem{corollary}[theorem]{Corollary}

\newtheorem*{claim*}{Claim}

\theoremstyle{definition}
\newtheorem{definition}[theorem]{Definition}
\newtheorem{example}[theorem]{Example}

\newtheorem{conjecture}[theorem]{Conjecture}
\newtheorem{question}{Question}

\theoremstyle{remark}
\newtheorem{remark}[theorem]{Remark}
\newtheorem{notation}[theorem]{Notation}
\newtheorem{assumptions}[theorem]{Assumptions}

\numberwithin{equation}{section}

\mathtoolsset{showonlyrefs=true}

\newcommand{\ep}{\varepsilon}

\newcommand{\R}{\mathbb{R}}
\newcommand{\N}{\mathbb{N}}

\newcommand{\Hcal}{\mathcal{H}}

\newcommand{\Acal}{\mathcal{A}}

\newcommand{\Scal}{\mathcal{S}}

\newcommand{\loc}{\mathrm{loc}}

\newcommand{\Gr}{\mathbf{G}}

\newcommand{\bV}{\mathbf{V}}
\newcommand{\bRV}{\mathbf{RV}}
\newcommand{\bCV}{\mathbf{CV}}
\newcommand{\bIV}{\mathbf{IV}}

\newcommand{\bC}{\mathbf{C}}

\newcommand{\tM}{\widetilde{M}}

\newcommand{\tScal}{\widetilde{\mathcal{S}}}

\newcommand{\oTh}{\overline{\Theta}}
\newcommand{\obC}{\overline{\bC}}

\newcommand{\weakto}{\rightharpoonup}

\newcommand{\mres}{\mathbin{\vrule height 1.6ex depth 0pt width
0.13ex\vrule height 0.13ex depth 0pt width 1.3ex}}
\newcommand{\ve}{v}

\newcommand{\VarTan}{\operatorname{VarTan}}

\newcommand{\spt}{\operatorname{spt}}
\newcommand{\dist}{\operatorname{dist}}
\newcommand{\Div}{\operatorname{div}}
\newcommand{\Tr}{\operatorname{Tr}}
\newcommand{\inj}{\operatorname{inj}}
\newcommand{\Hess}{\operatorname{Hess}}

\newcommand{\reg}{\operatorname{reg}}
\newcommand{\sing}{\operatorname{sing}}

\newcommand{\Ric}{\operatorname{Ric}}

\newcommand{\Span}{\operatorname{span}}

\newcommand{\Mat}{\operatorname{Mat}}
\newcommand{\dvol}{d\mathrm{vol}}

\NewDocumentCommand{\sff}{}{\mathrm{I\!I}}

\begin{document}

\title[Boundary behaviour of stable Allen-Cahn limit-interfaces]{Boundary behaviour of stable Allen-Cahn limit-interfaces in Riemannian manifolds}

\author[Martin Li]{Martin Man-chun Li}
\address{Department of Mathematics, The Chinese University of Hong Kong, Shatin, N.T., Hong Kong}
\email{martin.li@cuhk.edu.hk}

\author[Davide Parise]{Davide Parise}
\address{Department of Mathematics, Zeeman Building, University of Warwick, Gibbet Hill Road, Coventry CV4 7AL, UK}
\email{Davide.Parise@warwick.ac.uk}

\author[Lorenzo Sarnataro]{Lorenzo Sarnataro}
\address{Department of Mathematics, University of Toronto, 40 St George St., Toronto, Ontario M5S 2E4, Canada}
\email{lorenzo.sarnataro@utoronto.ca}

\begin{abstract}
This is the second paper in our series of works to develop the min-max theory for free boundary minimal hypersurfaces in Riemannian manifolds with boundary via the Allen-Cahn equation. Building on our previous work for general critical points, we study the boundary behaviour of limit-interface integral varifolds arising from stable (or finite Morse index) solutions to the Allen-Cahn equation with Neumann boundary condition. More precisely, we prove that the limit-interface is a curvature varifold with boundary in the sense of Mantegazza satisfying a generalized stability inequality. Furthermore, we rule out classical tangent cones at the boundary and show that the first variation of the limit-interface varifold is singular with respect to its weight measure. In particular, our result confirms a folklore conjecture that the boundary concentration examples of Malchiodi, Ni and Wei have unbounded Morse index. As will be evident in our forthcoming work, the results obtained in this paper are crucial ingredients towards the optimal boundary regularity in the Allen-Cahn free boundary min-max theory.  
\end{abstract}

\maketitle

\section{Introduction and Main Results}
\label{S:intro}

Let $(M^n,g)$ be an $n$-dimensional smooth compact Riemannian manifold with boundary. We consider the following (elliptic) Allen-Cahn equation with Neumann boundary condition:
\begin{equation}
\label{E:Allen-Cahn-intro}
    \left\{ \begin{array}{cl}
            -\ep^2 \Delta_g u +W'(u)=0 & \text{in $M$} \\
            \frac{\partial u}{\partial \nu} =0 & \text{on $\partial M$}
    \end{array} \right.
\end{equation}
where $\Delta_g$ is the Laplace-Beltrami operator on $(M,g)$, $\nu$ is the unit outward normal of $\partial M$ and $W:\mathbb{R} \to \mathbb{R}$ is a smooth double well potential such as $W(t)=\frac{1}{4}(1-t^2)^2$. Here, $\ep>0$ is a small parameter. A solution $u \in C^\infty(M)$ to \eqref{E:Allen-Cahn-intro} is said to be \emph{stable} if it satisfies in addition the stability inequality, i.e. 
\begin{equation}
\label{E:stability-intro}
\int_M \ep |\nabla \psi|^2 + \frac{1}{\ep} W''(u) \psi^2 \; \dvol_g \geq 0 \quad \text{for all $\psi \in C^\infty(M)$.}
\end{equation}
The main goal of this paper is to study the (local) boundary behaviour of the limit-interface varifold arising from a sequence of stable solutions $u_{\ep_i}$ to \eqref{E:Allen-Cahn-intro} with $\ep_i \to 0$. 

\subsection{Background and motivations}
\label{SS:motivations}

We begin with a discussion on the background and some recent development in the field which motivates our study here. 

\subsubsection{Bernstein problems and De Giorgi Conjecture}
\label{SSS:De-Giorgi}

The parabolic counterpart of the Allen-Cahn equation \eqref{E:Allen-Cahn-intro} is a reaction-diffusion equation first proposed in phase transition theory to describe the process of phase separation in multi-component alloy systems. In the sharp interface limit as $\ep \to0$, heuristically the limit-interface separating stable phases would be evolving under mean curvature flow whose stationary solutions are given by minimal surfaces. Since the 1970s, the theory of $\Gamma$-convergence \cite{Modica-Mortola-1977, Modica1987a, Kohn1989} pioneered by De Giorgi showed a deep connection between the Allen-Cahn equation and minimal surface theory. Motivated by the classical Bernstein problem in minimal surface theory, De Giorgi made several famous conjectures about similar rigidity results for entire solutions to the Allen-Cahn equation. Progress made around De Giorgi's conjectures have driven some of the major developments for semi-linear partial differential equations (see \cite{Chan2018} for an excellent survey up to 2018) in the past few decades.  

In recent years, we have witnessed significant breakthroughs on the \emph{stable} Bernstein problem for minimal hypersurfaces \cite{Chodosh-Li2024, Chodosh-Li-Minter-Stryker2026, Mazet2024, Bellettini2025, Hong-Li-Wang2026}. As a result, we now know that stable two-sided complete minimal hypersurfaces in $\mathbb{R}^n$ are flat for $n \leq 7$ while non-flat examples exist for $n \geq 8$. Nevertheless, the analogous De Giorgi conjecture for \emph{stable} entire solutions to the Allen-Cahn equation remains unsolved in its full generality  (see, however, \cite{Florit-Simon-Serra2025,Florit-Simon2026,LLWWWW2026,CFFFS2026} for some remarkable recent progress). On the other hand, variants of the stable De Giorgi conjecture have been confirmed affirmatively in certain low dimensional settings (c.f. \cite{Cabre-Figalli-RosOton-Serra2020, Figalli-Serra2020, Cabre2026, Chan2026}).

\subsubsection{Min-max theory via Allen-Cahn equation}
\label{SSS:min-max-AC}

The compactness and regularity of stable (or finite Morse index) solutions to the Allen-Cahn equation is closely related to the existence theory of minimal hypersurfaces by global Morse-theoretic min-max constructions. Building on the fundamental work of \cite{Hutchinson2000,Tonegawa2012}, Guaraco \cite{Guaraco2018} pioneered a PDE approach via Allen-Cahn equation to construct closed embedded minimal hypersurfaces in closed Riemannian manifolds, as an alternative to the classical Almgren-Pitts approach based on geometric measure theory. This new PDE-based approach has proven powerful especially in low dimensions (e.g. \cite{Chodosh2020, Chodosh2023}) by exploiting novel curvature estimates first discovered by Wang-Wei \cite{Wang-Wei2019}. See also the recent work \cite{DDPS2026} of the second and third-named author with their collaborators in the free boundary setting.

\subsubsection{Free boundary minimal hypersurfaces}
\label{SSS:free-boundary}

This paper constitutes part of the broader program initiated by the same authors in \cite{LiPariseSarnataro2023} to develop a PDE-based min-max theory for free boundary minimal hypersurfaces via the Allen-Cahn equation on Riemannian manifolds with boundary. In the free boundary setting, the Almgren-Pitts min-max theory in the geometric measure theoretic framework was developed by the first-named author with X. Zhou in \cite{Li2021}. In contrast to the boundaryless situation, a technically subtle but major difficulty is to understand the boundary concentration phenomena known as \emph{touching} or \emph{improperness} as observed in \cite{Li2015}. In the context of Allen-Cahn equations, such boundary concentrating solutions have been constructed rigorously by Malchiodi, Ni and Wei \cite{Malchiodi2007} inside the unit ball and by Malchiodi-Wei \cite{Malchiodi2007a} inside any strictly mean convex domain in Euclidean space. More precisely, for any given $N \in \mathbb{N}$, \cite{Malchiodi2007} established the existence of a sequence of radially symmetric solutions $\{u_{\ep_i}\}$ to \eqref{E:Allen-Cahn-intro} for $M=\mathbb{B}^n$ and $W(t)=\frac{1}{4}(1-t^2)^2$ such that the zero level sets $\{u_{\ep_i}=0\}$ converge as $\ep_i \to 0$ to the boundary sphere $\partial \mathbb{B}^n$ with multiplicity $N$. Moreover, each solution $u_{\ep_i}$ has Morse index equal to $N$ within the class of radially symmetric functions. Nonetheless, it is a folklore conjecture that if we consider the bigger class of functions which are not necessarily radially symmetric, the Morse index should become unbounded in these examples.

\begin{conjecture}
\label{Conj:MNW-examples}
    The solutions $u_{\ep_i}$ constructed by Malchiodi, Ni and Wei in \cite{Malchiodi2007} and \cite{Malchiodi2007a} have unbounded Morse index as $\ep_i \to 0$.
\end{conjecture}

Heuristically, the boundary mean-convexity suggests that the limit-interface should be unstable in the direction along the inward normal of the domain boundary. By imposing oscillatory Fourier modes along the direction tangential to the domain boundary, this would then create infinitely many independent directions for the second variation. However, it is very difficult to make this quantitative as a-priori the level sets $\{u_{\ep_i}=0\}$ only converges to the domain boundary in a very weak topology; and it is ambiguous to make sense of the Morse index of the limit-interface (in our case here is simply the domain boundary) as it is not a genuine minimal surface in $\mathbb{R}^n$. To the best of our knowledge, the exact Morse index of the examples in \cite{Malchiodi2007} and \cite{Malchiodi2007a} has yet to be computed.

\subsection{Main results}
\label{SS:main-results}

By our previous work \cite{LiPariseSarnataro2023}, for any sequence of solutions $\{u_{\ep_i}\}$ to \eqref{E:Allen-Cahn-intro} satisfying natural uniform energy and $L^\infty$ bounds (see Assumptions \ref{As:stationary-AC}), we can associate to each $u_{\ep_i}$ a codimension-one varifold $V^{u_{\ep_i}}$ in $M$ such that, after passing to a subsequence as $\ep_i \to 0$, $V^{u_{\ep_i}}$ converges as varifolds to a limit varifold $V$, known as a \emph{limit-interface}\footnote{Note that the limit-interface $V$ may depend on the choice of the subsequence of $\ep_i$.}. Moreover, we know from \cite{LiPariseSarnataro2023} that $V$ is integer-rectifiable up to the boundary and is free boundary stationary as a varifold in $M$ (we refer the readers to Section \ref{S:varifolds} for the precise definitions). 

We now describe the main results obtained in this paper which establish further analytic and geometric properties of $V$ under the additional assumptions that each $u_{\ep_i}$ is \emph{stable} (see Assumptions \ref{As:stable-AC}). See Theorem \ref{T:main-AC} and \ref{T:exclude-classical-cones} for more precise statements.

\begin{theorem}
    \label{T:main-theorem}
    Let $\varepsilon_i \to 0$  and suppose $\{u_{\ep_i}\}$ is a sequence of stable solutions to the Allen-Cahn equation \eqref{E:Allen-Cahn-intro} with $E_{\ep_i}(u_{\ep_i})\leq E_0$ and $\|u_{\ep_i}\|_{L^\infty(M)} \leq C_0$ for some constants $E_0,C_0>0$. Then, for any limit-interface $V$ obtained from a subsequential limit of $V^{\ep_i}$ as above, the following hold:
    \begin{itemize}
        \item[(1)] $V$ is a curvature varifold with boundary in the sense of Mantegazza (c.f. Definition \ref{D:curvature-varifolds})
        \item[(2)] $V$ satisfies a stability-type inequality: for any $\psi \in C^1_c(M)$, we have
\[
    \int |\mathbf{B}_V|^2\psi^2 \, d\|V\| \leq \int |\nabla\psi|^2 \, d\|V\| + C_M \int \psi^2 \, d\|V\| + C_M \int \psi^2 \, d \|\delta V\|_{sing},
\]
where $\mathbf{B}_V$ is the generalized second fundamental form of $V$ (c.f. Definition \ref{D:curvature-varifolds})  and $C_M \geq 0$ is a geometric constant depending only on $M$.
\item[(3)] $V$ is stationary across the ``touching set'', i.e. the first variation of $V$ is given by
\[
    \delta V(X)= \int X \cdot \nu \; d \|\delta V\|_{sing} 
\]
where $X$ is any $C^1$ vector field on $M$ which is not necessarily tangential to $\partial M$.
\item[(4)] $V$ admits no classical tangent cones except possibly on an $\Hcal^{n-2}$-null set of $\spt \|V\|$.
    \end{itemize}
\end{theorem}

\begin{remark}
\label{R:minimal-across-touching-set}
    To understand (3), note that if $V$ is the unit-density varifold corresponding to the domain boundary $\partial M$, then it is automatically a free boundary stationary varifold with free boundary in $\partial M$ (c.f. Definition \ref{D:free-boundary-varifold}). However, it will satisfy (3) only when $\partial M$  is itself a minimal hypersurface (since in this case $\|\delta V\|_{sing}=0$ as $\partial M$ itself has no boundary). Therefore, (3) is saying in a weak sense (which can be made precise by \eqref{E:1st-variation-fbv-stationary}) the limit-interface $V$ must be ``minimal'' on the part where it leans onto $\partial M$ (c.f. Remark \ref{R:stationarity-across-touching}). 
\end{remark}

Because of Remark \ref{R:minimal-across-touching-set}, boundary concentration for stable limit-interface cannot happen on the non-minimal part of $\partial M$. Combining with a standard trick due to R. Schoen (c.f. \cite{Sharp2017}) by considering disjoint small open balls, boundary concentration can happen only at at most $N$ \emph{non-minimal} points on $\partial M$ if a limit-interface arises from solutions $\{u_{\ep_i}\}$ of the Allen-Cahn equation \eqref{E:Allen-Cahn-intro} with Morse index uniformly bounded above by $N$. Note that the examples considered in \cite{Malchiodi2007, Malchiodi2007a} have strictly mean-convex domain boundary, as an immediate consequence of Theorem \ref{T:main-theorem} (3) we have:

\begin{corollary}
    \label{C:main-corollary}
    Conjecture \ref{Conj:MNW-examples} is true.
\end{corollary}

\begin{remark}
    (4) is the free boundary analogue of a fundamental result of Tonegawa-Wickramasekera \cite[Proposition 3.2]{Tonegawa2012} for the interior case in the Euclidean setting. Here, we extend their result up to the boundary and in the Riemannian setting. In another paper \cite{Bellettini-Li-Parise-Sarnataro} of the same authors with Bellettini, we established the compactness and regularity theory for stable free boundary minimal hypersurfaces under an a-priori smallness assumption on the singular set in the spirit of Schoen-Simon (c.f. \cite{SchoenSimon81} and \cite{Bellettini2025}). In our forthcoming work, where we develop the full free boundary optimal regularity of Wickramasekera's far-reaching interior compactness and regularity theory \cite{Wick14}, (3) is the key variational hypothesis and (4) will be used to verify the main structural hypothesis on the absence of classical singularities. 
\end{remark}

\subsection{Main ideas and key challenges}
\label{SS:new-ideas}

We now give a more detailed discussion on the major difficulties and some of the novel ideas in proving our main result Theorem \ref{T:main-theorem} above (see also Theorem \ref{T:main-AC} and \ref{T:exclude-classical-cones} for more precise statements).

\subsubsection{Complications due to boundary interaction}
\label{SSS:boundary-interaction}

While the overall strategy to prove Theorem \ref{T:main-theorem} is similar to the one laid out in the foundational work of Tonegawa and Wickramasekera in \cite{Tonegawa2005a, Tonegawa2012} which established the same results in the interior Euclidean setting (see also \cite{Guaraco2018} for an extension to the Riemannian setting), the extra Neumann boundary condition in \eqref{E:Allen-Cahn-intro} interacts in a highly non-trivial way with the semi-linear Allen-Cahn equation in the interior. To better understand the new challenges due to such boundary interaction, let us first recall the main steps of \cite{Tonegawa2005a, Tonegawa2012} in establishing the optimal regularity of stable Allen-Cahn limit-interfaces in the interior setting (we shall restrict to the Euclidean setting for now and leave the discussions about the Riemannian setting until Section \ref{SSS:intrinsic-approach}).

\begin{itemize}
    \item[Step 1:] The first step is to study the limit-interface $V$ obtained as the limit of varifolds associated to general critical points $u_{\ep}$ of the Allen-Cahn energy functional with uniform energy and $L^\infty$ bounds. It was shown in \cite{Hutchinson2000} that $V$ must be an integer-rectifiable stationary varifold. The key element here is an almost monotonicity formula at the $\ep$-level which relies on a pointwise upper bound on the discrepancy function.
    \item[Step 2:] The second step is to analyse stable solutions to the Allen-Cahn equation \eqref{E:Allen-Cahn-intro} at the $\ep$-level. This was done in \cite{Padilla1998}. Here an $\ep$-stability inequality satisfied by the enhanced second fundamental form of $u_{\ep}$ is obtained.
    \item[Step 3:] The next step is to pass $\ep \to 0$ in the $\ep$-stability inequality to infer that the limit-interface $V$ is a curvature varifold in the sense of Hutchinson \cite{HutchinsonIUMJ}. This allows the generalized second fundamental form of $V$ to be defined and a stability inequality holds by passing $\ep \to 0$ the $\ep$-stability inequality from Step 3. This was done in \cite{Tonegawa2005a}.
    \item[Step 4:] As an intermediate step towards the regularity of the stable limit-interface $V$, one needs to analyse in detail the top two strata of the singular set according to the Almgren-Federer stratification. In particular, one needs to rule out the existence of classical tangent cones on a generic subset of singular points. This is done in \cite{Tonegawa2012} by a blow-up and contradiction argument, using the stability inequality from Step 3 together with some fine analysis in \cite{Hutchinson2000}.
    \item[Step 5:] The final step is to apply the far-reaching Schoen-Simon-Wickramasekera's compactness and regularity theory \cite{SchoenSimon81, Wick14} to conclude that the support of the limit-interface $V$ is smooth except possibly on a small singular set of codimension $7$. 
\end{itemize}

We now highlight some of the major new challenges in the presence of boundary interaction. Due to the boundary concentration examples in \cite{Malchiodi2007a,Malchiodi2007}, the study of general critical points in Step 1 has to be robust enough to cover these situations. This was done in our previous paper \cite{LiPariseSarnataro2023}. Unlike the interior Euclidean case treated in \cite{Hutchinson2000}, the discrepancy function only satisfies a weaker pointwise upper bound $\xi_\ep \leq \ep^{-4/5}$ (as opposed to $\xi_\ep \leq C$ in \cite{Hutchinson2000}) near the domain boundary. Using this weaker pointwise bound, the almost monotonicity formula is only effective down to the scale of order $\ep^{4/5}$. In order to improve it down to the scale of order $\ep$, we utilize in \cite{LiPariseSarnataro2023} careful covering arguments to derive the desired variational and regularity properties of the limit-interface up to the domain boundary based on integral estimates. In particular, the limit-interface is free boundary stationary and integer-rectifiable up to the boundary. Among other things, we also established that the discrepancy function $\xi_\ep$ converges to zero in the strong $L^1$-sense up to the domain boundary, which implies ``equi-partition of energy'' as in \cite{Hutchinson2000}. This finishes Step 1 in the setting with boundary.

Steps 2-4 are treated in this paper. For Step 2, we proceed as in \cite{Padilla1998} to obtain an $\ep$-stability inequality for the enhanced second fundamental form of $u_\ep$ with an additional boundary term involving the second fundamental form of the domain boundary, which arises from integrating by parts of the Bochner formula. For Step 3, the notion of curvature varifolds in \cite{HutchinsonIUMJ} is insufficient for our purpose here, as the limit-interface is supposed to contain a free boundary in our setting. Therefore, we need to adopt instead the concept of curvature varifold with boundary by Mantegazza \cite{Mantegazza1996} to describe the limit-interface. One technical difference to the interior setting of \cite{Tonegawa2005a} is that we do not know whether the restriction of the discrepancy function $\xi_\ep$ to the domain boundary still converges to zero strongly in $L^1$ with respect to the domain boundary. As a result, we do not have equi-partition of energy when restricted to the domain boundary and thus several naturally defined Radon measures on the domain boundary may be inequivalent. To overcome this difficulty, we employed a local $L^1$-trace inequality for the energy density in \cite{LiPariseSarnataro2023} to control the extra boundary term in the $\ep$-stability inequality in order to obtain uniform $L^2$-bounds on the enhanced second fundamental form of $u_\ep$ to carry out the compactness arguments. When passing $\ep \to 0$ the $\ep$-stability inequality, the singular part can be identified with appropriate quantities from the first variation of $V$. This relies on some fine boundary properties of curvature varifold with boundary established in \cite{Mantegazza1996}. From this we also infer a novel conclusion that the limit-interface $V$ must be stationary across the part where it leans on the domain boundary.

For Step 4, we introduce the definition of boundary classical tangent cones and classify them into two types, depending on whether the spine of the cone is contained in or perpendicular to the domain boundary. While the second type can be treated in a similar way to the interior setting, the first type depicts boundary concentration phenomenon which has to be studied separately. We carry out in detail here the blow-up and contradiction arguments as in \cite{Tonegawa2012}, but using the ideas to treat boundary concentration as developed in \cite{LiPariseSarnataro2023}. In the final contradiction argument, we need to argue that the planar curves obtained from generic slices of level sets of $u_\ep$ have to turn by a definite positive amount (depending solely on the boundary classical tangent cone). In contrast to the boundaryless setting in \cite{Tonegawa2012}, besides closing up or having their ends going to infinity, such a planar curve may also end on the domain boundary. This is the place where we need to make use of the Neumann boundary condition, which translates to the geometric fact that such a planar curve has to hit the domain boundary orthogonally. By carefully exhausting all possible cases on the ``ends'' of such planar curves, we can derive the same contradiction as in \cite{Tonegawa2012}.

One technical point in the above arguments is that one requires \emph{up-to-the-boundary} Hausdorff convergence of the intermediate level sets of $u_\ep$ to the support of the limit-interface $V$. This was not established in \cite{LiPariseSarnataro2023} since the almost monotonicity formula only holds for balls that are either centred at a boundary point or are completely disjoint from the boundary. This is the case because we chose to give a much streamlined approach to avoid the use of reflected balls in \cite{LiPariseSarnataro2023}. In our current paper, we also supply the arguments using reflected balls in Appendix \ref{A:reflected-density} and show the full Hausdorff convergence of the intermediate level sets up to the boundary. Moreover, in Appendix \ref{app:rectifiability-integrality} we give an alternative proof on the integer-rectifiability of \emph{stable} limit-interfaces based on the approach in \cite{DRP25}. The advantage is that it is much shorter than \cite{LiPariseSarnataro2023} but it relies heavily on stability.

\subsubsection{Intrinsic approach on varifolds in Riemannian manifolds}
\label{SSS:intrinsic-approach}

For our purpose of developing an Allen-Cahn min-max construction of free-boundary minimal hypersurfaces in Riemannian manifolds, as in \cite{LiPariseSarnataro2023} we established all the results in this paper in the setting of Riemannian manifolds. There are two existing mainstream approaches in the literature to generalize the theory of varifolds in Euclidean spaces to the Riemannian setting. One is via Nash's isometric embedding as was done in \cite{Allard1972, Naber-Valtorta2020}; another one is via an intrinsic approach as was done in \cite{SchoenSimon81, Scharrer2022}. We adopt the second approach in this paper for two reasons. The first is concerning the required regularity assumption on the Riemannian metric. The existence of isometric embedding requires at least $C^{2,\alpha}$ regularity of the Riemannian metric while one only requires $C^2$ regularity for the intrinsic approach. The second is that the boundary value problem for isometric embedding is not canonically well-posed. One can get around by first extending our Riemannian manifold $M^n$ with boundary into a closed Riemannian manifold $\widetilde{M}^n$ of the same dimension, and then apply Nash's isometric embedding theorem to the extended Riemannian manifold (c.f. \cite{Li2021}). Although the same conclusions can be drawn as far as regularity is concerned, much of the geometric information of the original Riemannian manifold $M$ is obscured due to the arbitrary choice of the extension $\widetilde{M}$ and the rather obscure isometric embedding. For example, it would be difficult to obtain effective quantitative estimates in terms of geometric quantities such as curvatures or second fundamental form of $M$. 

In the first part of this paper, we systematically develop the foundational theory of varifolds in the intrinsic setting. For example, we derive the monotonicity formula, the existence and properties of tangent cones, the stratification of the singular set, as well as the notion of free-boundary varifolds and curvature varifolds with boundary in the intrinsic Riemannian setting. Apart from identifying the natural coordinate-free definitions, a new significant difficulty arises in the intrinsic approach due to the fact that a varifold in the Riemannian manifold $M$ which has locally bounded first variation (or even stationary), when written as a varifold in the tangent space under geodesic normal coordinates, may not even have locally bounded first variation (with respect to the flat metric)! Therefore, a-priori one cannot apply any of the classical results like the compactness theorem of integral varifolds in \cite{Allard1972} to such varifolds expressed in a geodesic normal coordinate chart. This is because the Riemannian metric $g$ only agrees with the flat metric \emph{at origin} up to first order in geodesic normal coordinates. The variational information (such as stationarity) is lost once we move slightly away from the origin. This obstruction was first observed in \cite{SchoenSimon81} and some of the related results have been established in their setting. However, they made use of some additional assumptions such as stability or sheeting (c.f. \cite[Section 5, Remark 12]{SchoenSimon81}) to derive their results. Here in this paper, we establish all the results such as integrality of tangent cones, under merely first-order assumption such as locally bounded first variation in the original Riemannian manifold $(M^n,g)$. We believe that these fundamental results, that are shown to hold in such generality, will be useful in other geometric contexts (c.f. for example, \cite{Simon2023} for some recent breakthrough on the singular set of minimal hypersurfaces in the Rimemannian setting). 

\subsection{Organization of the paper}
\label{SS:organization}

The rest of the paper is organized into three sections. In Section \ref{S:prelim}, we describe our geometric setup and state various definitions, notations and concepts that will be used in this paper. In Section \ref{S:varifolds}, we systematically develop the theory of varifolds in Riemannian manifold using an intrinsic approach. We also treat in detail the boundary behaviour in the context of free-boundary varifolds and curvature varifolds with boundary. In Section \ref{S:Allen-Cahn}, we prove all the main results (Theorem \ref{T:main-theorem} in local form) concerning properties of the stable limit-interfaces arising from Allen-Cahn equation. Finally, we collect in the appendix various results that supplement the content of this paper.

\subsection*{Acknowledgements} M. Li is substantially supported by research grants from the Research Grants Council of the Hong Kong Special Administrative Region, China [Project No.: CUHK 14306624, CUHK 14308325 and RFS2526-4S02]. He would like to express his gratitude to Prof. Rick Schoen for his constant support and encouragement. D. Parise was partially supported by the Cecilia Tanner Research Impulse Fund from Imperial College London and would like to thank Luca Spolaor for several useful conversations on the topic. L. Sarnataro would like to thank Robert Haslhofer and Yevgeny Liokumovich for their support, and Paul Minter and Alessandro Pigati for several useful conversations. Part of the work here arises from some preliminary discussions while the authors were in residence or visiting at the Simons Laufer Mathematical Sciences Institute (formerly MSRI) in Berkeley, California, during the Fall 2024 Semester. All the authors would like to thank Prof. Yoshihiro Tonegawa for his interest and kind encouragement.

\section{Definitions, Notations and Preliminaries}
\label{S:prelim}

In this section, we state various basic definitions, introduce some notations and recall certain preliminary concepts that will be used throughout the rest of this paper. 

\subsection{Ambient geometry}
\label{SS:setup}

Let $(\widetilde{M}^n,g)$ be a smooth\footnote{In fact all the results in this paper are valid provided that the metric $g$ is $C^2$ and $M$ is a manifold of class $C^3$. For the sake of readability, we will assume that $g$ is $C^\infty$ throughout this paper.} complete $n$-dimensional Riemannian manifold \emph{without boundary}, and let $\nabla^g$ denote the Levi-Civita connection on $\widetilde{M}$ induced by $g$. We shall sometimes use $\langle \cdot, \cdot \rangle_g$ to denote the inner product $g$ on the tangent spaces of $\widetilde{M}$ or simply use $\cdot$ when there is no ambiguity. We use $B^g_r(x)$ to denote the geodesic ball in $(\tM,g)$ of radius $r$ centred at $x$. For any $0 \leq k \leq n$, the $k$-dimensional Hausdorff measure on $(\widetilde{M},g)$ is denoted by $\mathcal{H}^k$. Note that $\mathcal{H}^n$ agrees with the volume measure $\dvol_g$ on $(\widetilde{M},g)$. As in \cite{LiPariseSarnataro2023}, we will always assume $(\widetilde{M},g)$ has positive injectivity radius lower bound 
\begin{equation}
    \label{E:inj-bound}
\inj_g(\widetilde{M})\geq i_0  \quad \text{for some constant $i_0 > 0$.} 
\end{equation} 
and a two-sided sectional curvature bound 
\begin{equation}
\label{E:sec-bound}
|\sec_{(\widetilde{M},g)}| \leq k_0 \quad \text{for some constant $k_0 \geq 0$.}
\end{equation}
We shall omit the subscripts or superscripts $g$ if the metric is clear from the context.

Let $\Omega \subset \widetilde{M}$ be an open and connected domain with smooth\footnote{Again, $C^2$ regularity of $\partial \Omega$ would suffice for our purposes.} \emph{non-empty} boundary $\partial \Omega$, whose outward unit normal is denoted by $\nu$. As in \cite{LiPariseSarnataro2023}, we will always assume that $\Omega$ has bounded ball curvatures (c.f. \cite[Section 2.2]{LiPariseSarnataro2023}) 
\begin{equation}
\label{E:ball-curvatures}
\max\{|\overline{Z}_\Omega|, |\underline{Z}_\Omega|\} \leq \kappa_0 \quad \text{for some constant $\kappa_0 \geq 0$.}
\end{equation}
Under this assumption, the signed boundary distance function $d(\cdot):=\dist_g(\cdot, \partial \Omega)$, with $d>0$ in $\Omega$, is smooth in the tubular neighbourhood $U_{\frac{1}{2\kappa_0}}(\partial \Omega)$ of size $\frac{1}{2\kappa_0}$ around $\partial \Omega$ and satisfies the following:
\begin{equation}
\label{E:dist-estimates}
|\nabla d|_g =1, \quad (\Hess_g d) \, (\nabla d, \cdot)=0 \quad \text{ and } \quad \| \Hess_g d \, \| \leq C_1(n,k_0,\kappa_0).
\end{equation}
Moreover, for any boundary point $x \in \partial \Omega$, the distance function $r(\cdot):=\dist_g(\cdot,x)$ satisfies
\begin{equation}
\label{E:allard-estimate}
|r \nabla r \cdot \nu |(y) \leq \frac{\kappa_0}{2} r^2(y) \quad \text{for all $y \in \partial \Omega \cap B_{\frac{i_0}{2}}(x)$.}
\end{equation}


\subsection{Functions, vector fields and measures}
\label{SS:functions}

Since our arguments will be local, we let $U \subset \widetilde{M}$ be an open subset (sometimes of the Grassmannian manifold). We will use the following notations to denote various vector spaces of functions and vector fields: 
\[\begin{array}{cl}
C_c(U) & \text{real-valued continuous functions on $U$ with compact support}\\
C^k_c(U) & \text{real-valued $C^k$ functions on $U$ with compact support}\\
L^p(U;\mu) & \text{real-valued $L^p$ functions on $U$ with respect to a measure $\mu$}\\
L^p_{loc}(U;\mu) & \text{real-valued locally $L^p$ functions on $U$ with respect to a measure $\mu$}\\
\mathfrak{X}_c(U) & \text{$C^1$ vector fields on $U$ with compact support}\\
\mathfrak{X}^{tan}_c(U) & \text{$X \in \mathfrak{X}_c(U)$ such that $X \cdot \nu=0$ along $\partial \Omega \cap U$}
\end{array}\]


\subsection{Euclidean geometry}
\label{SS:Euclidean}

We will use $(\mathbb{R}^n,\delta)$ to denote the standard $n$-dimensional Euclidean space with the flat metric $\delta$ and the flat covariant derivative $D$. For any $p \in \mathbb{R}^n$ and $r>0$, we denote the homothety map $\eta_{p,r}:\mathbb{R}^n \to \mathbb{R}^n$ by
\begin{equation} 
\label{E:homothety}
\eta_{p,r}(q):=\frac{q-p}{r}.
\end{equation}
In particular, $\eta_{p,1}$ is the translation which takes $p$ to the origin, and $\eta_{0,r}$ is the dilation by a factor $r^{-1}$ with respect to the origin. Given any unit vector $v \in \mathbb{R}^n$, we use $\mathbb{R}^n_+(v) := \{u \in \mathbb{R}^n \; | \; u \cdot v \geq 0\}$ to denote the closed half-space and $\theta_v : \mathbb{R}^n \rightarrow \mathbb{R}^n$ to denote the reflection about $v$ defined by 
\begin{equation} 
\label{E:reflection}
\theta_v(p) := p - 2 (p \cdot v) v. 
\end{equation}
We shall omit the dependence on $v$ when there is no ambiguity. For the purpose of this paper, the most important case we have in mind is the flat Euclidean space $(T_x \tM,g_x)$  for $x \in \partial \Omega$ with $v=-\nu(x)$. We will denote the $k$-dimensional Lebesgue measure on $\mathbb{R}^k$ by $\mathcal{L}^k$.

\subsection{Cutoffs and constants}
\label{SS:cutoff}

Throughout this paper, we will use $\chi:\mathbb{R} \to [0,\infty)$ to denote any $C^3$ even (i.e. $\chi(t)=\chi(-t)$ for all $t \in \mathbb{R}$) cut-off function satisfying the following requirements: 
\begin{equation}
    \label{E:cutoff}
\chi(t) = \left\{ \begin{array}{cl}
1 & \text{ for $0 \leq t \leq \frac{1}{2}$}\\
0 & \text{ for $t \geq 1$}
\end{array} \right. \; \text{ and } \chi^\prime(t) \leq 0 \text{ for every $t>0$.}
\end{equation}
We shall use the convention that $C(\alpha,\beta,\gamma,\cdots)$ denotes a non-negative constant that depends only on the parameters $\alpha$, $\beta$, $\gamma$, $\cdots$.

\section{Intrinsic varifolds in Riemannian manifolds} 
\label{S:varifolds}

We recall some definitions from geometric measure theory and refer the readers to the standard references \cite{Allard1972,Simon1983} for more details. Most of the existing literature on varifolds is considered in the framework of Euclidean ambient space. One common way to extend these to the Riemannian setting is via Nash isometric embedding as in \cite{Allard1972}, which requires at least $C^{2,\alpha}$ regularity of the metric with $\alpha>0$ \cite[Section 8]{Jacobowitz72}. Another way is to develop the whole theory from the intrinsic point of view as, for example, in \cite{Scharrer2022}, which only requires the metric to be $C^2$. The latter also has the advantage that we can obtain estimates with constants depending explicitly on geometric quantities such as curvatures and injectivity radius bounds.

Let $\Gr_x(n,k)$ be the space of unoriented $k$-dimensional subspaces of $T_x \widetilde{M}$. For $T \in \Gr_x(n,k)$, we identify $T$ with the corresponding orthogonal projection (with respect to $g_x$) of $T_x \widetilde{M}$ onto $T$. The orthogonal complement of $T \in \Gr_x(n,k)$ and its corresponding orthogonal projection will be denoted by $T^\perp \in \Gr_x(n,n-k)$. In the case of $k=n-1$, we also identify $T \in \Gr_x(n,n-1)$ with the unit vector $\pm N_T \in T_x \widetilde{M}$ which is perpendicular to $T$. Note that we may express the relation by $T=I- N_T \otimes N_T$ where $I$ is the identity map on $T_x \widetilde{M}$. We shall mainly consider codimension-one varifolds in this paper but many of our discussions can be carried over to the case of general dimension and codimension.

For any open set $U \subset \widetilde{M}$, let $G(U):=\cup_{x \in U} \Gr_x(n,n-1)$ be the Grassmannian bundle over $U$ which is a locally compact separable metric space. An $(n-1)$-dimensional varifold in $U$ is simply a Radon measure on $G(U)$. We denote the set of all $(n-1)$-dimensional varifolds in $U$ by $\bV_{n-1}(U)$, equipped with the usual weak topology. For $V \in \bV_{n-1}(U)$, let $\|V\|$ be the weight of $V$, which is the Radon measure on $U$ defined by
\[ \|V\| (\phi) :=\int_{G(U)} \phi(x) \; dV(x,T), \qquad \forall \phi \in C_c(U).\]
We say $V \in \bV_{n-1}(U)$ is rectifiable if there exists an $\Hcal^{n-1}$-measurable, countably $(n-1)$-rectifiable subset $\Sigma \subset U$ and a locally $\Hcal^{n-1}$-integrable non-negative function $\Theta_V$ defined on $\Sigma$ such that 
\[ V(\phi)=\int_\Sigma \phi(x,T_x \Sigma) \; \Theta_V(x)  d\Hcal^{n-1}(x), \qquad \forall \phi \in C_c(G(U)).\]
Here $T_x \Sigma \in \Gr_x(n,n-1)$ is the unique approximate tangent space of $\Sigma$ at $x$ which exists $\Hcal^{n-1}$-a.e. on $\Sigma$.  A rectifiable varifold $V$ is uniquely determined by its weight $\|V\|=\Theta_V \Hcal^{n-1} \mres \Sigma$ through the formula above. For this reason, we naturally say a Radon measure $\mu$ on $U$ is rectifiable when one can associate a rectifiable varifold $V$ such that $\|V\|=\mu$. If $\Theta_V(x) \in \N$ for $\Hcal^{n-1}$-a.e. $x \in \Sigma$, we say $V$ is integer rectifiable or an integral varifold. The set of all $(n-1)$-dimensional rectifiable varifolds and integral varifolds in $U$ is denoted respectively by $\bRV_{n-1}(U)$ and $\bIV_{n-1}(U)$. We say that $V$ is a unit-density varifold if $\Theta_V(x)=1$ for $\|V\|$-a.e. $x$. For any $C^1$ embedded hypersurface $\Sigma \subset U$ with (possibly empty) boundary, we will denote $|\Sigma| \in \bIV_{n-1}(U)$ the integral varifold associated to $\Sigma$ with multiplicity one.

\begin{remark}
\label{R:rectifiability}
Note that the definition of a varifold does not really require the Riemannian metric $g$. In fact, the notion of rectifiability and integer-rectifiability is independent of the choice of $g$. The reason is that (integer) rectifiability is bi-Lipschitz invariant and any Riemannian metric is infinitesimally isometric to the flat Euclidean space. In particular, if $F:U \to \widetilde{U}$ is a $C^1$ proper map between two open sets of Riemannian manifolds, then the pushforward map $F_\sharp:\bV_{n-1}(U) \to \bV_{n-1}(\widetilde{U})$, as defined in \cite[Section 39]{Simon1983}, satisfies $F_\sharp(\bRV_{n-1}(U)) \subset \bRV_{n-1}(U)$ and $F_\sharp(\bIV_{n-1}(U)) \subset \bIV_{n-1}(U)$ (c.f. \cite[3.5(3)]{Allard1972} in the context of isometric embedding or \cite[I.1 (13)]{Almgren1976}).
\end{remark}

For $V \in \bV_{n-1}(U)$ let $\delta V$ be the first variation of $V$, namely,
\begin{equation}
\label{E:1st-variation}
\delta V(X):=\left. \frac{d}{dt}\right|_{t=0} \|(F_t)_\sharp V\|(U)=\int_{G(U)} \Div_T X(x) \; dV(x,T), \qquad \forall X \in \mathfrak{X}_c(U)
\end{equation}
where $\Div_T X(x):=\sum_{i=1}^{n-1} \nabla_{e_i} X(x) \cdot e_i$ for any orthonormal basis $\{e_1,\cdots,e_{n-1}\}$ of $T \in \Gr_x(n,n-1)$ with respect to the Riemannian metric $g$, and $\{F_t\}_{t \in \mathbb{R}}$ is the flow generated by $X$. We also denote the total variation measure (see \cite[(39.3)]{Simon1983}) of $\delta V$ by $\|\delta V\|$. If $\|\delta V\|$ is a Radon measure on $U$ (i.e. $V$ has locally bounded first variation in $U$), by Riesz representation and the Radon-Nikodym theorem there exists a $\|V\|$-measurable vector field $H_V$, a $\|\delta V\|$-measurable vector field $\eta_V$ with $|\eta_V|=1$ $\| \delta V\|$-a.e., and a Borel set $Z_V \subset U$ such that $\|V\|(Z_V)=0$ and
\begin{equation}
\label{E:Radon-Nikodym} 
\delta V(X)=-\int_U X(x) \cdot H_V(x) \; d\|V\|(x) +\int_{Z_V} X(x) \cdot \eta_V(x) \; d\|\delta V\|_{sing}(x)
\end{equation}
for all $X \in \mathfrak{X}_c(U)$, where by $\|\delta V\|_{sing}$ we denote the singular part of $\|\delta V\|$ with respect to $\|V\|$. We call $H_V$ the generalised mean curvature vector, $\eta_V$ the generalized outward co-normal, $\|\delta V\|_{sing}$ the generalized boundary measure, and $Z_V$ the generalized boundary of $V$. Note that in general we have $H_V \in L^1_{loc}(U;\|V\|)$ and $\|\delta V\|_{sing}$ is a Radon measure on $U$. We may omit the subscript $V$ when there is no ambiguity.

\begin{remark}
In contrast to Remark \ref{R:rectifiability}, the notion of first variation $\delta V$ is a geometric quantity that depends on the Riemannian metric $g$. This is intuitively clear as the first variation represents in a generalized sense the mean curvature vector and the boundary with its outward unit normal which are both geometric quantities. In particular, the notion of locally bounded first variation is not a diffeomorphism invariant. This is more evident when there is a singular part of the first variation. For example, consider the one-dimensional rectifiable varifold in $\mathbb{R}^2$ in \cite[Example 2, p.256]{Allard-Almgren1973} which is stationary in a ball centred at the origin, if we apply a diffeomorphism to this ball which preserves the origin but distorts the angles at all the triple junctions, we will end up with a rectifiable varifold with unbounded first variation near the origin. 
\end{remark}

We next describe the concept of tangent varifolds in the Riemannian setting from the intrinsic point of view. Given $x\in U$ and $0<\rho<\frac{1}{2}\min\{i_0, \frac{1}{\kappa_0}\}$, the exponential map $\exp_x$ is a diffeomorphism\footnote{The exponential map is a $C^2$ diffeomorphism if the metric $g$ is of class $C^2$.} from the Euclidean ball $B^\delta_\rho(0)\subset (T_x \tM,g_x)$ onto the geodesic ball\footnote{Because of \eqref{E:ball-curvatures}, we know that $B_\rho(x) \cap \partial \Omega$ is a connected, embedded hypersurface diffeomorphic to an $(n-1)$-dimensional ball which separates $B_\rho(x)$ into two connected components, one of which is contained in $\Omega$.} $B_\rho(x) \subset (\tM,g)$. Let $V\in \bV_{n-1}(U)$. We can consider the pushforward varifold $\overline{V}:=(\exp_x^{-1}|_{B_\rho(x)})_\sharp V \in \bV_{n-1}(T_x \tM)$ inside the Euclidean space $(T_x \tM,g_x)$. If (recall \eqref{E:homothety})
\begin{equation}
\label{E:tangent-varifolds}
C=\lim_{i\to \infty} (\eta_{0,r_i})_\sharp \overline{V} \in \bV_{n-1}(T_x\tM) 
\end{equation}
is a weak limit for a sequence of positive numbers $r_i \to 0$, then we say that $C$ is a varifold tangent to $V$ at $x$, and write $C\in \VarTan(V,x) \subset \bV_{n-1}(T_x\tM)$, where $\VarTan(V,x)$ denotes the set of all possible varifold tangents to $V$ at $x$. Note that $(T_x \tM,g_x)$ can be identified with the Euclidean space $(\mathbb{R}^n,\delta)$ for which the usual theory of varifolds applies. 

\begin{remark}
    Note that one can also define the set of tangent varifolds via isometric embedding as in \cite{Allard1972}. If the metric $g$ is regular enough (say $C^{2,\alpha}$ regularity), the two notions of tangent varifolds are equivalent. 
    Since we chose to avoid using Nash's isometric embeddeding in this paper, we adopt the above intrinsic definition instead. Another formulation is to consider the rescaling of the metric $r_i^{-2}g$ as done in \cite{LiPariseSarnataro2023} which is again equivalent as $d\exp_x|_0$ is the identity map of $T_x\tM$. 
\end{remark}

\subsection{Free boundary varifolds} 
\label{SS:FB-varifolds}

To study the fine properties close to the boundary, we shall also need the notion of free boundary varifolds defined in \cite[Section 3]{Edelen2020} (see also \cite[Definition 2.4]{Masi2020}). Since both \cite{Edelen2020} and \cite{Masi2020} are stated in the Euclidean setting, we will recall several useful results concerning this sub-class of varifolds and point out the differences (if any) in the Riemannian setting. 

\begin{definition}[Free boundary varifolds]
\label{D:free-boundary-varifold}
Given a smoothly\footnote{It is sufficient to assume $C^2$ regularity of $S$.} embedded hypersurface $S \subset U$ oriented by the (outward) unit normal $\nu_S$, a varifold $V \in \bV_{n-1}(U)$ is said to have free boundary in $S$ \footnote{As in \cite{Edelen2020,Masi2020}, this definition implicitly assumes that $\delta V$ has locally bounded first variation with $\|\delta V\|_{sing}=0$ away from the support hypersurface $S$.} if there exists a $\|V\|$-measurable vector field $H^{tan}_V \in L^1_{loc}(U;\|V\|)$ and $H^{tan}_V(x) \cdot \nu_S (x) =0$ for $\|V\|$-a.e. $x \in S$, satisfying
\begin{equation}
\label{E:free-boundary-varifold}
\delta V(X)= - \int_U X(x)\cdot H^{tan}_V(x) \; d\|V\|(x) \qquad \text{$\forall$ $X \in \mathfrak{X}_c(U)$ such that $X \cdot \nu_S=0$ on $S$.}
\end{equation}
If $\delta V(X)=0$ for all $X \in \mathfrak{X}_c(U)$ such that $X \cdot \nu_S=0$ on $S$ (in other words, $H^{tan}_V=0$ $\|V\|$-a.e.), we say that $V$ is free boundary stationary with free boundary in $S$.
\end{definition}

Note that we use the notation $H^{tan}_V$ to distinguish it from the generalised mean curvature vector $H_V$ defined in \eqref{E:Radon-Nikodym}. In fact, if such an $H_V$ exists as in \eqref{E:Radon-Nikodym} with $\|\delta V\|_{sing}=0$, then $H^{tan}_V$ is indeed the tangential projection of $H_V$ along $S$ (see \cite[Definition 3.0.3]{Edelen2020}). We also remark that $H^{tan}_V$ in Definition \ref{D:free-boundary-varifold} is uniquely defined $\|V\|$-a.e. with the condition $H^{tan}_V(x) \cdot \nu_S (x) =0$ for $\|V\|$-a.e. $x \in S$ (c.f. \cite[Remark 2.5]{Masi2020}).

\begin{remark} Unlike in \cite{Masi2020} (and implicitly in \cite{Edelen2020}), according to our Definition \ref{D:free-boundary-varifold}, the varifold $V$ is not required to lie on one side (determined by the unit normal $\nu_S$) of the support hypersurface $S$. Although most of the time we will only consider $S=\partial \Omega \cap U$ and $V \in \bV_{n-1}(U)$ with $\spt \|V\| \subset \overline{\Omega} \cap U$, having this extra flexibility is convenient for us in certain rescaling arguments in the Riemannian setting as will be evident in the next subsection.
\end{remark}

The following elementary but fundamental result was obtained by \cite[Corollary 4.7]{Masi2020} (and \cite[Proposition 3.2]{Edelen2020} for integral varifolds) in the Euclidean setting. Roughly speaking, it says that free-boundary varifolds are varifolds with locally bounded first variation together with certain prescribed Neumann boundary condition along the singular part $\|\delta V\|_{sing}$ of the first variation measure.

\begin{proposition}[Characterization of free boundary varifolds]
\label{P:free-boundary-characterization}
Let $V \in \bV_{n-1}(U)$. Then $V$ has free boundary in $S$ if and only if $V$ has locally bounded first variation in $U$ with $\spt \|\delta V\|_{sing} \subset S$ and $\eta_V(x)=\pm \nu_S(x)$ for $\|\delta V\|_{sing}$-a.e. $x \in S$.
\end{proposition}

\begin{proof}
First, note that the same constancy lemma \cite[Lemma 3.1]{Masi2020} and the main theorem \cite[Theorem 1.1]{Masi2020} hold in the Riemannian setting by \eqref{E:dist-estimates}. Therefore, the proposition follows immediately as in \cite[Corollary 4.7]{Masi2020} from the following formula:
\begin{equation} 
\label{E:1st-variation-fbv-refined}
\delta V(X) = - \int_{U} X(x)\cdot (H^{tan}_V(x) + H_S (x)) \; d \|V\|(x) + \int_{S} X(x)\cdot \eta_V(x) \; d\|\delta V\|_{sing}(x), 
\end{equation} 
for all $X \in \mathfrak{X}_c(U)$, where $H_S$ is the mean curvature vector of $S$ (which by convention vanishes outside $S$) and $\eta_V(x)=\pm \nu_S(x)$ for $\|\delta V\|_{sing}$-a.e. $x$. Note that $H^{tan}_V$ and $H_S$ are mutually orthogonal for $\|V\|$-a.e. $x$. We want to remark that we have $\eta_V=\nu_S$ only when $V$ is assumed to be supported on one side of $S$ (c.f. the last paragraph in Step 4 of the proof of Theorem 1.1 in \cite{Masi2020}). In general, we only have $\eta_V=\pm \nu_S$.
\end{proof}

\begin{remark}
Given the main results in \cite{LiPariseSarnataro2023} (see Theorem \ref{T:old-main}), we shall only be concerned with free boundary varifolds that are also integral, we shall from now on implicitly make that assumption, even though it is not strictly necessary for some of the general results in this section. We shall refer to \cite{Masi2020} for a more general treatment for free boundary varifolds which are not necessarily integral (or even rectifiable). 
\end{remark}

From now on, we will mainly focus on the case that $V \in \bV_{n-1}(U)$ has free boundary in $\partial \Omega \cap U$ with $\spt \|V\| \subset \overline{\Omega} \cap U$. Although we do not have a monotonicity formula for $\|\delta V\|_{sing}$, by \cite[Theorem 1.1]{Masi2020} the following one-sided local bound holds:
\begin{equation} 
\label{E:local-volume-bound-fbv}
    \|\delta V\|_{sing}(B_{r/2}(x)) \leq \frac{C_2(n,k_0,\kappa_0)}{r} \|V\| (B_r(x)) + \int_{B_r(x)} |H^{tan}_V| \; d \|V\|, 
\end{equation}
for all $x \in \partial \Omega \cap U$ and $0< r < \min \left\{\dist_g(x,\partial U), \frac{1}{2\kappa_0} \right\}$. 

We will mostly be interested in varifolds that are \textit{free boundary stationary} with free boundary in $\partial \Omega \cap U$, i.e. varifolds $V\in\bV_{n-1}(U)$ with the additional property that $\delta V(X)=0$ for all $X\in\mathfrak{X}_c^{tan}(U)$. 
For such free boundary stationary varifolds $V\in\bV_{n-1}(U)$ with $\spt \|V\| \subset \overline{\Omega} \cap U$, \eqref{E:1st-variation-fbv-refined} becomes 
\begin{equation}
\label{E:1st-variation-fbv-stationary}
\delta V(X) = - \int_{\partial \Omega \cap U} X(x)\cdot  H_{\partial \Omega}(x) \; d \|V\| (x) + \int_{\partial \Omega \cap U} X(x)\cdot \nu(x) \; d\|\delta V\|_{sing}(x)
\end{equation}
for all $X \in \mathfrak{X}_c(U)$. Note that both terms on the right hand side are supported along $\partial \Omega \cap U$. Similarly, \eqref{E:local-volume-bound-fbv} simplifies to
\begin{equation}
\label{E:local-volume-bound-fbv-stationary}
\|\delta V\|_{sing}(B_{r/2}(x)) \leq \frac{C_2(n,k_0,\kappa_0)}{r} \|V\| (B_r(x)).    
\end{equation}

For general free boundary varifolds that are not necessarily stationary, we can still obtain the following weighted version of \eqref{E:local-volume-bound-fbv}.

\begin{lemma}[$L^1$-trace inequality for free boundary varifolds]
\label{L:trace-inequality}
    Let $V \in \bV_{n - 1}(U)$ be a varifold with free boundary in $\partial \Omega \cap U$ and $\spt \|V\| \subset \overline{\Omega} \cap U$. For any nonnegative $f \in C_{c}^{1}(U)$, we have 
    \begin{equation} 
    \label{equation: refined isop general I}
        \int f \; d\|\delta V\|_{sing} \leq C_3 \int \left( ( 1 + | H^{tan}_V| ) f 
        + |\nabla f| \right) \; d\|V \|
    \end{equation}
    where $C_3=C_3(n,i_0,k_0,\kappa_0)>0$. 
\end{lemma}
\begin{proof}
We will basically follow the arguments in \cite[Subsection 3.2]{Masi2020} adapted to the local Riemannian setting. Fix $f \in C^1_c(U)$. Since $f$ has compact support in $U$, a straightforward covering argument shows that it suffices to establish \eqref{equation: refined isop general I} locally around any $x_0 \in \partial \Omega \cap U$. Let $r=\dist_g(\cdot,x_0)$ and $d=\dist_g(\cdot,\partial \Omega)$. We test \eqref{E:1st-variation-fbv-refined} with the compactly-supported vector field 
\[ X(x) = - f(x) \chi\left(\frac{r(x)}{R} \right) \nabla d(x)\]
where $\chi$ is the cut-off function defined in Section \ref{SS:cutoff}, and $R = \frac{1}{2} \min\{ i_0,\frac{1}{2\kappa_0}\}$. Thus, we have the following
    \begin{align*}
        \begin{split}
            \int_{B_{R/2}(x_0)} f(x) \; & d\|\delta V\|_{sing}(x)  \leq \int f(x) \chi\left(\frac{r(x)}{R} \right)  \; d\|\delta V\|_{sing}(x) = \int X(x) \cdot \nu(x)  \; d\|\delta V\|_{sing}(x) \\
            & = \int \mathrm{div}_S X(x) \; dV(x, S) + \int X(x)\cdot \left(H^{tan}_V(x) + H_{\partial \Omega} (x) \right) \; d \|V\| (x).
        \end{split}
    \end{align*}
On the other hand, a direct computation shows $\mathrm{div}_S X(x)$ is equal to
    \begin{equation*}
        - f(x) \chi^\prime\left(\frac{r(x)}{R} \right) \frac{1}{R} \left( S(\nabla r) \cdot \nabla d\right)(x) - f(x) \chi\left( \frac{r(x)}{R} \right) \mathrm{div}_S \nabla d(x) - \chi\left( \frac{r(x)}{R} \right) \left( \nabla d \cdot S(\nabla f) \right) (x). 
    \end{equation*}
Note that $| \mathrm{div}_S \nabla d | \leq (n-1)C_1(n,k_0,\kappa_0)$ in $B_R(x_0)$ by \eqref{E:dist-estimates}. Furthermore, $| H_{\partial \Omega} | \leq (n-1)\kappa_0$ by \eqref{E:ball-curvatures}. Therefore, we can bound
    \begin{align*}
        - \int f(x) \chi\left( \frac{r(x)}{R} \right) & \mathrm{div}_S \nabla d(x) \; dV(x, S)  - \int f(x) \chi\left(\frac{r(x)}{R} \right) \nabla d(x) \cdot \left( H^{tan}_V(x) + H_{\partial \Omega}(x) \right)   d \| V \|(x) \\
        & \leq C(n,k_0,\kappa_0) \int_{B_R(x_0)} \left(1 +  |H^{tan}_V(x)| \right)f(x) \; d\| V \|(x). 
    \end{align*}
The other terms appearing in $\mathrm{div}_S X(x)$ can be controlled by fixing a cutoff function such that $\| \chi^\prime \|_{L^\infty} \leq 3$ and the observation that $|\nabla r|=|\nabla d|=1$. This concludes the proof. 
\end{proof}

Taking $f = \psi^2$ in \eqref{equation: refined isop general I} and applying Cauchy's inequality, we have the following simple but important consequence of Lemma \ref{L:trace-inequality}.

\begin{corollary}
\label{C:free-bdy-trace-ineq}
For every $\eta >0$, there exists $C_4(\eta,n,i_0,k_0,\kappa_0)>0$ such that the following holds: let $V \in \bV_{n - 1}(U)$ be a varifold with free boundary in $\partial \Omega \cap U$ and $\spt \|V\| \subset \overline{\Omega} \cap U$, for any $\psi \in C^1_c(U)$, we have
  \begin{equation}
        \int \psi^2\, d\|\delta V\|_{sing} \leq (1+\eta) \int |\nabla\psi|^2\,d\|V\| + C_4 \int \left( 1+ |H^{tan}_V| \right) \psi^2 \, d\|V\|. 
    \end{equation}
\end{corollary}

We also recall a useful reflection principle and constancy lemma for free boundary stationary varifolds from \cite{Li2021} and \cite{Masi2020}, adapted to our setup and terminology. We will mostly apply the lemma below at the level of tangent cones, hence we can take $\mathbb{R}^n=T_x \tM$ and $v=-\nu$ (recall the notations from \eqref{E:reflection} for the lemma below).

\begin{lemma}[{\cite[Lemma 2.2]{Li2021}}]
\label{L:reflection stationary}
    For any $V \in \bV_{n-1}(\mathbb{R}^n)$ with $\spt V \subset \mathbb{R}^n_+$, define the doubled-varifold 
    \begin{equation*}
        \overline{V} := V + (\theta_v)_\sharp V \in \bV_{n-1}(\mathbb{R}^n). 
    \end{equation*}
    If $V$ is free boundary stationary with free boundary in $\partial \mathbb{R}^n_+$, then $\overline{V}$ is stationary in $\mathbb{R}^n$. 
\end{lemma}

\begin{lemma}[{\cite[Lemma 3.1 and Corollary 4.3]{Masi2020}}]
\label{L:constancy lemma De Masi}
 Let $V \in \bV_{n - 1}(U)$ be a varifold with free boundary in $\partial \Omega \cap U$. Then, for any $\phi \in C_c(G(U))$, we have
 \[ \int_{G(U)} \phi(x,S) \; dV(x,S) = \int_{\partial \Omega \cap U} \phi(x,T_x\partial \Omega) \Theta^{n-1}(\|V\|,x) \; d\mathcal{H}^{n-1}(x) + \int_{G(U \setminus \partial \Omega)} \phi(x,S) \; dV(x,S).\]
 In particular, $\|V\| \, \mres_{\partial \Omega}= \Theta^{n-1}(\|V\|,\cdot) \; d \mathcal{H}^{n-1} \, \mres_{\partial \Omega}$ and $V$ is rectifiable along $\partial \Omega \cap U$ with 
    \begin{equation}
    \label{E:constancy-lemma}
        V \left( \{(x,S)\in G(U) \, | \, x \in \partial \Omega \cap U, S \neq T_x \partial \Omega \} \right)=0.
    \end{equation}
\end{lemma}

We shall use Lemma \ref{L:reflection stationary} and Lemma \ref{L:constancy lemma De Masi} to prove the following technical result.

\begin{lemma}
\label{L:stationarity planar tangents}
    Using the same notation as in Lemma \ref{L:reflection stationary}, and suppose $V\in \bV_{n-1}(\mathbb{R}^n)$ is free boundary stationary with free boundary in $\partial \mathbb{R}^n_+$.
    Assume, in addition, that $\spt \|V\| \subset \partial \mathbb{R}^n_+$. Then $V$ is stationary in $\mathbb{R}^n$ and thus $V=m |\partial \mathbb{R}^n_+|$ for some real constant $m >0$.
\end{lemma}

\begin{proof}
By the classical constancy theorem \cite[41.1]{Simon1983}, it suffices to show that $V$ is stationary in $\mathbb{R}^n$. We claim that $(\theta_v)_\sharp V=V$. Once the claim is established, we see from Lemma \ref{L:reflection stationary} that the doubled varifold $\overline{V}=V+(\theta_v)_\sharp V=2V$ is stationary in $\mathbb{R}^n$. This of course implies stationarity of $V=\frac{1}{2} \overline{V}$ in $\mathbb{R}^n$.

It remains to establish the claim that $(\theta_v)_\sharp V=V$. Since $v \in \mathbb{R}^n$ and $\partial \mathbb{R}^n_+=\{u\in \mathbb{R}^n \, | \, u \cdot v=0\}$, for all $x\in \partial \mathbb{R}^n_+$ we have $\theta_v(x)=x$ and $D_x\theta_v$ acts as the identity map on $\partial \mathbb{R}^n_+$. Moreover, notice that for any $(y,S) \in G(\mathbb{R}^n)$, the Jacobian (see \cite[\S 39]{Simon1983}) is given by $J_S \theta_v (y)=1$. Therefore, by the definition of pushforward of a varifold (see again \cite[\S 39]{Simon1983}), we have for all $f \in C^0_c (G(\mathbb{R}^n))$
    \begin{equation}
    \begin{split}
        \int_{G(\mathbb{R}^n)}f(z, T) \; d((\theta_v)_\sharp V)(z,T) &= \int_{G(\mathbb{R}^n)} f(\theta_v(y), D_y\theta_v(S))J_S\theta_v(S) \; dV(y, S)\\
        &=\int_{\partial \mathbb{R}^n_+ \times \Gr(n,n-1)}f(\theta_v(y), D_y\theta_v(S)) \; dV(y,S)\\
        &\qquad + \int_{G(\mathbb{R}^n \setminus \partial \mathbb{R}^n_+)} f(\theta_v(y), D_y\theta_v(S)) \; dV(y,S).
    \end{split}
    \end{equation}
    Using \eqref{E:constancy-lemma}, $\theta_v(y)=y$ for $y \in \partial \mathbb{R}^n_+$ and the fact that $D_y\theta_v(S) \neq \partial \mathbb{R}^n_+$ for all $y \in \partial \mathbb{R}^n_+$, $S \neq \partial \mathbb{R}^n_+$, we see that 
    \begin{equation}
        \int_{\partial \mathbb{R}^n_+ \times \Gr(n,n-1)}f(\theta_v(y), D_y\theta_v(S)) \; dV(y,S) =\int_{\partial \mathbb{R}^n_+ \times \{\partial \mathbb{R}^n_+\}} f(y, \partial \mathbb{R}^n_+) \; dV(y,S).
    \end{equation}
On the other hand, since $\spt \| V \| \subset \partial \mathbb{R}^n_+$ by assumption, we have 
    \begin{equation}
        \int_{G(\mathbb{R}^n \setminus \partial \mathbb{R}^n_+)} f(\theta_v(y), D_y\theta_v(S)) \; dV(y,S) =0.
    \end{equation}
Combining all these and Lemma \ref{L:constancy lemma De Masi}, we have established the claim that $(\theta_v)_\sharp V=V$.
\end{proof}

\begin{remark}
\label{R:stationarity planar tangents}
In fact, the above result is local and the same statement holds when $\mathbb{R}^n$ (and respectively $\partial \mathbb{R}^n_+$) is replaced by any open set $U \subset \mathbb{R}^n$ (and respectively $U \cap \partial \mathbb{R}^n_+$) which is invariant under the reflection $\theta_v$.
\end{remark}

\subsection{Monotonicity formula and tangent cones} 
\label{SS:Monotonicity}

We will assume mostly in this section that $V \in \bIV_{n-1}(U)$ is a free boundary stationary (integral) varifold with free boundary in $\partial \Omega \cap U$ and $\spt \|V\| \subset \overline{\Omega} \cap U$. The goal is to recall various monotonicity formulae and look at their immediate consequences, especially regarding the basic properties of the varifold tangents $C \in \VarTan(V,x)$.

We first recall the following (almost) monotonicity formula for geodesic balls for free boundary varifolds with an $L^\infty$ bound on $H^{tan}_V$. Note that we do not need to assume $V$ is integral.

\begin{lemma}[Monotonicity formula]
\label{L:monotonicity limit varifold}
    For every $\Lambda \geq 0$, there exists $C_5(n, k_0, \kappa_0, \Lambda)>0$ such that for any varifold $V \in \bV_{n-1}(U)$ with free boundary in $\partial \Omega \cap U$, $\spt \|V\| \subset \overline{\Omega} \cap U$ and $\|H^{tan}_V\|_{L^\infty(U)} \leq \Lambda$, the function
    \begin{equation}
    \label{E:monotonicity limit varifold}
        \rho \mapsto e^{C_5 \rho}\frac{\|V\|(B_\rho(x))}{\rho^{n-1}}
    \end{equation}
    is monotone increasing for $0<\rho< \frac{1}{2} \min \{ i_0, \frac{1}{\kappa_0}, \frac{1}{\sqrt{k_0}},2\dist_g(x,\partial U) \}$ when $x \in \partial \Omega \cap U$; and for $0< \rho <\frac{1}{2} \min \{ i_0, \frac{1}{\sqrt{k_0}}, 2\dist_g(x,\partial \Omega \cup \partial U) \} $ when $x \in U \setminus \partial \Omega$. In particular, the $(n-1)$-dimensional density
    \begin{equation}
        \Theta_V(x):=\lim_{\rho\to 0}\frac{\|V\|(B_\rho(x))}{\omega_{n-1}\rho^{n-1}}
    \end{equation}
    exists for all $x\in U$ (and vanishes outside $\spt \|V\|$).
\end{lemma}

\begin{proof}
The proof (in the Euclidean setting) is essentially contained in \cite[Section 4.5]{Masi2020}. We include the main steps here for our readers' convenience and to adapt the calculations in the Riemannian setting. We will only focus on the case $x \in \partial \Omega$ as the interior case can be found for example in \cite{Scharrer2022}. The key idea is to test \eqref{E:1st-variation-fbv-refined} with suitable test vector fields. In what follows, we denote as in Section \ref{SS:setup} that $r(\cdot)=\dist_g(\cdot,x)$ and $d(\cdot)=\dist_g(\cdot,\partial \Omega)$ be the intrinsic distance function from $x$ and $\partial \Omega$ respectively. Recall the cutoff function $\chi$ defined in section \ref{SS:cutoff}.

Fix $x\in\partial\Omega\cap U$ and let $\rho$ be as in the lemma, so that $\overline{B_\rho(x)}\subset U$. Set
\[
I(\rho):=\int\chi\left(\frac r\rho\right)\,d\|V\|,
\qquad
I'(\rho)=\int-\frac r{\rho^2}\chi'\left(\frac r\rho\right)\,d\|V\|\geq0,
\]
and let
\[
A:=\Lambda+(n-1)(\kappa_0+\sqrt{k_0}),\qquad
B:=\Lambda+(n-1)(\kappa_0+C_1),
\]
where $C_1=C_1(n,k_0,\kappa_0)$ is as in \eqref{E:dist-estimates}. First consider
\[
X=\chi\left(\frac r\rho\right)r\nabla r\in\mathfrak{X}_c(U).
\]
By Hessian comparison (c.f. \cite[(2.1)]{LiPariseSarnataro2023}),
\[
\left|\Div_S X-(n-1)\chi\left(\frac r\rho\right)
-\frac r\rho\chi'\left(\frac r\rho\right)|S(\nabla r)|^2\right|
\leq(n-1)\sqrt{k_0}\,r\chi\left(\frac r\rho\right).
\]
Thus \eqref{E:1st-variation-fbv-refined}, \eqref{E:allard-estimate}, and $|H_{\partial\Omega}|\leq(n-1)\kappa_0$ give
\begin{equation}\label{E:monotonicity-proof-1}
\begin{split}
&\int\left((n-1)\chi\left(\frac r\rho\right)
+\frac r\rho\chi'\left(\frac r\rho\right)|S(\nabla r)|^2\right)\,dV(y,S)\\
&\hspace{2cm}\leq A\rho I(\rho)
+\frac{\kappa_0}{2}\int r^2\chi\left(\frac r\rho\right)\,d\|\delta V\|_{sing}.
\end{split}
\end{equation}
To estimate the boundary term while retaining its quadratic factor, use
\[
\tilde X=-r^2\chi\left(\frac r\rho\right)\nabla d\in\mathfrak{X}_c(U).
\]
Since $-\nabla d=\nu$ on $\partial\Omega$, and the support of $V$ lies in $\overline\Omega$, the boundary term in \eqref{E:1st-variation-fbv-refined} is now exactly $\int r^2\chi(r/\rho)\,d\|\delta V\|_{sing}$. Moreover, \eqref{E:dist-estimates} gives
\[
|\Div_S\tilde X|\leq(n-1)C_1r^2\chi\left(\frac r\rho\right)
+2r\chi\left(\frac r\rho\right)
-\frac{r^2}{\rho}\chi'\left(\frac r\rho\right).
\]
Consequently,
\begin{align*}
\int r^2\chi\left(\frac r\rho\right)\,d\|\delta V\|_{sing}
&\leq\int(Br^2+2r)\chi\left(\frac r\rho\right)\,d\|V\|
+\int-\frac{r^2}{\rho}\chi'\left(\frac r\rho\right)\,d\|V\|\\
&\leq(B\rho^2+2\rho)I(\rho)+\rho^2 I'(\rho).
\end{align*}
Here the term $\Lambda$ in $B$ controls $\tilde X\cdot H^{tan}_V$.

Since $\chi'\leq0$ and $|S(\nabla r)|\leq1$, the left-hand side of \eqref{E:monotonicity-proof-1} is at least $(n-1)I(\rho)-\rho I'(\rho)$. Combining the two estimates gives
\[
\rho\left(1+\frac{\kappa_0\rho}{2}\right)I'(\rho)
\geq\left(n-1-(A+\kappa_0)\rho-\frac{\kappa_0B}{2}\rho^2\right)I(\rho).
\]
In particular, \begin{equation}\label{E:monotonicity-diff-inequality}
\frac{d}{d\rho}\left(\frac{I(\rho)}{\rho^{n-1}}\right)
\geq-\frac{A+\frac{n+1}{2}\kappa_0+\frac{\kappa_0B}{2}\rho}
{1+\frac{\kappa_0\rho}{2}}\frac{I(\rho)}{\rho^{n-1}}
\geq-C_5\frac{I(\rho)}{\rho^{n-1}},
\end{equation}
where we may choose any $C_5>0$ with
\[
C_5\geq\max\left\{A+\frac{n+1}{2}\kappa_0,\,B\right\}.
\]
Hence $e^{C_5\rho}I(\rho)/\rho^{n-1}$ is nondecreasing. For centres away from $\partial\Omega$, the radial test field alone gives the interior assertion with the same $C_5$.
\end{proof}

\begin{remark} 
\label{R:monotonicity}
We would like to point out that the monotonicity formula above holds only for ambient geodesic balls that are either centred at a boundary point on $\partial \Omega$ or is completely disjoint from the boundary $\partial \Omega$. A standard argument (c.f. \cite[Corollary 17.8]{Simon1983}) then shows that the density function $\Theta_V$ is ``partially'' upper semi-continuous at the boundary in the sense that $\limsup_{x_i \to x} \Theta_V(x_i) \leq \Theta_V(x)$ in general holds for $x \in \partial \Omega$ only when $x_i \in \partial \Omega$ for all $i$. A free boundary half-plane contained in the half-space of $\mathbb{R}^n$ illustrates this point as the boundary points have density $\frac{1}{2}$ and the interior points have density $1$. However, as already observed in the proof of \cite[Theorem 5.7]{LiPariseSarnataro2023}, one can also use the arguments of reflected balls as in \cite[Corollary 3.2]{Grueter1986} (c.f. Appendix \ref{A:reflected-density}) to conclude that the \textit{reflected density} defined by 
\begin{equation}\label{E:reflected-density}
    \oTh_V(x):= \left\{\begin{matrix}\begin{aligned}
        &\Theta_V(x) &\text{if }x\notin\partial\Omega \cap U \\ 
        &2\Theta_V(x) &\text{if }x\in\partial\Omega\cap U
        \end{aligned}\end{matrix}\right.
\end{equation}
is actually an upper semi-continuous function in $x$.  In particular, if $V$ is further assumed to be integral, i.e. $V \in \bIV_{n-1}(U)$, then we have the following uniform density lower bound\footnote{In fact, $\Theta_V(x) \geq 1$ for all $x \in \spt \|V\| \cap \Omega$. For boundary points $x \in \spt \|V\| \cap \partial \Omega$ which are in the closure of $\spt \|V\| \cap \Omega$, it follows from upper semi-continuity of the reflected density. For boundary points $x \in \spt \|V\| \cap \partial \Omega$ which are not in the closure of $\spt \|V\| \cap \Omega$, by the free boundary assumption and Lemma \ref{L:constancy lemma De Masi}, $V=\Theta_V |\partial \Omega|$ in a small ball $B_r(x)$ and the conclusion follows once again from integrality of $V$ and upper semi-continuity of the reflected density.}
\begin{equation}
    \label{E:density-lower-bound}
\Theta_V(x) \geq \frac{1}{2} \quad \text{for every $x \in \spt \|V\|$.} 
\end{equation} 
By the monotonicity formula, this implies the following uniform density ratio bounds
\begin{equation}
\label{E:density-ratio-bound}
\frac{1}{C_6} \leq \frac{\|V\|(B_\rho(x))}{\rho^{n-1}} \leq C_6 \frac{\|V\|(B_{R_x}(x))}{R_x^{n-1}} 
\end{equation}
for any $0 < \rho < R_x=\min\{R_6,\dist_g(x,\partial U)\}$ when $x \in \spt \|V\| \cap \partial \Omega$; and $0<\rho < R_x=\min\{R_6,\dist_g(x,\partial \Omega \cup \partial U)\}$ when $x \in \spt \|V\| \setminus \partial \Omega$. Here, $C_6=C_6(n,i_0,k_0,\kappa_0,\Lambda)>0$ and $R_6=R_6(n,i_0,k_0,\kappa_0)>0$ are positive constants.
\end{remark}

\begin{remark}
    \label{R:weighted-monotonicity}
The same test fields give a weighted version of \eqref{E:monotonicity-diff-inequality}. Let $h\in C^1_c(U)$ be nonnegative and set
\[
I_h(\rho):=\int h\chi\left(\frac r\rho\right)\,d\|V\|,
\qquad
G_h(\rho):=\int r\chi\left(\frac r\rho\right)|S(\nabla h)|\,dV.
\]
At a boundary centre, test \eqref{E:1st-variation-fbv-refined} with $hX$ and $h\tilde X$, where $X$ and $\tilde X$ are the fields in the proof above. The additional derivative terms are bounded by $G_h(\rho)$ and $\rho G_h(\rho)$, respectively. With the same constants $A,B$, this yields
\[
\begin{split}
\rho\left(1+\frac{\kappa_0\rho}{2}\right)I_h'(\rho)
\geq{}&\left(n-1-(A+\kappa_0)\rho-\frac{\kappa_0B}{2}\rho^2\right)I_h(\rho)\\
&-\left(1+\frac{\kappa_0\rho}{2}\right)G_h(\rho).
\end{split}
\]
Thus, for the same choice of $C_5$,
\begin{equation}\label{E:weighted-monotonicity-diff-inequality}
\frac{d}{d\rho}\left(\frac{e^{C_5\rho}}{\rho^{n-1}}
\int h\chi\left(\frac r\rho\right)\,d\|V\|\right)
\geq-\frac{e^{C_5\rho}}{\rho^n}\int r\chi\left(\frac r\rho\right)|S(\nabla h)|\,dV.
\end{equation}
At an interior centre the same inequality follows using $hX$ alone. Integrating in $\rho$ and then letting the cutoffs increase to the characteristic function of $[0,1)$ gives
\begin{equation}\label{E:weighted-montonicity}
\begin{split}
&\frac{e^{C_5R_2}}{R_2^{n-1}}\int_{B_{R_2}(x)}h\,d\|V\|
-\frac{e^{C_5R_1}}{R_1^{n-1}}\int_{B_{R_1}(x)}h\,d\|V\|\\
&\qquad\geq-\int_{R_1}^{R_2}\frac{e^{C_5t}}{t^n}
\left(\int_{G(B_t(x))}r\,|S(\nabla h)|\,dV\right)dt
\end{split}
\end{equation}
for $0<R_1<R_2$ in the corresponding range of Lemma \ref{L:monotonicity limit varifold}. 
\end{remark}

The monotonicity formula in Lemma \ref{L:monotonicity limit varifold} is fundamental in the singularity analysis of free boundary varifolds as this allow us to perform blow-up arguments to reduce some of the problems to the study of tangent varifolds. The following proposition in the Euclidean setting is contained in \cite[Lemma 5.4]{Masi2020}. The proof turns out to be technically rather subtle in the Riemannian setting so we have decided to give complete details. We would like to point out that the analysis here is somewhat similar to that in \cite[Section 5]{SchoenSimon81}. Besides the fact that we are not assuming (almost) stability nor a-priori sheeting, we have to deal with the new difficulty coming from the singular part of the first variation which is not present in \cite{SchoenSimon81}.

\begin{proposition}[Existence and properties of tangent cones]
\label{P:tangent cones FB stationary varifolds}
    Let $V \in \bV_{n-1}(U)$ be a varifold with free boundary in $\partial \Omega \cap U$, $\spt \|V\| \subset \overline{\Omega} \cap U$ and $\|H^{tan}_V\|_{L^\infty(U)} \leq \Lambda$. Then for any $x\in \spt\|V\|\cap \partial\Omega$, the space $\VarTan(V,x)$ of varifold tangents at $x$ is non-empty and compact (in the weak topology). Moreover, for any $\bC \in\VarTan(V,x)$ the following holds:
    \begin{enumerate}[(i)]
        \item $\bC$ is free boundary stationary with free boundary on $T_x \partial \Omega$ and $\spt \|\bC\| \subset T_x^{+} \tM$, i.e.
        \[ \delta \bC (X)=0 \quad \text{for all $X \in \mathfrak{X}_c(T_x \tM)$ such that $X \cdot \nu(x)=0$ along $T_x \partial \Omega$};\]
        \item $\Theta_\bC(0)=\Theta_V(x)$ and $\|\bC\|$ is a cone with vertex at the origin of $T_x \tM$, i.e. 
        \[ (\eta_{0,r})_* \|\bC\|= r^{n-1}\|\bC\| \quad \text{for all $r>0$}; \] 
        \item The generalized boundary measure $\|\delta \bC\|_{sing}$ in $T_x \tM$ is scale-invariant (with dimension $n-2$), i.e. 
        \[ (\eta_{0,r})_* \|\delta \bC\|_{sing}=r^{n-2} \|\delta \bC\|_{sing} \quad \text{for all $r>0$.}\]
        Here, we identify $T_x \tM$ with $\mathbb{R}^n$ and denote $T_x^+ \tM:=\{v \in T_x \tM: v \cdot \nu(x) \leq 0\}$ the tangent half-space that points into $\Omega$, and the homothety maps $\eta_{p,r}:T_x \tM \to T_x \tM$ defined by $\eta_{p,r}(q):=r^{-1}(q-p)$ (c.f. Section \ref{SS:Euclidean}).
    \end{enumerate}
    
    \noindent If we assume, furthermore, that $V$ is rectifiable, i.e. $V \in \bRV_{n-1}(U)$, and the density of $V$ satisfies a uniform positive lower bound, i.e.
    \begin{equation}
    \label{E:uniform-lower-bound}
    \Theta_V(x) \geq 1 \text{ for $\|V\|$-a.e. $x\in U$.}
    \end{equation}
    then for any $\bC \in\VarTan(V,x)$, we have $\bC \in \bRV_{n-1}(T_x \tM)$ is a varifold cone with vertex at the origin of $T_x \tM$, i.e. 
        \[ (\eta_{0,r})_\sharp \bC= \bC \quad \text{for all $r>0$}.\]
        If $V \in \bIV_{n-1}(U)$ (in which case \eqref{E:uniform-lower-bound} is trivially satisfied), then $C \in \bIV_{n-1}(T_x \tM)$.
\end{proposition}

\begin{proof}
We provide again the key steps of the proof of \cite{Masi2020} adapted to the Riemannian setting. Consider a sequence of positive numbers $r_i \to 0$, and the sequence of varifolds\footnote{More precisely, for any fixed $R>0$, $V_i$ is an integral varifold in $B^\delta_R(0)$ for all sufficiently large $i$.} 
\[ V_i:=(\eta_{0,r_i} \circ \exp_x^{-1})_\sharp V \in \bV_{n-1}(T_x \tM).\]
Since we are working with geodesic normal coordinates of $(\tM,g)$ centred at $x$, we have the standard estimates from \eqref{E:inj-bound} and \eqref{E:sec-bound}
\begin{equation}
\label{E:metric-estimate}
|g_{ij}(q)-\delta_{ij}| \leq C(n,k_0) |q|^2
\end{equation}
for sufficiently small $|q|$ depending only on $n,k_0$ and $i_0$. This implies the following estimate on the exponential map and its Jacobian factor (recall \cite[(39.1)]{Simon1983}):
\begin{equation}
    \label{E:dexp-estimate}
 \max\{\|d(\exp^{-1}_x)|_y\|,\|d(\exp_x)|_{\exp_x^{-1}(y)}\| \} \leq 1+ C(n,k_0) \dist_g(y,x),
\end{equation}
\begin{equation}
    \label{E:Jacobian-estimate}
 |J_S (\exp^{-1}_x)(y)-1| \leq C(n,k_0) \dist_g(y,x),
\end{equation}
From \eqref{E:dexp-estimate} and \eqref{E:Jacobian-estimate}, by the monotonicity formula Lemma \ref{L:monotonicity limit varifold} and the fact that the rescaled pushforward metrics $r_i^{-2}(\exp_x^{-1})_* g$ converges in $C^2_{loc}$ to the flat Euclidean metric $\delta$ on $T_x \tM$, $V_i$ has locally uniformly bounded mass in $T_x \tM$ as varifolds in $(T_x \tM,\delta)$. Therefore, standard compactness theorem for Radon measures implies that, after passing to a subsequence, we have $V_i$ converging weakly to some $\bC \in \VarTan(V,x)$. This implies $\VarTan(V,x) \neq \emptyset$ and it follows from the same argument and the definition that $\VarTan(V,x)$ is compact in the weak topology. 

Suppose now $\bC=\lim_{i\to \infty} V_i \in \VarTan(V,x)$. We now compute its first variation as varifolds in Euclidean space $(T_x \tM,\delta)$ using \eqref{E:1st-variation} as follows: suppose $X \in \mathfrak{X}_c(T_x \tM)$ with $\spt X \subset B^\delta_R(0)$, then for all sufficiently large $i$ we have
\[ \delta V_i(X)=\int_{G(B^{\delta}_R(0))} (\Div_T^{\delta} X)(p) \; dV_i(p,T)= \frac{1}{r_i^{n-1}} \int_{G(B^{\delta}_{Rr_i}(0))} (\Div^{\delta}_T X)(r_i^{-1} q) \; d \overline{V} (q,T) \]
\[ = \frac{1}{r_i^{n-2}} \int_{G(B^{\delta}_{Rr_i}(0))} \Div^{g}_T \left( X(r_i^{-1} q) \right) \; d \overline{V} (q,T) - \frac{1}{r_i^{n-2}} \int_{G(B^{\delta}_{Rr_i}(0))} (\Div^{g}_T -\Div^{\delta}_T) \left( X(r_i^{-1} q) \right)  \; d \overline{V} (q,T)  \]
where $\overline{V}:=(\exp^{-1}_x)_\sharp V \in \bV_{n-1}(T_x \tM)$\footnote{Again, $\overline{V}$ is more precisely a varifold in some open neighbourhood of $0$.}. By the definition of pushforward of varifolds, we have $d \overline{V} (q,T)=J_S (\exp^{-1}_x)(y)\; d V (y,S)$ where $q=\exp^{-1}_x(y)$, $T=(d \exp^{-1}_x)|_y(S)$, and $J_S (\exp^{-1}_x)(y):=\sqrt{\det\left((d\exp^{-1}_x|_S)^* \circ (d\exp^{-1}_x|_S) \right)}(y)$ is the Jacobian factor of $\exp_x^{-1}$ relative to $S$. It then follows from \eqref{E:metric-estimate} that for all $i$ sufficiently large
\[ \left|\Div^{g}_T \left( X(r_i^{-1} q) \right) \right| \leq 2 \left( |q| |X(r_i^{-1}q)|+ r_i^{-1} |DX(r_i^{-1}q)| \right), \]
\[ \left|(\Div^{g}_T -\Div^{\delta}_T) \left( X(r_i^{-1} q) \right) \right| \leq C(n,k_0) \left( |q| |X(r_i^{-1}q)|+|q|^2 r_i^{-1} |DX(r_i^{-1}q)| \right). \]
Combining these estimate with \eqref{E:Jacobian-estimate}, we see that
\begin{equation}
\label{E:proof of tangent cones FB stationary varifolds}
|\delta V_i(X) - r_i^{-(n-2)} \delta V(X_i)| \leq C(n,k_0) r_i \int \left( r_i |p| |X(p)|+|p| |DX(p)| \right) \; d\|V_i\|(p) 
\end{equation} 
where $X_i:=d(\exp_x)(X \circ \eta_{0,r_i}) \in \mathfrak{X}_c(U)$ is supported inside $B_{R r_i}(x)=\exp_x(B^\delta_{Rr_i}(0))$ for all large $i$. 
Note that $V_i$ may not have locally bounded first variation in $T_x \tM$ due to the term involving $DX$. On the other hand, using \eqref{E:1st-variation-fbv-refined}, we have
\[ r_i^{-(n-2)} \delta V(X_i)= - \frac{1}{r_i^{n-1}} \int_{\partial \Omega} X_i \cdot (r_i H^{tan}_V+r_i H_{\partial \Omega}) \; d\|V\| + \frac{1}{r_i^{n-2}} \int_{\partial \Omega} X_i \cdot \nu \; d\|\delta V\|_{sing}. \]
By \eqref{E:ball-curvatures} and \eqref{E:Jacobian-estimate}, we have for all sufficiently large $i$ the following estimate for the first term on the right hand side above:
\[\left| \frac{1}{r_i^{n-1}} \int_{\partial \Omega} X_i \cdot (r_i H^{tan}_V+r_i H_{\partial \Omega}) \; d\|V\| \right| \leq C(n,\kappa_0,\Lambda) r_i \int |X(p)| \; d\|V_i\|(p). \]
For the second term, if we further assume that $X$ is tangential to $T_x \partial \Omega$, by \eqref{E:metric-estimate} and \eqref{E:ball-curvatures}, we have the estimate
\[\|X_i \cdot \nu\|_{L^\infty(B_{Rr_i}(x))} \leq C(\kappa_0) R r_i\left( R \|DX\|_{L^\infty(U_{C(\kappa_0) R^2 r_i}(T_x\partial \Omega))} + \|X\|_{L^\infty(U_{C(\kappa_0) R^2 r_i}(T_x\partial \Omega))} \right)  \]
where $U_h(T_x \partial \Omega):=\{p \in T_x \tM: \dist^\delta(p,T_x \partial \Omega) < h\}$ denotes the parallel slab of width $h$ around the hyperplane $T_x \partial \Omega$. Combining this with \eqref{E:local-volume-bound-fbv-stationary}, we can estimate the second term as follow:
\[ \left| \frac{1}{r_i^{n-2}} \int_{\partial \Omega} X_i \cdot \nu \; d\|\delta V\|_{sing} \right| \leq C(n,k_0,\kappa_0,\Theta_V(x)) R^{n-1} r_i  \left( R \|DX\|_{L^\infty} + \|X\|_{L^\infty} \right).  \]
Putting all these estimates together, \eqref{E:proof of tangent cones FB stationary varifolds} can be written as: for all sufficiently large $i$, and $X \in \mathfrak{X}_c(T_x \tM)$ with $X$ is tangential to $T_x \partial \Omega$, $\spt X \subset B^\delta_R(0)$, we have
\[ |\delta V_i(X)| \leq C(n,k_0,\kappa_0,\Lambda) r_i \int \left( (1+r_i |p|) |X(p)|+|p| |DX(p)| \right) \; d\|V_i\|(p) \]
\begin{equation}
\label{E:proof of tangent cones FB stationary varifolds-2}
\hspace{2cm} +  C(n,k_0,\kappa_0,\Theta_V(x)) R^{n-1} r_i \left( R \|DX\|_{L^\infty} + \|X\|_{L^\infty} \right)
\end{equation} 
where the $L^\infty$ norms on the right hand side are taken in the slab $U_{C(\kappa_0) R^2 r_i}(T_x\partial \Omega)$ as before.

The estimate \eqref{E:proof of tangent cones FB stationary varifolds-2} on the first variation of $V_i$ implies that 
\begin{equation}
\label{E:limit-stationary-Vi}
\lim_{i\to \infty} \delta V_i(X)=0 \quad \text{$\forall X \in \mathfrak{X}_c(T_x \tM)$ tangential to $T_x \partial \Omega$.} 
\end{equation}
By the continuity of first variation with respect to varifold convergence, we conclude that $\bC$ is free boundary stationary with free boundary on $T_x \partial \Omega$. Moreover, as $\spt \|\overline{V}\| \subset \exp_x^{-1}(\overline{\Omega})$ and $\eta_{0,r_i}(\exp^{-1}_x(\Omega))$ converges locally in Hausdorff topology to $T^+_x \tM$, it is clear that $\spt \|C\| \subset T^+_x \tM$. 

Next, we move on to the assertions in (ii) and (iii). Because of \eqref{E:Jacobian-estimate}, $\Theta_{\overline{V}}(0)$ exists and is equal to $\Theta_{V}(x)$. As $\bC \in \VarTan(\overline{V},0)$, standard arguments (c.f. \cite[Section 42]{Simon1983}) shows that 
\[ \frac{\|\bC\|(B_R(0))}{\omega_{n-1} R^{n-1}}  \equiv \Theta_V(x) \quad \text{for all $R>0$}.\]
Recall that $\bC$ is free boundary stationary in $T_x \tM$ with free boundary on the hyperplane $T_x \partial \Omega$, we therefore conclude from the monotonicity formula that $\|\bC\|$ is a cone centred at $0$. This proves (ii). Finally, (iii) is a direct consequence of (i) and (ii) because of the scaling of first variation \cite[Remark 4.12(1)]{Allard1972} and $\|\delta \bC\|=\|\delta \bC\|_{sing}$.

To establish the other assertions in the proposition, it remains to show that $\bC$ is rectifiable (resp. integral) under the assumption that $V$ is rectifiable (resp. integral) satisfying \eqref{E:uniform-lower-bound}. First of all, notice that we cannot apply the standard compactness theorem for rectifiable or integral varifolds as $V_i$ does not have locally bounded variation due to the terms involving $DX$ on the right hand side of \eqref{E:proof of tangent cones FB stationary varifolds-2}. To get around this difficulty, we first note that since $V \in \bRV_{n-1}(U)$ is free boundary stationary with uniform density lower bound \eqref{E:uniform-lower-bound}, by Remark \ref{R:monotonicity} and \eqref{E:density-ratio-bound} therein, we have the following uniform density ratio bounds at the geometric level:
\[ \frac{1}{C_6} \leq \frac{\|V\|(B_\rho(y))}{\rho^{n-1}} \leq C_6 \frac{\|V\|(B_{R_y}(y))}{R_y^{n-1}} \]
for all $y \in \spt \|V\| \cap B_{d_x/2}(x)$ provided that $0 < \rho < R_y:=\min\{R_6,d_x/2\}$ when $y \in \partial \Omega$; and $0<\rho< R_y:=\min\{R_6,d_x/2,\dist_g(y,\partial \Omega)\}$ when $y \in \Omega$. Here, $d_x:=\dist_g(x,\partial U)>0$. It is important to keep in mind that $C_6, R_6$ are positive constants depending only on the ambient geometric quantities $n,i_0,k_0,\kappa_0$ and $\Lambda$ but not on $V$ and $y$. These estimates pass to the level of geodesic normal coordinates (due to \eqref{E:dexp-estimate}, \eqref{E:Jacobian-estimate} and the fact that uniform density ratio bounds are bi-Lipschitz invariant) to yield
for any fixed $R>0$ and all sufficiently large $i$,
\begin{equation}
    \label{E:density-ratio-bound-Vi}
\frac{1}{2C_6} \leq \frac{\|V_i\|(B^\delta_\rho(p_i))}{\rho^{n-1}} \leq 2C_6 \frac{\|V_i\|(B^\delta_{R_{p_i}}(p_i))}{R_{p_i}^{n-1}}
\end{equation} 
for all $p_i \in \spt \|V_i\| \cap B^\delta_R(0)$ provided that $0 < \rho < R_{p_i}:=R$ when $p_i \in \eta_{0,r_i}(\exp_x^{-1}(\partial \Omega))$; and $0<\rho< R_{p_i}:=\min\{R, \dist_\delta (p_i,\eta_{0,r_i}(\exp_x^{-1}(\partial \Omega)))\}$ when $p_i \in \eta_{0,r_i}(\exp_x^{-1}\Omega)$. Recall the elementary observation that density is preserved under $C^1$-diffeomorphism, we have $\Theta_{V_i}$ exists everywhere in $B^\delta_R(0)$ and $\Theta_{V_i}(p_i) \geq \frac{1}{2}$ for every $p_i \in \spt\|V_i\| \cap B^\delta_R(0)$ because of \eqref{E:density-lower-bound}. 

We now argue that the uniform density ratio bounds (with slightly different constants) pass on to the tangent varifold $\bC=\lim_{i \to \infty} \|V_i\|$. Let $p \in \spt \|C\|$, it is clear from definition that there exists $p_i \in \spt \|V_i\|$ such that $p_i \to p$. Since $\spt \|C\| \subset T^+_x \tM$, either $p$ is an interior point, i.e. $p \notin T_x \partial \Omega$, or $p$ is a boundary point, i.e. $p \in T_x \partial \Omega$. For an interior point $p \notin T_x \partial \Omega$, the approximating sequence $p_i \in \spt\|V_i\|$ will also be interior points for all sufficiently large $i$. Applying the interior case of \eqref{E:density-ratio-bound-Vi}, we have for such $p$ the following density ratio lower bound 
\[ \frac{\|\bC\|(B_\rho^\delta(p))}{\rho^{n-1}} \geq \frac{1}{4^n C_6}>0\]
provided that $0< \rho < \min\{R, \frac{1}{2}\dist_\delta(p,T_x \partial \Omega)\}$. Since $\bC$ is free boundary stationary, the density $\Theta_{\bC}$ exists everywhere and letting $\rho \to 0$ in the above inequality implies that $\Theta_{\bC}(p) \geq \frac{1}{4^n \omega_{n-1} C_6}>0$ for all interior points $p \in \spt \|\bC\| \setminus T_x \partial \Omega$. Suppose now $p \in \spt \|\bC\| \cap T_x \partial \Omega$ is a boundary point. If there are infinitely many $p_i$ in the approximating sequence that lies on the boundary $\eta_{0,r_i}(\exp_x^{-1}\Omega)$, then the same argument as above gives the same density lower bound for such $p$. It remains to treat the remaining case for those $p \in \spt \|\bC\| \cap T_x \partial \Omega$ but there exists some $r>0$ with $\spt\|V_i\| \cap B^\delta_r(p)$ contained in the interior $\exp_x^{-1}(\Omega)$ for all sufficiently large $i$. Going back to the geometric setting, it means that $V$ indeed has empty free boundary and hence we are once again reduced to the interior situation which has been treated before\footnote{More precisely, in this case although $\dist_\delta(p_i,\eta_{0,r_i}(\exp_x^{-1}\Omega)) \to 0$, we can still obtain \eqref{E:density-ratio-bound-Vi} with $R_{p_i}$ replaced by a fixed $R_0>0$ independent of $i$.}. Finally, since $\bC \in \bV_{n-1}(T_x \tM)$ is free boundary stationary with free boundary in $T_x \partial \Omega$ (hence $\bC$ has locally bounded first variation in $T_x \tM$) and its density exists and has a fixed positive lower bound everywhere in $\spt \|\bC\|$, we can apply Allard's Rectifiability Theorem \cite[5.5(1)]{Allard1972} to conclude that $\bC \in \bRV_{n-1}(T_x \tM)$. Once we know that $\bC$ is rectifiable and is free boundary stationary in $T_x \tM$ with free boundary on the hyperplane $T_x \partial \Omega$, standard arguments (c.f. \cite[Section 42]{Simon1983}) show that $\bC$ is a varifold cone centred at $0$. 

We now show that $\bC$ is integral, i.e. $\bC \in \bIV_{n-1}(T_x \tM)$, if $V$ is integral. In other words, we need to show that $\Theta_{\bC}(p) \in \mathbb{N}$ for $\|\bC\|$-a.e. $p \in \spt \|\bC\|$. First we address the interior points, i.e. $p \in \spt \|\bC\| \setminus T_x \partial \Omega$. Since we already know that $\bC$ is rectifiable, we can assume without loss of generality that $\|\bC\|$ has an approximate tangent space $T_p\bC \in \Gr_x(n,n-1)$ at $p$ in the sense that
\[ \lim_{s \to 0} s^{-(n-1)}(\eta_{p,s})_\sharp \|\bC\| = \Theta_{\bC}(p) \Hcal^{n-1} \mres T_p \bC   \]
as Radon measures in $T_x \tM$. Our goal is to show that $\Theta_{\bC}(p) \in \mathbb{N}$. 

Since $\lim_{i \to \infty} V_i = \bC$ as varifolds, we also have $\lim_{i \to \infty} \|V_i\|= \|\bC\|$ and therefore we can fix some sequence $s_i \to 0$ and then, after passing to a subsequence, assume that 
\begin{equation}
    \label{E:proof-tangent-cone-1}
\lim_{i \to \infty} r_i s_i^{-1} =0 \qquad \text{and} \qquad \lim_{i \to \infty} (\eta_{p,s_i})_\sharp V_i = \Theta_{\bC}(p) |T_p \bC|
\end{equation} 
as varifolds in $T_x \tM$. Again, although each $W_i:=(\eta_{p,s_i})_\sharp V_i \in \bIV_{n-1}(T_x \tM)$, we cannot apply the compactness theorem for integral varifolds in \cite[Theorem 6.4]{Allard1972} to conclude that the limit is integral since $(\eta_{p,s_i})_\sharp V_i$ a-priori do not have locally bounded first variation. Instead, we will proceed as in the proof of \cite[Theorem 6.4]{Allard1972}. Since all the $W_i$ satisfy uniform density ratio bounds as in \eqref{E:density-ratio-bound-Vi}, we have $\spt \|W_i\|$ converges to $T_p \bC$ locally in Hausdorff topology (c.f. \cite[Remark 10]{SchoenSimon81}). Let $\pi:T_x \tM \to T_x \tM$ denote the orthogonal projection onto $T_p \bC$. Then, we have from \eqref{E:proof-tangent-cone-1} and the rectifiability of $W_i$ that after projecting $W_i$ inside the slab $U_\epsilon(T_p \bC):=\{q \in T_x \tM : |q-\pi(q)|<\epsilon\}$ of width $\epsilon >0$ (to be fixed later) we have 
\[ \lim_{i \to \infty} \pi_\sharp (W_i \mres G(U_\epsilon(T_p \bC))) = \Theta_{\bC}(p) |T_p \bC|.\] 
Since $W_i \in \bIV_{n-1}(T_x \tM)$, after possibly re-defining on a set of $\Hcal^{n-1}$-measure zero, we consider the $L^1_{loc}$ function $\psi^\epsilon_i:T_p \bC \to \mathbb{N} \cup \{\infty\}$ defined by
\[ \psi^\epsilon_i(q)=\sum_{\substack{\tilde{q} \in \spt \|W_i\| \cap U_\epsilon(T_p \bC) \\
\pi(\tilde{q})=q}} \Theta_{W_i}(\tilde{q})  \]
converges weakly in $L^1_{loc}(T_p \bC;\Hcal^{n-1})$ to the constant function $\Theta_\bC(p)$. However, weak $L^1$-convergence is not sufficient to conclude convergence almost everywhere. We need to upgrade it (after suitable truncation) to strong $L^1$-convergence. This will be achieved by the following (c.f. \cite[Lemma 42.9]{Simon1983})

\vspace{1ex}

\textit{Claim: For any $\gamma \in (0,1)$, there exists $\epsilon=\epsilon(\gamma)>0$ such that for any sufficiently large $i$, there exists $A_i \subset B^\delta_1(0) \cap T_p \bC$ with $\Hcal^{n-1}(A_i)<\gamma$ such that for every $q \in (B^\delta_1(0) \cap T_p \bC) \setminus A_i$, we have}
\begin{equation}
\label{E:sheeting-estimates}
\psi^\epsilon_i(q) \leq  (1+\gamma) \frac{\|W_i\|(B^\delta_2(q))}{\omega_{n-1} 2^n} + \gamma.
\end{equation}

\vspace{1ex}

We first see how the claim implies $\Theta_\bC(p) \in \mathbb{N}$. Choose any sequence $\gamma_i \to 0$, there exists $\epsilon_i \to 0$ and, after passing to a subsequence, there exists $A_i \subset B^\delta_1(0) \cap T_p \bC$ with $\Hcal^{n-1}(A_i) \to 0$ such that by \eqref{E:proof-tangent-cone-1}
\[ \psi^{\epsilon_i}_i(q) \leq \Theta_\bC(p) + \frac{1}{i} \quad \text{for all $q \in (B^\delta_1(0) \cap T_p \bC) \setminus A_i$.}\]
Hence, for any positive integer $N>\Theta_\bC(p)$, the integer-valued function $\max\{\psi^{\epsilon_i}_i,N\}$ converges strongly in $L^1(B^\delta_1(0) \cap T_p \bC;\Hcal^{n-1})$ to $\Theta_\bC(p)$. By taking a subsequence that converges $\Hcal^{n-1}$-a.e. in $B^\delta_1(0) \cap T_p \bC$, we conclude that $\Theta_\bC(p) \in \mathbb{N}$. 

It remains to establish the claim above. Fix $\gamma \in (0,1)$. We will follow the arguments in \cite[Lemma 42.9]{Simon1983} with the caveat that $W_i$ does not have locally bounded first variation. First of all, we observe that by \eqref{E:proof-tangent-cone-1}, for all sufficiently large $i$ we have 
\begin{equation}
\label{E:Wi-uniform-mass-bound}
\|W_i\|(B^\delta_4(0)) \leq 5^{n-1} \omega_{n-1} \Theta_{\bC}(p).
\end{equation}
Therefore, $W_i$ have uniformly bounded mass in $B^\delta_4(0)$. 
Moreover, \eqref{E:proof-tangent-cone-1} also implies the $L^1$-tilt excess for $W_i$ in $B^\delta_4(0)$ converges to zero, i.e.
\[ \lim_{i \to \infty} \int_{G(B^\delta_4(0))} \|S-T_p \bC\| \; dW_i(q,S) = 0. \]
For any $\tau>0$, we now define the subset $A^\tau_i:=\pi(\tilde{A}^\tau_i) \subset B^\delta_1(0) \cap T_p \bC$ where
\[ \tilde{A}^\tau_i:=\left\{ \tilde{q} \in \spt\|W_i\| \cap B^\delta_2(0) \; : \; \exists \rho \in (0,2) \text{ such that } \int_{G(B^\delta_{\rho}(q))} \|S-T_p \bC\| \; dW_i(\tilde{q},S) >  \tau \rho^{n-1}  \right\}. \]
Since $W_i$ satisfies uniform density ratio bounds, by Besicovitch covering lemma \cite[Lemma 4.6]{Simon1983}, we have the following estimate on the measure of $\tilde{A}^\tau_i$:
\[ \|W_i\|(\tilde{A}^\tau_i) \leq C(n,i_0,k_0,\kappa_0,\Lambda) \tau^{-1}  \int_{G(B^\delta_4(0))} \|S-T_p \bC\| \; dW_i(q,S).\]
Since $W_i \in \bIV_{n-1}(T_x \tM)$, this implies 
\[ \lim_{i \to \infty} \Hcal^{n-1}(A^\tau_i) =0 \quad \text{for any fixed $\tau>0$.} \] 

Next, by the (interior) weighted monotonicity formula \eqref{E:weighted-montonicity} applied to the original free boundary varifold $V \in \bIV_{n-1}(U)$, together with \eqref{E:dexp-estimate}, \eqref{E:Jacobian-estimate} and some scaling arguments as before, we see that for all sufficiently large $i$, and any $q \in B^\delta_1(0) \setminus \tilde{A}^\tau_i$ (note that $q$ need not be in $\spt \|W_i\|$) and any $h>0$,
\begin{equation}
\label{E:slabbed-ball-estimate-1}    
\frac{\|W_i\|(U^h_{R_1}(q))}{\omega_{n-1} R_1^{n-1}} \leq (1+C(n,i_0,k_0,\kappa_0,\Lambda) r_i s_iR_2) \frac{\|W_i\|(U^{2h}_{R_2}(q))}{\omega_{n-1} R_2^{n-1}} + \frac{4}{3} \tau \left(\frac{R_2}{h}\right)^{n-1}
\end{equation} 
for any $0<R_1 \leq R_2\leq 2$. Here, we denote $U^h_R(q):= B^\delta_R(q) \cap \{p \in T_x \tM: |N_{T_p \bC} \cdot (p-q)| < h\}$ where $N_{T_p \bC}$ is a unit vector orthogonal to $T_p \bC$. If furthermore $q \in \spt \|W_i\|$, by letting $R_1 \to 0$ and $R_2=h=2$ in \eqref{E:slabbed-ball-estimate-1}, we obtain (note that $U^{2h}_2(q) \subset B^\delta_2(q)$ for any $h>0$)
\[ \Theta_{W_i}(q) \leq (1+\gamma) \frac{\|W_i\|(B_2^\delta(q))}{\omega_{n-1} 2^{n-1}} + \gamma\]
for all sufficiently large $i$ and sufficiently small $\tau$ (both depending on $\gamma$ besides the ambient constants $n,i_0,k_0,\kappa_0$ and $\Lambda$ but independent of $q$). This proves the claim in case there is only one $\tilde{q}\in \spt \|W_i\| \cap U_\epsilon(T_p \bC)$ such that $\pi(\tilde{q})=q \in (B^\delta_1(0) \cap T_p \bC) \setminus A_i$ (note that all such $\tilde{q} \in B^\delta_2(0)$) with $A_i=A_i^\tau$. 

We will finish the proof by inducting on the number of $\tilde{q}$'s in $\spt \|W_i\| \cap U_\epsilon(T_p \bC)$ such that $\pi(\tilde{q})=q \in (B^\delta_1(0) \cap T_p \bC) \setminus A^\tau_i$. Suppose there are $N$ distinct $\tilde{q}_1,\cdots,\tilde{q}_N \in \spt\|W_i\| \cap U_\epsilon(T_p \bC)$ such that $\pi(\tilde{q}_j)=q \in (B^\delta_1(0) \cap T_p \bC) \setminus A^\tau_i$ for $j=1,\cdots,N$. We assert that for all sufficiently large $i$ (depending on $\tau$ and $n,i_0,k_0,\kappa_0,\Lambda$), with $\epsilon=\tau^{\frac{1}{2(n-1)}}/2$ for any $\tau \in (0,1)$,
\begin{equation}
\label{E:multi-layer-monotonicity}
\sum_{j=1}^N \frac{\|W_i\|(U^h_{R}(\tilde{q}_j))}{\omega_{n-1} R^{n-1}} \leq (1+ C(n,N) \tau^{\frac{1}{2(n-1)}} ) \frac{\|W_i\|(B_2^\delta(0))}{\omega_{n-1} 2^{n-1}} + C(n,N) \tau^{1/2}
\end{equation}
provided that $h \leq \frac{1}{4}\min_{j \neq k} |\tilde{q}_j - \tilde{q}_k|$ and $R< \min\{2,\tau^{-\frac{1}{2(n-1)}} C(N) h\}$.  

To establish \eqref{E:multi-layer-monotonicity}, we first note that by the interior monotonicity formula (Lemma \ref{L:monotonicity limit varifold}), the same argument as before shows for any $q \in B^\delta_1(0)$ (again not necessarily inside $\spt \|W_i\|$), $0 <R_1 \leq R_2 \leq 2$, and all sufficiently large $i$,
\begin{equation}
\label{E:multi-layer-monotonicity-1}
   \frac{\|W_i\|(B^\delta_{R_1}(q))}{\omega_{n-1} R_1^{n-1}} \leq (1+ C(n,i_0,k_0,\kappa_0,\Lambda) r_i s_iR_2) \frac{\|W_i\|(B^\delta_{R_2}(q))}{\omega_{n-1} R_2^{n-1}}  
\end{equation}
For the purpose of induction on $N$ in \eqref{E:multi-layer-monotonicity}, note that when $N=1$, it follows directly from \eqref{E:slabbed-ball-estimate-1} by setting $R_1=R_2=R$, $C(1)=2$, $\epsilon=1$ and using \eqref{E:multi-layer-monotonicity-1}. Suppose now \eqref{E:multi-layer-monotonicity} holds up to $N-1$. Consider now $N$ distinct $\tilde{q}_1,\cdots,\tilde{q}_N \in \spt\|W_i\| \cap U_\epsilon(T_p \bC)$ such that $\pi(\tilde{q}_j)=q \in (B^\delta_1(0) \cap T_p \bC) \setminus A^\tau_i$ for $j=1,\cdots,N$. Set 
\[ \hat{h}:= \min_{1 \leq j \neq k \leq N} |\tilde{q}_j - \tilde{q}_k| \leq \frac{2\epsilon}{N-1}.\]
On the other hand, it is elementary to see that, after possibly re-labelling again, we can assume that\footnote{This is based on the elementary fact: there exists constants $C(N)>0$ such that given any collection $\{a_1,\cdots,a_N\}$ of $N$ points in $\mathbb{R}$ with $d:=\min_{1 \leq j \neq k \neq N} |a_j - a_k|$ (by convention, we take $d=+\infty$ if $N=1)$, there is always a proper sub-collection $\{\hat{a}_1,\cdots,\hat{a}_K\}$ with $1 \leq K \leq N-1$ and $\hat{d} \geq \frac{3d}{2}$ such that $\hat{d} \leq C(N) d$, $4\hat{d} \leq \min_{1 \leq j \neq k \neq K} |\hat{a}_j - \hat{a}_k|$
\[ \bigcup_{j=1}^N \left(a_j-\frac{d}{2},a_j+\frac{d}{2} \right) \subset \bigcup_{i=1}^K \left(\hat{a}_i-\hat{d},\hat{a}_i+\hat{d} \right).  \]
This can be proved by a simple (but tedious) induction argument on $N$. The basic idea is to analyse different cases corresponding to the relative distance between adjacent points and throwing away (in finitely many steps bounded by $N$) those within distance of order $d$ to increase the mutual distance of the remaining $K$ elements by a factor depending on $C(N)$.} there exists some $1\leq K \leq N-1$ and $\tilde{h} \geq \frac{3\hat{h}}{2}$ such that $\tilde{h} \leq C(N) \hat{h}$, $\tilde{h} \leq \frac{1}{4} \min_{1 \leq j \neq k \leq K} |\tilde{q}_j - \tilde{q}_k|$, and provided that $\frac{\hat{h}}{R_2}$ is sufficiently small, 
\begin{equation}
    \label{E:elementary-disjoint}
\bigcup_{j=1}^N U^{\hat{h}/2}_{R_2}(\tilde{q}_j) \subset \bigcup_{j=1}^{K} U^{\tilde{h}}_{(1+  \frac{\hat{h}}{2R_2})R_2}(\tilde{q}_j)
\end{equation}
which is the key to the induction argument. Observe that by \eqref{E:slabbed-ball-estimate-1}, we have for each $j=1,\cdots,N$ and all sufficiently large $i$ that for any $0<h \leq \frac{\hat{h}}{4}$,
\[ \frac{\|W_i\|(U^{h}_R(\tilde{q}_j))}{\omega_{n-1} R^{n-1}} \leq (1+C(n,i_0,k_0,\kappa_0,\Lambda) r_i s_iR_2) \frac{\|W_i\|(U^{\hat{h}/2}_{R_2}(\tilde{q}_j))}{\omega_{n-1} R_2^{n-1}} + \frac{4}{3} \tau \left(\frac{2R_2}{\hat{h}}\right)^{n-1} \]
provided that $0 < R < R_2 \leq 2$. Summing over $j=1,\cdots,N$, using \eqref{E:elementary-disjoint} and noting that the family $\{ U^{\hat{h}/2}_{R_2}(\tilde{q}_j) \}_{j=1}^N$ are disjoint, we obtain by setting $R_2= \min\{1, \tau^{-\frac{1}{2(n-1)}} \hat{h}\}$ and $\epsilon=\tau^{\frac{1}{2(n-1)}}/2$ (recall $\tilde{q}_j \in U_\epsilon(T_p \bC)$), for all sufficiently large $i$ (depending on $\tau, n, i_0,k_0,\kappa_0,\Lambda$)
\begin{eqnarray*}
   && \sum_{j=1}^N  \frac{\|W_i\|(U^{h}_R(\tilde{q}_j))}{\omega_{n-1} R^{n-1}} \leq (1+C(n,i_0,k_0,\kappa_0,\Lambda) r_i s_i R_2) \sum_{j=1}^{K} \frac{\|W_i\|(U^{\tilde{h}}_{(1+  \frac{\hat{h}}{2R_2}) R_2} (\tilde{q}_j)) }{\omega_{n-1}  R_2^{n-1}} +  \frac{4}{3} N \tau \left(\frac{2R_2}{\hat{h}}\right)^{n-1} \\
& &\leq (1+C(n,i_0,k_0,\kappa_0,\Lambda) r_i s_i) \left(1+ \frac{\hat{h}}{2R_2} \right)^{n-1} \sum_{j=1}^{K} \frac{\|W_i\|(U^{\tilde{h}}_{(1+  \frac{\hat{h}}{2R_2}) R_2} (\tilde{q}_j)) }{\omega_{n-1} \left( (1+ \frac{\hat{h}}{2R_2}) R_2 \right)^{n-1}} +  \frac{4}{3} N \tau \left(\frac{2R_2}{\hat{h}}\right)^{n-1} \\
&&\leq (1+C(n,N) \tau^{\frac{1}{2(n-1)}})  \frac{\|W_i\|(B_2^\delta(0))}{\omega_{n-1} 2^{n-1}} +  C(n,N) \tau^{1/2}
\end{eqnarray*} 
for any $0<R<\min\{2,\tau^{-\frac{1}{2(n-1)}} C(N) h\}$. This finishes the proof of \eqref{E:multi-layer-monotonicity}. Finally, \eqref{E:multi-layer-monotonicity} clearly implies the claim by first choosing $\tau$ small (depending on $\gamma$, $n$ and $N$) and then letting $R \to 0$.
\end{proof}
    
\begin{remark}\label{R:interior-tangent-cones}
    By \eqref{E:1st-variation-fbv-stationary}, if $x \in \spt\|V\|\cap \Omega \cap U$ (i.e. $x$ is an \textit{interior point}), then the varifold $V$ is stationary in a neighbourhood of $x$ away from $\partial \Omega$, and we can proceed with the same arguments above (cf. e.g. \cite[\S 42]{Simon1983}), so in particular $\VarTan(V,x)\neq\emptyset$ and any $\bC \in\VarTan(V,x)$ is an integer-rectifiable cone in $T_x\tM$, which is stationary in $T_x\tM$ provided that $V \in \bIV_{n-1}(U)$.
\end{remark}

Let us record here a simple consequence of Proposition \ref{P:tangent cones FB stationary varifolds} and Lemma \ref{L:stationarity planar tangents}.

\begin{corollary}
Let $V \in \bV_{n-1}(U)$ be a varifold with free boundary in $\partial \Omega \cap U$, $\spt \|V\| \subset \overline{\Omega} \cap U$ and $\|H^{tan}_V\|_{L^\infty(U)} \leq \Lambda$. Suppose $\mathbf C\in\VarTan(V,x)$, where $x \in \partial \Omega$, satisfies $\spt \|C\|\subset T_x \partial \Omega$. Then, $\bC=m |T_x \partial \Omega|$ for some positive $m\in \mathbb{R}$ in $T_x \tM$. If $V \in \bIV_{n-1}(U)$, then $m \in \mathbb{N}$.
\end{corollary}

\begin{remark}
  Note that some of the statements in this subsection hold in the more general setting of free boundary varifold with $H^{tan}_V\in L^p(U;\|V\|)$ for $p > n-1$ (see, for example, \cite{Masi2020}).   
\end{remark}

\subsection{Singular set and stratification}
\label{SS:stratification}
In order to analyse the fine properties of the singularities of an integral free boundary stationary (or more generally, $\|H^{tan}_V\|_{L^\infty(U)} \leq \Lambda$) varifold $V\in\bIV_{n-1}(U)$, we need some further properties about the general structure of the singular set. In particular, we need to know that the singular set can be divided into different strata according to the maximal dimension of translation invariance for their tangent cones. We begin with a precise definition of the singular set.

\begin{definition}[Regular and singular set]\label{D:reg(V)-sing(V)}
Let $V\in\bIV_{n-1}(U)$ be an integral free boundary stationary varifold in $\Omega \cap U$ with free boundary in $\partial\Omega\cap U$ and $\spt \|V\| \subset \overline{\Omega} \cap U$. We define the \textit{regular set} $\reg(V)$ of $V$ to be
\begin{equation}
    \reg(V):=\bigg\{x \in \spt \|V\|:
        \begin{array}{l}
            \textrm{there exists \(r>0\) such that \(\spt\|V\|\cap B_{r}(x)\) is an embedded}
            \\
            \textrm{$C^\infty$ hypersurface with (possibly empty) boundary in \(\partial \Omega \cap U\)}
        \end{array}\bigg\}.
    \end{equation}
Moreover, we define its complement to be the \textit{singular set} $\sing(V)$ of $V$ by
\begin{equation}
    \sing(V):= \spt \|V\|\setminus\reg(V).
\end{equation}
The singular set can be further subdivided into its boundary part
\begin{equation}
    \sing_b(V)\coloneqq\sing(V)\cap \partial\Omega,
\end{equation}
and its interior part 
\begin{equation}
    \sing_i(V)\coloneqq \sing(V)\cap \Omega.
\end{equation}
\end{definition}

\begin{remark}
Note that by definition, $\reg(V)$ is a relatively open subset of $\spt \|V\|$. Hence, $\sing(V)$ is a relatively closed subset of $\spt \|V\|$. However, while $\sing_b(V)$ is also a relatively closed subset, $\sing_i(V)$ may not be. This simple but subtle observation is important in our analysis to follow.
\end{remark}

The most fundamental result in the general varifold setting is Allard's Regularity Theorem \cite[Theorem 8.19]{Allard1972}. Together with the free boundary version in \cite[Theorem 4.13]{Grueter1986} (see also \cite{Scharrer2022} and \cite{Nardulli}) in the intrinsic Riemannian setting), we can state it for our purpose as the following $\epsilon$-regularity result in terms of the reflected density defined in \eqref{E:reflected-density}.

\begin{proposition}[Consequence of Allard regularity - Density gap]
\label{P:Allard}
    There exists a small constant $\epsilon=\epsilon(n)\in(0,1)$ such that the following holds: given any free boundary stationary integral varifold $V\in\bIV_{n-1}(U)$ with free boundary in $\partial\Omega\cap U$ and $\spt \|V\| \subset \overline{\Omega} \cap U$, we have
    \begin{equation}
        \{x \in \spt \|V\|: 1 \leq \oTh_V(x)<1+\epsilon\}\subset \reg(V).
    \end{equation}
    In other words, any singular point must have reflected density $\geq 1+\epsilon$. As a consequence, $\reg(V)$ is an open and dense subset of $\spt \|V\|$.
\end{proposition}


\begin{remark}
In particular, if an interior point $x \in \spt \|V\| \cap \Omega$ has a tangent cone which is a multiplicity one plane, then $\oTh_V(x)=\Theta_V(x)=1$ and thus $B_r(x) \cap \spt \|V\|$ is a smoothly embedded minimal hypersurface without boundary for some $r>0$. On the other hand, if a boundary point $x \in \spt \|V\| \cap \partial \Omega$ has a tangent cone which is a multiplicity one half-plane orthogonal to $T_x \partial \Omega$, then $\oTh_V(x)=2\Theta_V(x)=1$ and thus $B_r(x) \cap \spt \|V\|$ is a smoothly embedded minimal hypersurface with non-empty free boundary lying on $\partial \Omega \cap U$ for some $r>0$.
\end{remark}

Not much is known beyond the above $\epsilon$-regularity without further assumptions (such as stability or multiplicity-one for certain tangent cones). However, there is also a very general result known as Almgren-Federer stratification, due to \cite{Federer1970} and \cite{Almgren2000}, that plays a key role in the singularity analysis of general stationary integral varifolds. The basic idea is to sub-divide the singular set into different stratum according to the order of symmetry (with respect to translations) for all possible tangent cones that can occur. We now describe this stratification procedure, paying special attention to the boundary behaviour. First, we recall some definitions and notation about varifold cones (not necessarily arising from tangent cones) inside Euclidean space.

\begin{definition}[Spine of a cone]\label{D:spine}
Let $\bC\in\bIV_{n-1}(\mathbb{R}^n)$ be an integer-rectifiable cone (in the sense that $(\eta_{0,r})_\sharp \bC=\bC$ for all $r>0$). We define the \textit{spine} $S(\bC)$ of $\bC$ to be the translation invariance set of $\bC$, namely
\begin{equation}
    S(\bC)\coloneqq \{p\in \mathbb{R}^n:(\eta_{p,1})_\sharp\bC=\bC\}.
\end{equation}
\end{definition}

It is clear from the definition that the spine $S(\bC)$ is always a linear subspace of $\mathbb{R}^n$ with dimension less than or equal to $n-1$. In the case of free boundary cones contained inside the Euclidean half-space $\mathbb{R}^n_+$ (recall the paragraph containing \eqref{E:reflection}), we have the following useful characterization of the spine.

\begin{lemma}\label{L:spine-boundary}
Let $\bC\in \bIV_{n-1}(\mathbb{R}^n)$ be an integer-rectifiable cone with free boundary in $\partial \mathbb{R}^n_+$ and $\spt \|\bC\| \subset \mathbb{R}^n_+$. Then $S(\bC)$ is a linear subspace contained in $\partial \mathbb{R}^n_+$. In particular,
\begin{equation}
    S(\bC)=\{p \in \partial \mathbb{R}^n_+ \; : \; (\eta_{p,1})_\sharp\bC=\bC\}.
\end{equation}
Consequently, we have $(\eta_{p,1})_\sharp \|\delta \bC\|_{sing}=\|\delta \bC\|_{sing}$ for all $p \in S(\bC)$.
\end{lemma}

\begin{proof}
It suffices to show that $S(\bC)$ is actually contained inside the hyperplane $\partial \mathbb{R}^n_+$. It follows from the simple observation $S(\bC) \subset \spt \|\bC\|$ (as $0 \in \spt \|\bC\|$ since $\bC$ is a cone) and the assumption that $\spt \|\bC\| \subset \mathbb{R}^n_+$. Hence, being a subspace lying inside the half-space forces $S(\bC) \subset \partial \mathbb{R}^n_+$. The last assertion follows readily from Proposition \ref{P:free-boundary-characterization}.
\end{proof}

Let us now focus on the special situation of cones $\bC \in \bIV_{n-1}(T_x\tM)$ that arise as tangent cones to a free boundary stationary integral varifold $V \in \bIV_{n-1}(U)$ with free boundary in $\partial \Omega \cap U$ and $\spt\|V\| \subset \Omega \cap U$. Fix $x \in \spt\|V\|$. By Proposition \ref{P:tangent cones FB stationary varifolds} and Remark \ref{R:interior-tangent-cones}, $\VarTan(V,x)\neq \emptyset$, and any element $\bC \in \VarTan(V,x)$ is an integral cone $\bC \in \bIV_{n-1}(T_p\tM)$ which is either stationary in $T_x \tM$ when $x \in \Omega \cap U$ is an interior point; or free boundary stationary in $T_x \tM$ with free boundary in $T_x \partial \Omega$ and $\spt \|\bC\| \subset T^+_x \tM$ when $x \in \partial\Omega\cap U$ is a boundary point. By Lemma \ref{L:reflection stationary}, in the latter case, the doubled-cone $\obC$ is a stationary cone in $T_x \tM$ and its density function $\Theta_{\obC}$ coincides exactly with the reflected density function of $\bC$ (recall \eqref{E:reflected-density}):
\begin{equation}
    \Theta_{\obC}=\oTh_\bC.
\end{equation}
Moreover, the spine of the original tangent cone $\bC$ and its doubled-cone $\obC$ are related as follows:
\begin{equation}
\label{E:spine-inclusion}
    S(\bC)=S(\obC)\cap T_x \partial\Omega.
\end{equation}
To see this, just note that the translation $\eta_{p,1}$ and the reflection $\theta_{-\nu}$ (recall \eqref{E:reflection}) commutes with each other when $p \in T_x \partial \Omega$. This implies $S(\bC) \subset S(\obC)\cap T_x \partial\Omega$. To see the other inclusion $S(\obC)\cap T_x \partial\Omega \subset S(\bC)$, take any $p \in S(\obC)\cap T_x \partial\Omega$. By the definition of spine and commutativity of $\eta_{p,1}$ and $\theta_{-\nu}$, we have
\begin{equation}
    \bC+(\theta_{-\nu})_\sharp\bC=\obC=(\eta_{p,1})_\sharp\obC=(\eta_{p,1})_\sharp\bC+(\theta_{-\nu})_\sharp((\eta_{p,1})_\sharp\bC).
\end{equation}
Note that $\spt\|\bC\|\subset T_x^+\tM$ and $\spt \|(\theta_{-\nu})_\sharp \bC\|\subset \theta_{-\nu} (T_x^+ \tM)$. If $\spt\|\bC\|$ is not contained in $T_x \partial \Omega$, then we must have $\bC=(\eta_{p,1})_\sharp \bC$, i.e. $p \in S(\bC)$. If $\spt\|\bC\| \subset T_x \partial \Omega$, then $\bC= m |T_x \partial \Omega|$ for some $m \in \mathbb{N}$ by Lemma \ref{L:stationarity planar tangents} and thus $\bC=(\eta_{p,1})_\sharp \bC$ as well. This verifies \eqref{E:spine-inclusion}.

Next, we observe another useful dichotomy property of the spine with respect to $T_x \partial \Omega$:
\begin{equation}
\label{E:dichotomy-cones}
\text{Either} \qquad  S(\overline{\bC}) \subset T_x \partial \Omega \qquad \text{ or } \qquad S(\overline{\bC}) \perp T_x \partial \Omega.
\end{equation}
If $S(\overline{\bC}) \subset T_x \partial \Omega$, then \eqref{E:spine-inclusion} implies $S(\bC)=S(\overline{\bC})$. If $S(\overline{\bC}) \perp T_x \partial \Omega$, then $\nu \in S(\overline{\bC})$ and we have an orthogonal splitting $S(\overline{\bC})=S(\bC) \times  \mathbb{R}\nu$. To see why \eqref{E:dichotomy-cones} is true, since by definition $\obC$ is invariant under reflection, i.e. $(\theta_{-\nu})_\sharp \obC=\obC$, we clearly have the spine $S(\obC)$ is also invariant under reflection.

It is well-known (see e.g. \cite[Section 4]{White1997}) that for a stationary integral cone $\bC$ in $\mathbb{R}^n$, the spine can also be characterized as the set of points of maximal density which is equal to $\Theta_\bC(0)$ (by the upper semi-continuity of density and that $\bC$ is a cone), i.e.
\begin{equation} 
\label{E:spine-max-density-interior}
S(\bC)=\{p \in \mathbb{R}^n \; : \; \Theta_\bC(p)=\Theta_\bC(0) \}.
\end{equation}
For free boundary stationary integral cones $\bC$ in $\mathbb{R}^n$ with $\spt\|\bC\| \subset \mathbb{R}^n_+$, we also have the same characterization (c.f. \cite[Lemma 5.6]{Masi2020}):
\begin{equation} 
\label{E:spine-max-density-boundary}
S(\bC)=\{p \in \partial \mathbb{R}^n_+ \; : \; \Theta_\bC(p)=\Theta_\bC(0) \}.
\end{equation}

In fact, the translation invariance along the spine gives a useful ``dimension-reduction'' argument known as \emph{cone splitting}. In the case of a stationary integral cone $\bC$ in $\mathbb{R}^n$, we know that up to a rotation the cone splits as follow: 
	\begin{equation}
    \label{E:cone-splitting}
		\bC=\bC_0\times \R^{\dim(S(\bC))}
	\end{equation}
where $\bC_0$ is a stationary integral cone of dimension $n-1-\dim(S(\bC))$ in $\R^{n-\dim(S(\bC))}$. In the case of a free boundary stationary integral cone $\bC$ contained inside a half-space $\mathbb{R}^n_+$, the same cone splitting applies to the doubled cone $\obC$. However, the dimension of the $\mathbb{R}$-factor may be different as we have $\dim S(\obC)=\dim S(\bC)+1$ in the second case of \eqref{E:dichotomy-cones}.  

We now describe the Almgren-Federer stratification adapted to the free boundary setting. The key idea is again to sub-divide the singular set according to the maximal spine dimension of all possible tangent cones at a singular point. One has to be careful at a boundary singular point as the spine dimension could be different for the cone and its doubled cone. In order to give a unified statement that works for both interior and boundary singular points, we introduce the notion of \emph{modified cone} in the definition below.

\begin{definition}
Let $V \in \bIV_{n-1}(U)$ be a free boundary stationary varifold with free boundary in $\partial \Omega \cap U$ and $\spt \|V\| \subset \overline{\Omega} \cap U$. For each $j\in\{0, 1,\dots, n-1\}$, we define the \emph{$j$-stratum} $\Scal_j$ of the singular set as follows:
\begin{equation}
    \Scal_j := \{x \in \sing(V) \; : \; \dim S(\widehat{\bC})\leq j \text{ for all } \bC \in \VarTan(V,x)\}
\end{equation}
where for each tangent cone $\bC \in \VarTan(V,x)$, the \emph{modified cone} $\widehat{\bC} \in \bIV_{n-1}(T_x \tM)$ is defined as
    \begin{equation}
        \widehat{\bC}:= \left\{\begin{array}{cl}
        \bC & \text{if $x \in \Omega$,} \\ 
        \bC+(\theta_{-\nu(x)})_\sharp \bC & \text{if $x \in \partial \Omega$.}
        \end{array}\right.
    \end{equation}
In other words, $\widehat{\bC}$ is either the cone $\bC$ itself or the doubled cone $\obC$ depending on whether the cone is based at an interior or a boundary point $x \in \spt \|V\|$. Note that in either case $\widehat{\bC}$ is stationary (without free boundary!) in $T_p\tM$ and $\Theta_{\widehat{\bC}}(0)=\oTh_V(x)$. From the definition we clearly have
\begin{equation}
    \Scal_0\subset\Scal_1\subset\dots\subset\Scal_{n-1}=\sing(V).
\end{equation}
For each stratum $\Scal_j$, it is  useful to define the following subset of $\Scal_j$:
\begin{equation}
    \tScal_j := \Scal_j\setminus\Scal_{j-1}=\left\{x \in \sing(V) \; : \; \max_{\bC \in \VarTan(V,x)} \dim(S(\widehat{\bC})) = j \right\}.
\end{equation}
\end{definition}

\begin{example}[Classical boundary tangent cones]
\label{Ex:boundary-classical-cones}
Consider the following union of two half-planes 
\[ V_1=|\{(x,y) \in \mathbb{R}^2 \; : \; |x|=|y| \text{ and } y \geq 0\}| \times \mathbb{R}^{n-2} \in \bIV_{n-1}(\mathbb{R}^n)\]
which is a free boundary stationary varifold contained in the half-space $\mathbb{R}^n_+$, in this case we have 
\[ \sing (V_1)=\tScal_{n-2}=\{0\} \times \mathbb{R}^{n-2}.\] 
Perhaps a more interesting example is 
\[ V_2= \left( 2 | \{(x,y):y=\sqrt{3}x, \; y \geq 0  \} | +  |\{(x,y):x \leq 0, \; y=0\}| \right) \times \mathbb{R}^{n-2} \in \bIV_{n-1}(\mathbb{R}^n)\]
which is also a free boundary stationary varifold in the half-space $\mathbb{R}^n_+$, in this case we still have 
\[ \sing (V_2)=\tScal_{n-2}=\{0\} \times \mathbb{R}^{n-2}.\] 
\end{example}

\begin{remark}\label{R:cones-large-spine}
Before we state the main result (Proposition \ref{P:stratification}) concerning estimates on the Hausdorff dimension of each stratum of the singular set, we first take a closer look at the $j$-stratum for $j \geq n-3$, which will be most relevant for use in our paper. In our discussions below, let $x \in \sing (V)$ and $\bC \in \VarTan(V,x)$. Recall that the modified cone $\widehat{\bC}$ is always a stationary integral cone in $T_p\tM$ by Proposition \ref{P:tangent cones FB stationary varifolds} and Lemma \ref{L:reflection stationary}.
\begin{enumerate}[(i)]
    \item If $\dim(S(\widehat{\bC}))=n-1$ (i.e. $x \in \tScal_{n-1}$), then the classical constancy theorem says that $\widehat{\bC}= m |P|$ for some hyperplane $P \in \Gr_x(n,n-1)$ and $m  \in \mathbb{N}$. In fact, $m \geq 2$ by Allard's regularity (Proposition \ref{P:Allard}). If $x \in \sing_i(V)$ is an interior singular point, then clearly $\bC=m|P|$. If $x \in \sing_b(V)$ is a boundary singular point, then we have the following dichotomy:
    \begin{itemize}
        \item either $P=T_x \partial\Omega$ and $\bC=m|T_x \partial \Omega|$ for some $m \in  \mathbb{N}$ (possibly $m=1$);
        \item or $P \perp T_x \partial \Omega$ and $\bC=m|P_+|$ for some $m \in  \mathbb{N}$ with $m \geq 2$ and $P_+:=P \cap T^+_x \tM$.
    \end{itemize}
    
    \item If $\dim(S(\widehat{\bC}))=n-2$, then by \eqref{E:cone-splitting} we have up to rotations the splitting 
    \begin{equation}
        \label{E:interior-classical-cone}
    \widehat{\bC}=\widehat{\bC}_0 \times S(\widehat{\bC})\cong \widehat{\bC}_0 \times \R^{n-2},
    \end{equation} 
    where $\widehat{\bC}_0$ is a $1$-dimensional stationary integral cone in $\R^2$. By \cite{Allard-Almgren1976}, $\widehat{\bC}_0$ consists of a finite union of rays emanating from the origin, each ray is counted with integer multiplicities satisfying the stationarity condition at the origin. Cones of the form \eqref{E:interior-classical-cone} are often referred to in the literature (c.f. \cite{Wick14}) as \textit{(interior) classical tangent cones}. If $x \in \sing_i(V)$ is an interior singular point, then $\widehat{\bC}_0$ must consist of at least three rays. If $x \in \sing_b(V)$ is a boundary singular point, then we have the following dichotomy by \eqref{E:dichotomy-cones}:
    \begin{itemize}
        \item either $S(\widehat{\bC})=S(\bC)\subset T_x \partial \Omega$, in which case 
        \begin{equation}
        \label{E:boundary-classical-cone}
        \bC=\bC_0 \times S(\bC)
        \end{equation}
        where $\bC_0$ is a $1$-dimensional free boundary stationary integral cone contained in the plane $\mathbb{R}^2 \cong \Span \{\nu(x),e\}$ with $\spt \|\bC_0\| \subset \mathbb{R}^2_+$. Here, $e \in T_x \partial \Omega$ is a unit vector perpendicular to $S(\bC)$ and $\mathbb{R}^2_+:=\{p \in \mathbb{R}^2 : p \cdot \nu(x) \leq 0\}$. We shall refer to cones of the form \eqref{E:boundary-classical-cone} as \emph{boundary classical tangent cone}.
        \item or $S(\widehat{\bC}) \perp T_x \partial \Omega$, in which case $S(\widehat{\bC})=S(\bC)\times \R \nu(x)$ so $\dim(S(\bC))=n-3$. In fact, 
        \begin{equation}
            \label{E:half-classical-cone}
        \bC=\bC_0 \times S(\bC) \times \mathbb{R}_+ \nu(x)
        \end{equation} 
        where $\mathbb{R}_+ \nu(x)$ denotes the ray $\{-r \nu(x):r \geq 0\}$ along $-\nu(x)$ that points into $\Omega$, and $\bC_0$ is a $1$-dimensional stationary (without free boundary!) integral cone contained in a $2$-dimensional subspace of $T_x \partial \Omega$. We shall refer to cones of the form \eqref{E:half-classical-cone} \emph{half classical tangent cone}.
    \end{itemize}
    
\end{enumerate}
\end{remark}

The usefulness of such stratification is due to the following fundamental result on the size of the $j$-stratum of the singular set, which was first proved by Almgren \cite{Almgren2000} in the setting of area-minimizing currents.

\begin{proposition}\label{P:stratification}
    Let $V\in \bIV_{n-1}(U)$ be a free boundary stationary integral varifold with free boundary in $\partial \Omega \cap U$ and $\spt \|V\| \subset \overline{\Omega} \cap U$. Then, for all $j\in\{0,1,\dots, n-1\}$
    \begin{equation}
        \dim_\Hcal\Scal_j\leq j,
    \end{equation}
    where $\dim_\Hcal$ denotes the Hausdorff dimension.
\end{proposition}

\begin{proof}
The proof is essentially the same as in the interior case (see e.g. \cite[Theorem 4]{White1997}). We shall follow the same structure as in \cite[\S 3.4]{Simon1996}, where an analogous result is proved in the context of energy-minimizing harmonic maps. 

The interior stratification theorem gives $\dim_\Hcal(\Scal_j\cap\Omega\cap U)\leq j$: indeed, $V$ is stationary in each ball compactly contained in $\Omega\cap U$, and we can cover $\Omega\cap U$ by countably many such balls. It therefore remains to prove $\dim_\Hcal(\Scal_j\cap\partial\Omega\cap U)\leq j$.

Let us fix $j\in\{0, 1, \dots, n-1\}$ and $\epsilon>0$. We shall say that a subset $A\subset \R^n$ has the $\epsilon$\textit{-approximation property} if there exists some $\rho_0>0$ such that for all $y\in A$ and $\rho \in (0, \rho_0]$, there is a $j$-dimensional subspace $W_{y, \rho}$ of $\R^n$ such that 
\begin{equation}
    \eta_{y, \rho}(A \cap B_\rho(y))\subset U_\epsilon(W_{y,\rho})
\end{equation}
where $\eta_{y,\rho}$ denotes the homothety map in \eqref{E:homothety} and $U_\epsilon(W_{y,\rho})$ denotes the $\epsilon$-open neighbourhood of $W_{y,\rho}$ in $\R^n$. By \cite[\S 3.4, Lemma 3]{Simon1996}, there exists a positive function $\beta:(0,+\infty) \to (0,+\infty)$ such that $\lim_{\epsilon \to 0}\beta(\epsilon)=0$ and $\Hcal^{j+\beta(\epsilon)}(A)=0$ whenever $A \subset \R^n$ satisfies the $\epsilon$-approximation property.

Thus, it suffices to cover $\Scal_j\cap\partial\Omega\cap U$, in local coordinates, by countably many sets with the $C\epsilon$-approximation property, where $C$ is independent of $\epsilon$. We use ambient geodesic normal coordinates throughout; the transfer to fixed coordinate charts is justified in Step 2. The main arguments can be divided into two steps.

\vspace{1ex}

\noindent\textit{Step 1: Show that the ``good density'' points accumulate around a $j$-dimensional subspace.}

\vspace{1ex}

Take any $x \in \Scal_j\cap\partial\Omega\cap U$. We claim that there exists $\overline{\rho}=\overline{\rho}(x,\epsilon,V)\in(0,i_0/2)$, with $B_{\overline{\rho}}(x)\subset U$, such that for all $\rho \in(0, \overline{\rho}]$ there exists a $j$-dimensional subspace $W_{x,\rho} \subset T_x\partial\Omega$ such that
\begin{equation}\label{E:stratification-proof-key-claim}
    \eta_{0,\rho}\left(\exp_x^{-1}(\{y \in B_\rho(x)\cap\partial\Omega:\Theta_V(y) \geq \Theta_V(x)-\overline{\rho}\}) \right) \subset U_\epsilon(W_{x,\rho}).
\end{equation}
Here $U_\epsilon(W_{x,\rho})$ is the neighbourhood in $T_x\tM$, as above; the image of $\partial\Omega$ under $\exp_x^{-1}$ need not lie in $T_x\partial\Omega$.

We will establish the claim by contradiction. Suppose the claim is false. Then there exist $x \in \Scal_j\cap\partial\Omega\cap U$ and sequences $\overline{\rho}_i \to 0$ and $\rho_i \in (0,\overline{\rho}_i]$ such that for each $i$
\begin{equation}\label{E:stratification-proof-failed-inclusion}
    \eta_{0,\rho_i}\left(\exp_x^{-1}(\{y \in B_{\rho_i}(x)\cap\partial\Omega:\Theta_V(y) \geq \Theta_V(x)-\overline{\rho}_i\}) \right) \not \subset U_\epsilon(W)
\end{equation}
for \emph{any} $j$-dimensional subspace $W \subset T_x\partial\Omega$. Up to passing to a subsequence, we can assume that $V_i := (\eta_{0,\rho_i})_\sharp (\exp_x^{-1})_\sharp V \to \bC \in \VarTan(V,x)$ as varifolds. Since $x \in \Scal_j$, we know that $\dim S(\widehat{\bC}) \leq j$, and hence $\dim S(\bC)\leq j$ by \eqref{E:spine-inclusion}. Choose a $j$-dimensional subspace $W \subset T_x\partial\Omega$ containing $S(\bC)$. By \eqref{E:stratification-proof-failed-inclusion}, there exist $y_i \in B_{\rho_i}(x)\cap\partial\Omega$ such that, with $q_i:=\eta_{0,\rho_i}(\exp_x^{-1}(y_i)) \in B^\delta_1(0)$, we have
\begin{equation}\label{E:stratification-proof-failed-inclusion-2}
\Theta_V(y_i) \geq \Theta_V(x) - \overline{\rho}_i= \Theta_{\bC}(0) - \overline{\rho}_i,
\end{equation}
but $q_i \notin U_\epsilon(W)$. After passing to a subsequence, we can also assume
\[
q_i\to q\in\bigl(\overline{B^\delta_1(0)}\cap T_x\partial\Omega\bigr)\setminus U_\epsilon(W),
\]
since the rescaled boundaries $\eta_{0,\rho_i}(\exp_x^{-1}(\partial\Omega))$ converge locally smoothly to $T_x\partial\Omega$. By \eqref{E:stratification-proof-failed-inclusion-2}, $\liminf_{i\to\infty}\Theta_V(y_i)\geq\Theta_\bC(0)$. We claim that
\begin{equation}
    \label{E:upper-semi-continuity-density-1}
\Theta_{\bC}(q) \geq \limsup_{i\to\infty}\Theta_V(y_i).
\end{equation}
Once the claim is established, $\Theta_\bC(q)\geq\Theta_\bC(0)$. Conversely, exact boundary monotonicity for $\bC$ at $q\in T_x\partial\Omega$, together with conicality, gives $\Theta_\bC(q)\leq\Theta_\bC(0)$ by letting the radius tend to infinity. Thus \eqref{E:spine-max-density-boundary} implies $q\in S(\bC)\subset W$, a contradiction.

\vspace{1ex}

It remains to verify \eqref{E:upper-semi-continuity-density-1}. Fix $s>0$ and $a>0$. Since $q_i\to q$, the differential estimate \eqref{E:dexp-estimate} gives, for all sufficiently large $i$,
\[
\eta_{0,\rho_i}\bigl(\exp_x^{-1}(B_{\rho_i s}(y_i))\bigr)
\subset\overline{B^\delta_{(1+a)s}(q)}.
\]
The balls $B_{\rho_i s}(y_i)$ are compactly contained in $U$ for large $i$. Applying Lemma \ref{L:monotonicity limit varifold} at the boundary point $y_i$ and using \eqref{E:Jacobian-estimate}, we obtain
\[
\Theta_V(y_i)
\leq e^{C_5\rho_i s}\frac{\|V\|(B_{\rho_i s}(y_i))}{\omega_{n-1}(\rho_i s)^{n-1}}
\leq e^{C_5\rho_i s}(1+o_i(1))
\frac{\|V_i\|(\overline{B^\delta_{(1+a)s}(q)})}{\omega_{n-1}s^{n-1}},
\]
where $o_i(1)\to0$ for fixed $s,a$. Weak convergence on the compact closed ball therefore yields
\[
\limsup_{i\to\infty}\Theta_V(y_i)
\leq\frac{\|\bC\|(\overline{B^\delta_{(1+a)s}(q)})}{\omega_{n-1}s^{n-1}}.
\]
Letting first $s\downarrow0$ and then $a\downarrow0$ proves \eqref{E:upper-semi-continuity-density-1}, since open and closed balls give the same density limit.

\vspace{1ex}

\noindent \textit{Step 2: Decompose $\Scal_j\cap\partial\Omega\cap U$ into a countable union of sets satisfying the approximation property.}

\vspace{1ex}

This step is exactly the same as in the interior case. We provide the proof here for readers' convenience. First, for each $k \in \mathbb{N}$ we let
\begin{equation}
    \Scal_{j}^{k} := \left\{ x \in \Scal_j\cap\partial\Omega\cap U : \text{ \eqref{E:stratification-proof-key-claim} holds with $\overline{\rho}=\frac{1}{k}$} \right\},
\end{equation}
and for each $k,\ell \in \mathbb{N}$ we define
\begin{equation}
    \Scal_{j}^{k,\ell} := \left\{ x \in \Scal_{j}^{k}: \Theta_V(x) \in \left[ \frac{\ell-1}{k}, \frac{\ell}{k} \right) \right\}.
\end{equation}
Clearly, $\Scal_j\cap\partial\Omega\cap U=\bigcup_{k,\ell =1}^\infty \Scal_{j}^{k,\ell}$. For any $x \in \Scal_j^{k,\ell} \subset \Scal_j^k$, the definitions give
\[ \Scal_j^{k,\ell} \subset \left\{ y \in \Scal_j\cap\partial\Omega\cap U : \Theta_V(y) \geq \Theta_V(x)-\frac{1}{k} \right\},\]
and for every $\rho \in (0,\frac{1}{k}]$, there exists a $j$-dimensional subspace $W_{x,\rho} \subset T_x\partial\Omega$ such that
\[ \eta_{0,\rho} \left(\exp_x^{-1}\left( \left\{y \in B_\rho(x)\cap\partial\Omega:\Theta_V(y) \geq \Theta_V(x)-\frac{1}{k} \right\} \right) \right) \subset U_\epsilon(W_{x,\rho}). \]
Therefore,
\[ \eta_{0,\rho} ( \exp_x^{-1} (\Scal_j^{k,\ell} \cap B_\rho(x)) ) \subset U_\epsilon (W_{x,\rho}).\]

To pass to a fixed normal coordinate chart, let $f=\exp_p^{-1}$ and restrict to a relatively compact smaller chart on which $f$ and $f^{-1}$ have Lipschitz constants at most $L\geq1$. Smoothness gives
\[
f(y)-f(x)=df_x\bigl(\exp_x^{-1}(y)\bigr)+o(r)
\]
uniformly for $x,y$ in this smaller chart with $|f(y)-f(x)|<r$. Apply the preceding inclusion with $\rho=Lr\leq1/k$, and use $\|df_x\|\leq L$. For all sufficiently small $r$, depending also on $\epsilon$, the distance from $f(y)-f(x)$ to the $j$-dimensional subspace $df_x(W_{x,Lr})$ is strictly less than $2L^2\epsilon r$ whenever $x,y\in\Scal_j^{k,\ell}$ and $|f(y)-f(x)|<r$. Thus the coordinate image of each such local piece has the $2L^2\epsilon$-approximation property. The charts can be chosen with a common bound for $L$ by \eqref{E:metric-estimate} and \eqref{E:dexp-estimate}. A countable chart cover and the approximation lemma give
\[
\Hcal^{j+\beta(2L^2\epsilon)}(\Scal_j\cap\partial\Omega\cap U)=0.
\]
Letting $\epsilon\downarrow0$ completes the proof.

\end{proof}

\subsection{Curvature varifolds with boundary} 
\label{SS:curvature-varifolds}

In this subsection, we recall the notion of curvature varifolds with boundary first introduced by Hutchinson \cite{HutchinsonIUMJ} in the boundaryless setting and later extended by Mantegazza \cite{Mantegazza1996} to encompass the notion of boundary. Both \cite{HutchinsonIUMJ} and \cite{Mantegazza1996} work in the Euclidean setting (or submanifolds inside Euclidean space). We shall adapt it to the intrinsic Riemannian setting avoiding the use of isometric embedding. The main idea is to define a generalized second fundamental form and impose the validity of integration by parts of such curvature varifolds at the Grassmannian level. It turns out that this sub-class of integral varifolds enjoys some nice compactness and a-priori regularity properties and is a natural object that arises in stable limit-interface of Allen-Cahn equation (c.f. \cite{Tonegawa2005a}).

\begin{remark}
    We point out that one needs to work within the class of integral varifolds to yield some useful geometric properties of curvature varifolds (with or without boundary). This is the setting considered in \cite{HutchinsonIUMJ} and \cite{Mantegazza1996}. One fundamental reason is that any integral varifolds with locally bounded first variation must have its generalized mean curvature vector perpendicular almost everywhere to the approximate tangent space (which exists a.e. by rectifiability). This result, known as Brakke's Perpendicularity Theorem, was first proved in the Euclidean setting by Brakke in his monumental monograph \cite[Chapter 5]{Brakke}. We shall extend this fundamental result to the intrinsic Riemannian setting. Note that this result is just about the properties of the generalized mean curvature vector but provides no information about the singular part $\|\delta V\|_{sing}$ of the first variation.
\end{remark}

\begin{theorem}[Brakke's Perpendicularity Theorem]
\label{T:Brakke-perpendicularity}
Let $V \in \bIV_{n-1}(U)$ be an integral varifold with locally bounded first variation. Then, we have
\[ V\Big( \{(x,S) \in G(U):S(H_V(x)) \neq 0 \} \Big) =0\]
where $H_V$ is the generalized mean curvature vector as defined in \eqref{E:Radon-Nikodym}. Moreover, the following ``flattening properties'' hold: for $\|V\|$-a.e. $x \in U$, 
\[ \lim_{r \to 0} \frac{1}{r^n} \int_{B_r(x)} \|P_{y,x}(T_y V)- T_x V\|^2 \; d\|V\|(y)=0  \]
where $P_{y,x}:\Gr_y(n,n-1) \to \Gr_x(n,n-1)$ is the map obtained by parallel transport from $y$ to $x$ along the unique minimizing geodesic between $y$ and $x$.
\end{theorem}

Since we are only interested in codimension one varifolds, we can give a simpler description of curvature varifolds compared to the ones in \cite{HutchinsonIUMJ} and \cite{Mantegazza1996}. First, recall that $\pi:G(U):=\cup_{x \in U} \Gr_x(n,n-1) \to U$ is a smooth $\mathbb{RP}^{n-1}$-bundle over $U \subset \tM$ by the identification $S \cong \pm N_S$. At each $(x,S) \in G(U)$, there is a natural splitting (which is orthogonal with respect to a Sasaki-type Riemannian metric $\tilde{g}$ on $G(U)$ naturally induced from $(\tM,g)$ c.f. \cite[Section 2]{Koike-Yamamoto2018}) of the tangent space 
\begin{equation}
\label{E:grassmann-splitting}
T_{(x,S)} G(U) = \mathcal{H}_{(x,S)} \oplus \ker d\pi_{(x,S)}
\end{equation}
for some ``horizontal'' subspace $\mathcal{H}_{(x,S)}$ such that $d\pi_{(x,S)}:\mathcal{H}_{(x,S)} \to T_x\tM$ is a linear isomorphism. On the other hand, the ``vertical'' subspace $\ker d\pi_{(x,S)}$ is tangential to the fiber $G_x(n,n-1)$ which can be identified with $S$ itself (since $T_{\pm N_S} \mathbb{S}^{n-1} \cong S$). We shall use $\widetilde{\nabla}$ to denote the Levi-Civita connection on $G(U)$ with respect to the Sasaki-type metric above. In what follows, we call any map $X:G(U) \to T \tM$ such that $X(x,S)\in T_x \tM$ for all $(x,S) \in G(U)$ a \emph{vector field on $G(U)$ with values in $T\tM$}.

\begin{definition}[Curvature varifolds with boundary]
\label{D:curvature-varifolds}
A varifold $V \in \bV_{n-1}(U)$ is said to be a \emph{curvature varifold with boundary} if there exist on the Grassmannian $G(U)$
\begin{itemize}
\item a tensor-valued $\mathbf{B}_V \in L^1_{loc}(G(U);V)$ with values in $T^*\tM \otimes T^* \tM \otimes T \tM$, i.e. for any continuous vector fields $X,Y$ on $G(U)$ with values in $T\tM$, 
\[ \mathbf{B}_V(X,Y) \text{ is a locally $V$-integrable vector field with values in $T\tM$}\]
satisfying  
\begin{equation}
\label{E:formal-properties-B}
V\left( \left\{ (x,S) \in G(U):S(\mathbf{B}_V (S^\perp(X), S^\perp (Y))  \neq 0 \text{ for some $X,Y$}  \right\}\right) =0;
\end{equation}
\item and a vector-valued Radon measure $\partial V$ on $G(U)$ with values in $T\tM$, i.e. 
\[ \partial V= \textbf{n}_V \|\partial V\|\]
where $\|\partial V\|$ is the total variation Radon measure of $\partial V$ and $\textbf{n}_V$ is a $\|\partial V\|$-measurable vector field on $G(U)$ with values in $T\tM$ and $|\textbf{n}_V|=1$ $\|\partial V\|$-a.e., satisfying
\begin{equation}
\label{E:formal-properties-partial-V}
\|\partial V\| \left( \left\{ (x,S) \in G(U):S(\textbf{n}_V)  \neq \textbf{n}_V   \right\}\right) =0;
\end{equation}
\end{itemize}
such that for any $\varphi \in C^1_c(G(U))$ and any $C^1$ vector field $X:U \to T\tM$, we have
\[ \int_{G(U)} \varphi(x,S) \, X(x) \cdot \textbf{n}_V(x,S) d \|\partial V\|= \int_{G(U)} S(\nabla^{hor} \varphi(x,S)) \cdot X(x) +\varphi(x,S) \Div_S X(x) \; dV(x,S)  \]
\begin{equation}
    \label{E:curvature-varifolds}
   +\int_{G(U)} \mathbf{B}_V(\nabla^{ver}\varphi(x,S),X(x)) \cdot N_S - \varphi(x,S) \, \mathbf{H}_V(x,S) \cdot X(x) \; dV(x,S)
\end{equation}
where $\mathbf{H}_V:=\Tr \mathbf{B}_V = \sum_{i=1}^{n} \mathbf{B}_V(e_i,e_i) \in L^1_{loc}(G(U);V)$ with values in $T\tM$. Here, $\{e_i\}_{i=1}^n$ is any orthonormal basis for $T_x \tM$ and $\nabla \varphi=\nabla^{hor} \varphi + \nabla^{ver} \varphi$ is the $\tilde{g}$-gradient of $\varphi$ written with respect to the orthogonal splitting \eqref{E:grassmann-splitting}. We call $\mathbf{B}_V$ the \emph{generalized second fundamental form}, $\textbf{n}_V$ the \emph{varifold outward co-normal} and $\|\partial V\|$ the \emph{varifold boundary measure} of $V$. Here, we adopt the convention to identify $T\tM$ and $T^*\tM$ using the metric $g$ whenever necessary. We shall use $\bCV^\partial_{n-1}(U)$ to denote the space of all $(n-1)$-dimensional curvature varifolds with boundary in $U$.
\end{definition}

\begin{remark}
Note that $\nabla^{ver} \varphi (x,S) \in \ker d\pi_{(x,S)}$, when regarded as a vector in $S$ is only well-defined up to the anti-podal isomorphism $p \mapsto -p$ on $S$. This reflects the choice of a unit normal $\pm N_S$ to the \emph{un-oriented} subspace $S$. Nonetheless, the term $\mathbf{B}_V(\nabla^{ver}\varphi(x,S),X(x)) \cdot N_S$ in \eqref{E:curvature-varifolds} is well-defined independent of the choice of $N_S$.
\end{remark}

\begin{remark}
    Condition \eqref{E:formal-properties-B} says that $\mathbf{B}_V$ at $V$-a.e. $(x,S) \in G(U)$ is essentially a bi-linear map $S \times S \to S^\perp$ (however, we do not assume that $\mathbf{B}_V(X,Y)= \mathbf{B}_V(Y,X)$ in general). Similarly, condition \eqref{E:formal-properties-partial-V} says that $\textbf{n}_V$ at $\|\partial V\|$-a.e. $(x,S)$ is essentially tangential to $S$. These conditions are necessary in order for $\mathbf{B}_V$ and $\textbf{n}_V$ to be uniquely defined by \eqref{E:curvature-varifolds}. Note that no such conditions are required in \cite{HutchinsonIUMJ} and \cite{Mantegazza1996} as they regard $\Gr_x(n,n-1)$ as a subset of $\mathbb{R}^{n^2}$, the space of all $n \times n$ real matrices, and they allow the test functions $\varphi \in C^1_c(U \times \mathbb{R}^{n^2})$ instead of the smaller subspace $C^1_c(G(U))$ considered here in \eqref{E:curvature-varifolds}. As a result, \eqref{E:formal-properties-B} and \eqref{E:formal-properties-partial-V} together imply all the formal and tangential properties in \cite[Proposition 3.6 and 3.7]{Mantegazza1996}.
\end{remark}

\begin{remark}
\label{R:relation-Mantegazza}
    When $V \in \bIV_{n-1}(U)$, our definition generalizes the notion of curvature varifolds with boundary in \cite{Mantegazza1996} to the case of general varifolds in the Riemannian setting. If $U \subset \mathbb{R}^n$, then the Grassmannian splits globally as a product $G(U) \subset U \times \mathfrak{gl}(n,\R)$ if we regard $S \in \mathfrak{gl}(n,\R) \cong \mathbb{R}^{n^2}$ as the orthogonal projection matrix. Taking $X=\partial_i$ to be parallel coordinate vector fields in $\mathbb{R}^n$, it is straightforward to check that \eqref{E:curvature-varifolds} reduces to \cite[Definition 3.1]{Mantegazza1996}. In general, we can always locally trivialize the Grassmannian bundle $G(B_r(x))=B_r(x) \times \Gr_x(n,n-1)$, say within a small geodesic ball $B_r(x) \subset U$ and use parallel transport along minimizing geodesics from $x$ to identify $\Gr_y(n,n-1)$ with $\Gr_x(n,n-1)$ for all $y \in B_r(x)$, to reduce the situation to the Euclidean setting (modulo some extra terms coming from the ambient geometry). Notice that under such local trivialization, we can express any $\varphi \in C^1_c(G(B_r(x)))$ as $\varphi(y,S)$ where $y \in B_r(x)$ and $S \in \Gr_x(n,n-1)$ and the orthogonal splitting \ref{E:grassmann-splitting} at the base point $x \in \tM$ becomes
    \[ \nabla \varphi (x,S)= \nabla^{\tM} \varphi (x,S) + \nabla^{\Gr_x(n,n-1)} \varphi (x,S).\]
    Note that the extra term $\Div_S X$ appearing in \eqref{E:curvature-varifolds} (that was not present in \cite{Mantegazza1996}), when $X=\partial_i$ is a (non-parallel!) coordinate vector field in geodesic normal coordinates, involves up to first order derivatives of $g_{ij}$ and can be estimated using \eqref{E:metric-estimate}. 
\end{remark}

\begin{remark}
\label{R:curvature-varifold-local}
    Note that the notion of curvature varifold with boundary is essentially a local property in the sense that one only needs to check \eqref{E:curvature-varifolds} for $\varphi \in C^1_c(G(U_\alpha))$ where $\{U_\alpha\}_{\alpha}$ is any locally finite open cover of $U$. This follows from the linearity of $\varphi$ in \eqref{E:curvature-varifolds} and a standard partition of unity argument. We will mainly apply this observation to any locally finite open cover of $U$ by small enough geodesic balls.
\end{remark}

We first note several elementary consequences by choosing suitable test functions $\varphi$ in \eqref{E:curvature-varifolds}, in case $V$ is also a rectifiable (resp. integral) varifold.

\begin{lemma}
\label{L:curvaure-varifolds-definition}
Let $V \in \bCV^\partial_{n-1}(U) \cap \bRV_{n-1}(U)$.
\begin{enumerate}
\item[(i)] $\mathbf{B}_V$ and $\partial V$ are uniquely determined by \eqref{E:curvature-varifolds} and $\spt \|\partial V\| \subset \spt V$.
\item[(ii)] If $V \in \bIV_{n-1}(U)$, then $\pi_* \|\partial V\|$ is singular with respect to $\|V\|$, i.e. 
\[ \pi_* \|\partial V\| \perp \|V\|.\]
Moreover, $V$ has locally bounded first variation in $U$ and the generalized mean curvature $H_V$ and the generalized vector-valued boundary measure $\eta_V \, \|\delta V\|_{sing}$ defined in \eqref{E:Radon-Nikodym} are given by
\begin{equation}
\label{E:relation-to-first-variation}    
H_V(x)=\mathbf{H}_V(x,T_xV) \text{ and } \eta_V \|\delta V\|_{sing}=\pi_*(\partial V)
\end{equation} 
where $T_xV$ is the approximate tangent space of $V$ which exists for $\|V\|$-a.e. $x \in U$.

\end{enumerate}
\end{lemma}

\begin{proof}
For (i), the proof in \cite[Proposition 3.4]{Mantegazza1996} carries over once we fix a local trivialization as in Remark \ref{R:relation-Mantegazza} and choose $\varphi(y,S)=\phi(y)\psi(S)$ and $X=\partial_i$. It is clear from \eqref{E:curvature-varifolds} that $\spt \|\partial V\| \subset \spt V$. Note that the proof in \cite{Mantegazza1996} only requires rectifiability of $V$. For (ii), if we take $\varphi=\varphi(x) \in C^1_c(U)$ (or more precisely, the lift of such a function to $G(U)$), using Theorem \ref{T:Brakke-perpendicularity} and \eqref{E:formal-properties-partial-V}, we see that $V$ has locally bounded first variation and \eqref{E:relation-to-first-variation} holds (c.f. \cite[p.827--829]{Mantegazza1996}).
\end{proof}

\begin{remark}
    If $V \in \bRV_{n-1}(U)$ is not integral, then $\pi_* \|\partial V\|$ is not necessarily singular with respect to $\|V\|$ since $\partial V$ will also contain information about the tangential derivative of the density $\Theta_V$. See Example \ref{Ex:generalised-sff-smooth}. 
\end{remark}

\begin{example}[Triple junction as curvature varifold with boundary]
Consider the unit-density varifold $V \in \bIV_{1}(\mathbb{R}^2)$ given by three rays from the origin $\ell^+_1, \ell^+_2, \ell^+_3$ meeting at equal angle of $2\pi/3$. Denote by $\vec{n}_i$ the unit vector centred at the origin pointing towards $\ell_i$. Due to the balancing of $\vec{n}_i$ at the origin, $V$ is a stationary varifold in $\mathbb{R}^2$, hence $H_V \equiv 0$ and $\|\delta V\|_{sing}=0$. One can check from the definition that $V$ is a curvature varifold with boundary with $\mathbf{B}_V \equiv 0$ and
\[ \partial V = \sum_{i=1}^3 \vec{n}_i \delta_{(0,\ell_i)} \neq 0 \]
where $\ell_i \in \Gr_0(2,1)$ is the unique line containing the ray $\ell^+_i$. Note that $\pi_* (\partial V)= \left(\sum_{i=1}^3 \vec{n}_i\right) \delta_0=0$. On the other hand, $\|\partial V\|=\sum_{i=1}^3 \delta_{(0,\ell_i)}$ is a sum of three delta measures on the Grassmannian $G(\mathbb{R}^2)$ and the pushforward measure $\pi_* \|\partial V\|=3 \delta_0$ is singular with respect to $\|V\|$. This example shows that in general we only have $\|\delta V\|_{sing} \leq \pi_* \|\partial V\|$ as Radon measures on $U$.
\end{example}

\begin{example}[Weighted smoothly embedded hypersurfaces]
\label{Ex:generalised-sff-smooth}
Let $\Sigma \subset \mathbb{R}^n$ be a properly embedded smooth hypersurface with boundary and $f:\Sigma \to \mathbb{R}$ be a smooth function. Suppose $\Sigma$ is two-sided and oriented by the global unit normal vector field $N_\Sigma$. Consider the rectifiable varifold $V \in \bRV_{n-1}(\mathbb{R}^n)$ with weight measure given by
\[ \|V\|=e^f d\Hcal^{n-1} \mres \Sigma,\]
then a direct computation shows that $V$ is a curvature varifold with boundary with $\mathbf{B}_V:\mathbb{R}^n \times \mathbb{R}^n \to \mathbb{R}^n$ supported on $T\Sigma$ given by the classical second fundamental form\footnote{We denote $w=w^T+w^N$ for the orthogonal decomposition $T_x \mathbb{R}^n=T_x \Sigma \oplus N_x \Sigma$}: 
\[ \mathbf{B}_V|_{(x,T_x \Sigma)}(u,v):= (D_{u^T} v^T)^N,\]
and the vector-valued Radon measure $\partial V$ is given by
\[ \partial V(x,T_x \Sigma)= -e^f\nabla^\Sigma f (x) \, d \Hcal^{n-1} \mres \Sigma (x) \otimes \delta_{T_x\Sigma} + e^f\eta(x) \, d\Hcal^{n-2} \mres \partial \Sigma (x) \otimes \delta_{T_x\Sigma}\]
where $\eta(x) \in T_x\Sigma$ is the outward unit co-normal of $\partial \Sigma$. Note that $\partial V$ contains a non-trivial part which is absolutely continuous with respect to $V$ unless $f$ is constant on $\Sigma$. On the other hand, $V$ has locally bounded first variation with $H_V(x)=\mathbf{H}_V(x,T_x\Sigma)$ and
\[ \|\delta V\|_{sing}= e^fd\Hcal^{n-2} \mres \partial \Sigma \leq e^f|\nabla^\Sigma f| \, d \Hcal^{n-1} \mres \Sigma + e^fd\Hcal^{n-2} \mres \partial \Sigma = \pi_* \|\partial V\|. \]
More generally, one can also consider $f \in BV(\Sigma)$ and replace $\nabla^\Sigma f  \, d \Hcal^{n-1} \mres \Sigma$ with the vector-valued Radon measure $Df$ on $\Sigma$.
\end{example}

Finally, we recall the following boundary rectifiability theorem \cite[Theorem 7.1]{Mantegazza1996} in the context of integral curvature varifold with boundary.

\begin{theorem}[Boundary Rectifiability Theorem]
\label{T:boundary-rectifiability}
     Let $V \in \bCV^\partial_{n-1}(U) \cap \bIV_{n-1}(U)$. Then, $\pi_* \|\partial V\|$ is an $(n-2)$-rectifiable measure on $U$, i.e. there exists a countably $(n-2)$-rectifiable subset $\Gamma \subset U$ and a positive Borel function $\tau :\Gamma \to (0,\infty)$ such that 
     \[ \pi_* \|\partial V\|=\tau \; d\Hcal^{n-2}\mres \Gamma.\]
\end{theorem}

\section{The Allen-Cahn equation on Riemannian manifolds}
\label{S:Allen-Cahn}

In this section, we will study the geometric properties of limit-interfaces arising in stationary or stable solutions to the Allen-Cahn equation on Riemannian manifolds with boundary. Let $\Omega$ be a domain in $(\tM,g)$ as in Section \ref{SS:setup} and fix an open set $U \subset \tM$. For any $\epsilon>0$, we consider the Allen-Cahn energy for any $u \in W^{1,2}(\Omega \cap U)$,
\begin{equation}
    \label{E:AC-energy}
E_\epsilon (u):= \int_{\Omega \cap U} \left( \frac{\ep}{2}|\nabla u|^2+\frac{W(u)}{\ep} \right) \; \dvol_g 
\end{equation} 
where $W:\mathbb{R} \to \mathbb{R}$ is a ``double-well potential'' which is a non-negative smooth function satisfying the following (a common choice is $W(s)=\frac{1}{4}(1-s^2)^2$ with the unique local maximum at $\gamma=0$):
\begin{itemize}
\item $W(\pm1)=W'(\pm1)=0$,
\item $W$ has a unique local non-degenerate maxima at $\gamma \in (-1,1)$,
\item there exist some $\alpha \in (0,1)$ and $\kappa \in (0,1)$ such that $W''(t) \geq \kappa$ for all $|t| \geq \alpha$.
\end{itemize}
We shall denote the energy constant\footnote{Note that our definition of the constant $h_0$ is the same as in our previous paper \cite{LiPariseSarnataro2023}, but differs by a factor of 2 from \cite{Hutchinson2000,Guaraco2018}.} with respect to $W$ by
\begin{equation}
\label{E:energy-constant}
h_0:=\int_{-1}^1 \sqrt{2 W(s)} \; ds.
\end{equation}
For easy reference, we will also denote the (normalized) \emph{energy density} $e_\epsilon(u)$ and the \emph{discrepancy function} $\xi_\epsilon(u)$ for any $u \in C^1(\Omega \cap U)$ respectively by
\begin{equation}
    \label{E:energy-density-discrepancy}
e_\epsilon(u):=\frac{1}{h_0} \left(\frac{\ep}{2}|\nabla u|^2+\frac{W(u)}{\ep}\right) \quad \text{ and } \quad \xi_\epsilon(u):= \frac{1}{h_0} \left(\frac{\ep}{2}|\nabla u|^2-\frac{W(u)}{\ep}\right).
\end{equation} 
Note that $e_\epsilon(u) \geq 0$ and $\xi_\epsilon(u) \leq e_\epsilon(u)$ everywhere.

For any $u_\ep \in W^{1,2}(\Omega \cap U)$ with $\ep>0$, we can associate a Radon measure $\mu_{u_{\ep}}$ on $U$ by 
\begin{equation}
\label{E:mu}
	\mu_{u_{\ep}}\coloneqq\frac{1}{h_0}\ep|\nabla u_{\ep} |^2\ \dvol_g \mres {\Omega \cap U},
\end{equation}
and a varifold $V^{u_{\ep}} \in \mathbf{V}_{n-1}(U)$ by
\begin{equation}
\label{E:varifold-def}
	V^{u_{\ep}}(\phi) := \int_{\{x \in \Omega \cap U: |\nabla u_\ep(x)|\neq 0\}}\phi\left(x, S^{u_\ep}(x)\right)\ d\mu_{u_{\ep}}(x) \qquad \forall \phi\in C_c(G(U))
\end{equation}
where $S^{u_\ep}(x) \in \Gr_x(n,n-1)$ is the hyperplane orthogonal to $\nabla u_\ep(x) \neq 0$. Note that $\mu_{u_{\ep}}$ is absolutely continuous with respect to the ambient volume measure $\dvol_g$ and is supported on the closure of $\{x \in \Omega \cap U: |\nabla u_\ep(x)|\neq 0\}$ in $U$ which is in general not an $(n-1)$-rectifiable subset. Therefore, although the varifold $V^{u_{\ep}}$ is concentrated on $S^{u_\ep}$ along the Grassmannian factor, the spatial support of $\|V^{u_{\ep}}\|=\mu_{u_{\ep}}$ is diffused.  

\begin{remark}
\label{R:varifold-definition}
	Our definition of the varifold $V^{u_\ep}$ in \eqref{E:varifold-def} is consistent with \cite{Hiesmayr2017, Chodosh2020, Mantoulidis2021}, but slightly different from the one in \cite{LiPariseSarnataro2023}. Nonetheless, whenever \textit{equi-partition of energy} holds (i.e. $\xi_{\ep} \to 0$ in $L^1(U)$), it will give rise to the same varifold limit $V^{u_\ep} \to V$ if it exists as $\ep\to 0$.
\end{remark}

\subsection{General critical points to the Allen--Cahn energy}
\label{SS:our-old-theorem}

We first recall the key properties of general critical points to the Allen-Cahn energy functional, which study was pioneered by Hutchinson and Tonegawa in \cite{Hutchinson2000} for the Euclidean setting; the Riemannian setting (c.f. \cite[Appendix B]{Guaraco2018}) together with the boundary behaviour was studied in our previous paper \cite{LiPariseSarnataro2023}. The Euler-Lagrange equation for \eqref{E:AC-energy} is given by the following elliptic semi-linear Allen-Cahn equation with homogeneous Neumann boundary condition:
\begin{equation}
	\label{E:Allen-Cahn}
	\left\{      
	\begin{array}{rl}
		-\ep^2 \Delta u + W'(u) &=0 \quad \text{in $\Omega \cap U$}\\
		\partial u / \partial \nu &=0 \quad \text{on $\partial \Omega \cap U$.}
	\end{array}
	\right. 
\end{equation}
By standard elliptic regularity, any weak solution $u \in W^{1,2}(\Omega \cap U)$ to \eqref{E:Allen-Cahn} is in fact a classical solution $u \in C^3(\overline{\Omega} \cap U)$. 

\begin{assumptions}
\label{As:stationary-AC}
We will consider solutions $u_\ep$ to \eqref{E:Allen-Cahn} under the following additional assumptions:
\begin{itemize}
    \item[(A1)] (Uniform energy bound) There exists some $E_0>0$ such that
    \[ E_{\ep}(u_\ep) \leq E_0.\]
    \item[(A2)] (Uniform $L^\infty$ bound on $u_\ep$) There exists some $C_0>0$ such that
    \[ \|u_\ep\|_{L^\infty (\Omega \cap U)} \leq C_0.\]
\end{itemize}
\end{assumptions}



The theorem below summarizes the key properties of the limit-interface arising from general critical points to the Allen-Cahn energy. Note that the main theorem \cite[Theorem 3.3]{LiPariseSarnataro2023} there actually holds with trivial modifications of the proof to the situation of varying metrics and domains as well.

\begin{theorem}{\cite[Theorem 3.3]{LiPariseSarnataro2023}}
\label{T:old-main}
Let $\{\ep_i\}_{i=1}^\infty$ be a sequence of positive real numbers such that $\ep_i\to 0$, and let $\{u_{\ep_i}\}_{i=1}^\infty$ be a sequence of solutions to \eqref{E:Allen-Cahn} satisfying (A1)-(A2) in Assumptions \ref{As:stationary-AC} with respect to a sequence of Riemannian metrics $g_i$ on $U$ and a sequence of domains $\Omega_i \subset \tM$ such that
\[ g_i \to g \quad \text{ and } \quad \Omega_i \to \Omega\]
in the $C^2$-sense to some Riemannian metric $g$ and some $C^2$ domain $\Omega \subset \tM$ satisfying all the assumptions in \ref{SS:setup}.
Then, after passing to a subsequence, there exist some $u \in BV(\Omega \cap U)$ and $V \in \bV_{n-1}(U)$ such that
\begin{itemize}
\item $u_{\ep_i} \to u$ strongly in $L^1(\Omega \cap U;\dvol_g)$; 
\item $V^{u_{\epsilon_i}} \to V$ as varifolds in $U$ (recall \eqref{E:varifold-def}).
\end{itemize}
Moreover, the following statements are true for any such $u$ and $V$ obtained above:
\begin{itemize}
\item[(1)] (Discrepancy estimates and equi-partition of energy) The discrepancy function $\xi_{\ep_i}(u_{\ep_i})$ as defined in \eqref{E:energy-density-discrepancy} satisfies the pointwise upper bound locally away from $\partial U$:
\begin{equation}
    \label{E:discrepancy-pointwise-bound}
    \xi_{\ep_i}(u_{\ep_i}) \leq \ep_i^{-4/5} \qquad \text{ for all large $i$};
\end{equation}
moreover the discrepancy function converges to zero strongly in $L^1$, i.e.
\begin{equation}
    \label{E:vanishing-discrepancy-limit}
    \xi_{\ep_i}(u_{\ep_i}) \to 0 \quad \text{ in } L^1(\Omega \cap U;\dvol_g).
\end{equation}
As a result, we have for each $\phi \in C_c(U)$,
\begin{equation}
\label{E:equipartition energy}
\|V\|(\phi)= \lim_{i \to \infty} \mu_{u_{\ep_i}}(\phi)= \lim_{i \to \infty} \frac{2}{h_0} \int_{\overline{\Omega} \cap U} \frac{W(u_{\ep_i})}{\ep_i} \phi \; \dvol_g.
\end{equation}
\item[(2)] (Hausdorff convergence of level sets) The phase boundary is contained in the support of the limit-interface, i.e. 
\[ \partial^* \{u=1\} \subset \spt \|V\| \]
where $\partial^* \{u=1\} \subset \Omega \cap U$ is the reduced boundary\footnote{See, for example, \cite{Koltai-Volkner} for a precise definition in the local Riemannian setting.} of $u \in BV(\Omega \cap U)$. Moreover, as $i \to \infty$,
\[ u_{\ep_i} \to \pm 1 \quad \text{ locally uniformly on $(\Omega \cap U) \setminus \spt \|V\|$}\]
and for each $b \in (0,1)$, 
\[ \{x \in \overline{\Omega} \cap U : |u_{\ep_i}(x)| \leq 1-b\} \to \spt \|V\| \quad  \text{ locally in Hausdorff topology.}\]
\item[(3)] (Integer-rectifiability of limit-interface) The limit-interface is integer-rectifiable, i.e. $V \in \bIV_{n-1}(U)$, with density 
\[ \Theta_V(x)= \left\{ \begin{array}{cl}
\text{odd} & \text{for $\mathcal{H}^{n-1}$-a.e. $x \in \partial^* \{u=1\}$},\\
\text{even} & \text{for $\mathcal{H}^{n-1}$-a.e. $x \in (\spt \|V\| \cap \Omega) \setminus \partial^* \{u=1\}$}.
\end{array} \right.
\]
\item[(4)] (Variational property of limit-interface) The limit-interface $V \in \bIV_{n-1}(U)$ is free boundary stationary, with respect to the metric $g$, with free boundary in $\partial \Omega \cap U$ and $\spt \|V\| \subset \overline{\Omega} \cap U$.
\end{itemize}
\end{theorem}

\begin{remark}\label{R:old-theorem-optimal}
	The assumptions (A1) and (A2) are really the minimal assumptions one has to make to expect the conclusions in Theorem \ref{T:old-main}. Condition (A2) typically follows from a maximum principle argument on the global Allen-Cahn equation \eqref{E:Allen-Cahn-intro}. Condition (A1) is necessary in order to guarantee a uniform bound on the BV-norm of $u_{\ep_i}$ and the mass of $V^{u_{\ep_i}}$ so that standard compactness theorems can be applied to extract a convergent subsequence. In particular, there are no assumptions on the stability (or Morse index) of $u_{\ep_i}$ with respect to the Allen-Cahn energy functional $E_{\ep_i}$. We emphasize that Theorem \ref{T:old-main} allows higher multiplicity convergence and boundary concentration phenomena as constructed by Malchiodi-Ni-Wei \cite{Malchiodi2007a} and Malchiodi-Wei \cite{Malchiodi2007}. As these examples show, in general the limit-interface $V$ may not be stationary in $U$ if there is a part of $\spt \|V\|$ leaning onto the boundary $\partial \Omega$. Nonetheless, the first variation of $V$ is given by \eqref{E:1st-variation-fbv-stationary} but it is a difficult question to infer further regularity (e.g. whether $\Hcal^{n-1}(\sing V)=0$) as a-priori the touching part can interact with $\|\delta V\|_{sing}$ in highly complicated manners.
\end{remark}

\begin{remark}\label{R:up-to-boundary-Hausdorff-convergence}
    Note that we have only established in \cite{LiPariseSarnataro2023} Hausdorff convergence of the transition level sets to the support of the limit-interface \emph{locally away from $\partial \Omega$}. This is due to the technical arguments that in \cite{LiPariseSarnataro2023} we only make use of monotonicity formula on balls that are either centred at a boundary point or disjoint from $\partial \Omega$. The last assertion of (2) in Theorem \ref{T:old-main} essentially relies on a uniform density-ratio bounds which have to be also valid for balls that intersect $\partial \Omega$ but are not centred at a boundary point. This can be achieved by a reflected ball argument as in \cite{Grueter1986} and \cite{Tonegawa2003}. For the sake of completeness, we provide the proof in Appendix \ref{A:reflected-density}.
\end{remark}

As highlighted in Remark \ref{R:old-theorem-optimal}, Theorem \ref{T:old-main} is likely to be sharp in such generality. However, as the examples in \cite{Malchiodi2007a,Malchiodi2007} show, a free boundary stationary integral varifold can be quite far from a genuine free boundary minimal hypersurface. In order to obtain better regularity properties for the limit interface, it is necessary to impose stronger assumptions on $u_{\ep_i}$. Since $u_{\ep_i}$ are critical points to a variational problem, a natural assumption is that of stability (or more generally, finite Morse index), which is going to be studied in the following sections.

\subsection{Stable critical points to Allen-Cahn equation}
\label{SS:Stability}

We now study further properties of the limit-interface $V$ arising from Theorem \ref{T:old-main} under the additional assumption that each $u_{\ep_i}$ is stable. 

\begin{assumptions}
\label{As:stable-AC}
From this section onwards, we will consider solutions $u_\ep$ to \eqref{E:Allen-Cahn} under the following additional assumption:
\begin{itemize}
    \item[(A3)] (Stability) Each $u_\ep$ is a stable critical point in the sense that the second variation of $E_\ep$ is non-negative-definite at $u_\ep$, i.e.
\begin{equation}
\label{E:def-stability}
	\left. \frac{d^2}{dt^2} \right\vert_{t = 0} E_\ep(u_\ep + t\psi)\geq 0 \qquad \forall \psi \in C_c^1(U).
\end{equation}
\end{itemize}
\end{assumptions}

Note that the test function $\psi$ in \eqref{E:def-stability} is allowed to be non-zero along $\partial \Omega \cap U$, and therefore $u_{\ep}$ is stable as a ``free boundary problem'' (in the sense that the boundary values of $u_{\ep}$ are allowed to change freely on $\partial \Omega \cap U$ in the variational problem). Recall that $u \in C^3(\overline{\Omega} \cap U)$ for any solution to \eqref{E:Allen-Cahn} and thus we can define two Radon measures $\mu^\partial_{u_\ep}$ and $\hat{\mu}^\partial_{u_\ep}$ supported on $\partial \Omega \cap U$ by
\begin{equation}
\label{E:mu-partial}
	\mu^\partial_{u_{\ep}}\coloneqq\frac{1}{h_0}\ep|\nabla u_{\ep} |^2 d\Hcal^{n-1} \mres {\partial \Omega \cap U}
\end{equation}
and 
\begin{equation}
\label{E:mu-hat-partial}
	\hat{\mu}^\partial_{u_{\ep}}\coloneqq e_{\ep}(u_\ep) d\Hcal^{n-1} \mres {\partial \Omega \cap U}.
\end{equation}
Clearly, we have $\mu^\partial_{u_\ep} \leq 2\hat{\mu}^\partial_{u_\ep}$ as measures but in general we do not have $\hat{\mu}^\partial_{u_\ep} \leq C \mu^\partial_{u_\ep}$ unless we have pointwise lower bound of the discrepancy function $\xi_{\ep}(u_\ep)$ up to the boundary $\partial \Omega$.

We shall also denote the unit normal $N_{u_\ep}$ to the (regular part of) level sets of $u_\ep$, which is a $C^2$ vector field on $\overline{\Omega} \cap U \cap \{\nabla u_\ep \neq 0\}$, defined by
\begin{equation}
    \label{E:u-normal}
 N_{u_\ep}:=\frac{\nabla u_\ep}{|\nabla u_\ep|}
 \end{equation}
We also recall the definition of the enhanced second fundamental form $\mathcal{A}_{u_\ep}$, which is a $(1,1)$-tensor of class $C^1$ in $\overline{\Omega} \cap U \cap \{\nabla u_\ep \neq 0\}$ defined by
\begin{equation}
\label{E:enhanced-sff}
   \mathcal{A}_{u_\ep}  := \nabla N_{u_\ep}. 
\end{equation}
By convention, we shall just let $N_{u_\ep}$ and $\mathcal{A}_{u_\ep}$ vanish on $\{|\nabla u_\ep| = 0\}$\footnote{Using this convention, $N_{u_\ep}$ and $\mathcal{A}_{u_\ep}$ will be discontinuous across the exceptional set $\{|\nabla u_\ep| = 0\}$. However, as we will see, we will always pair $N_{u_\ep}$ and $\mathcal{A}_{u_\ep}$ with the Radon measures such as $\mu_{u_{\ep}}$ and $\mu^\partial_{u_{\ep}}$ that vanish on this exceptional set so that we will only be integrating over the set $\{|\nabla u_\ep| \neq 0\}$.}. The pointwise Hilbert-Schmidt norm of $\mathcal{A}_{u_\ep}$ on $\{|\nabla u_\ep| \neq 0\}$ is given by
\[ |\mathcal{A}_{u_\ep}|^2 = \frac{|\Hess u|^2 - |\nabla |\nabla u||^2}{|\nabla u_\ep|^2}.  \]




We first derive an $\ep$-stability inequality for stable solutions to the Allen-Cahn equation \eqref{E:Allen-Cahn}. The same inequality was first established in the interior Euclidean setting by Padilla-Tonegawa \cite{Padilla1998} (c.f. \cite{Tonegawa2005a}) and then the Riemannian setting by Hiesmayr \cite{Hiesmayr2017}. We extend this to the Riemannian setting including the boundary terms.

\begin{proposition}[$\ep$-stability inequality]
    \label{P:ep-stability}
    Suppose $u_\ep$ is a stable solution to \eqref{E:Allen-Cahn} satisfying (A3) in Assumption \ref{As:stable-AC}. Then the following inequality holds: for any $\psi \in C^1_c(U)$,
    \begin{equation}
        \label{E:ep-stability}
        \int |\mathcal{A}_{u_\ep}|^2 \psi^2 \; d\mu_{u_\ep} \leq \int |\nabla \psi|^2 + \Ric ( N_{u_\ep},N_{u_\ep}) \psi^2  \; d\mu_{u_\ep}  + \int \sff_{\partial \Omega} (N_{u_\ep},N_{u_\ep}) \psi^2 \; d\mu^\partial_{u_\ep}
    \end{equation}
    where $\Ric$ is the Ricci curvature of $(\tM,g)$ and $\sff_{\partial \Omega}$ is the second fundamental form\footnote{Our sign convention here is that $\sff_{\partial \Omega}$ is positive definite for a convex domain $\Omega \subset (\tM,g)$.} of $\partial \Omega$ with respect to the outward unit normal $\nu$.
\end{proposition}

\begin{proof}
A direct calculation as in \cite{Padilla1998} and integrating by parts gives for any $\phi \in C^1_c(U)$,
    \begin{equation}
    \label{E:2nd-variation-boundary} 
	\begin{split}
    	0 \leq \left. \frac{d^2}{dt^2} \right\vert_{t = 0} E_\ep(u_\ep + t\phi) & = \int_{\Omega\cap U} \left( \ep \vert \nabla \phi \vert^2 + \frac{1}{\ep} W^{\prime \prime}(u_\ep) \phi^2 \right) \, \dvol_g, \\
    	& = \int_{\Omega \cap U} \phi\left( - \ep \Delta \phi + \frac{1}{\ep} W^{\prime \prime}(u_\ep) \phi \right) \dvol_g + \frac{\ep}{2} \int_{\partial \Omega \cap U} \frac{\partial \phi^2}{\partial \nu} \, d\Hcal^{n-1} \\
	\end{split}
\end{equation}
To rewrite the above in the form of \eqref{E:ep-stability}, first recall by the Bochner formula that 
\begin{equation}
\label{E:Bochner-u-ep}
    \frac{\ep}{2}\Delta \vert \nabla u_\ep \vert^2 = \ep\vert \Hess u_\ep \vert^2 + \frac{1}{\ep} W^{\prime \prime}(u_\ep) \vert \nabla u_\ep \vert^2 + \ep \Ric (\nabla u_\ep, \nabla u_\ep)
\end{equation}
and also the boundary analogue of Bochner formula (c.f. \cite[Lemma 3.1]{Tonegawa2003})
\begin{equation}
\label{E:Bochner-u-ep-boundary}
    \frac{\partial |\nabla u_\ep|^2}{\partial \nu} = - 2 \sff_{\partial \Omega}(\nabla u_\ep, \nabla u_\ep).  
\end{equation}
We can then\footnote{Although \eqref{E:2nd-variation-boundary} was derived for $C^1$ test functions $\phi$, it is standard to see by an approximation argument that our choice $\psi |\nabla u_\ep|$ is a valid test function even though $|\nabla u_\ep|$ is in general not $C^1$ across the exceptional set $\{|\nabla u_\ep|=0\}$. The interested readers can refer to \cite[Proposition 2.6]{Padilla1998} for more details. Also, technically $\psi |\nabla u_\ep|$ is only defined on $\overline{\Omega} \cap U$ but we can extend it suitably across $\partial \Omega \cap U$ which does not affect the integrals in \eqref{E:2nd-variation-boundary}.} plug $\phi=\psi |\nabla u_\ep|$ into \eqref{E:2nd-variation-boundary} to obtain \eqref{E:ep-stability}. 
\end{proof}

We now move on to the main result of this section concerning the properties of limit-interfaces arising from stable solutions to the Allen-Cahn equation. 

\begin{theorem}
\label{T:main-AC}
Let $\{\ep_i\}_{i=1}^\infty$ be a sequence of positive real numbers such that $\ep_i\to 0$, and let $\{u_{\ep_i}\}_{i=1}^\infty$ be a sequence of solutions to \eqref{E:Allen-Cahn} satisfying (A1)-(A2) in Assumptions \ref{As:stationary-AC}. 
Suppose, furthermore, that each $u_{\ep_i}$ is stable, i.e. (A3) in Assumption \ref{As:stable-AC} is satisfied, with respect to the metric $g_i$. Then, for any limit-interface $V \in \bIV_{n-1}(U)$ obtained from Theorem \ref{T:old-main}, the following holds:
\begin{enumerate}
\item $V$ is a curvature varifold with boundary (recall Definition \ref{D:curvature-varifolds}), i.e. 
\[ V \in \bIV_{n-1}(U) \cap \bCV^\partial_{n-1}(U).\]
\item (Ambient stability inequality) For any $\psi \in C^1_c(U)$, we have
\begin{equation}
\label{E:stability-inequality-boundary}
    \int |\mathbf{B}_V|^2\psi^2 \, d\|V\| \leq \int |\nabla\psi|^2 \, d\|V\| + (n-1) k_0 \int \psi^2 \, d\|V\| + 2 \kappa_0 \int \psi^2 \, d \|\delta V\|_{sing},
\end{equation}
where $|\mathbf{B}_V|^2(x)=|\mathbf{B}_V(x,T_x V)|^2$ with $T_x V$ being the approximate tangent plane of $\|V\|$ that exists at $\|V\|$-a.e. $x \in U$.
\item (Stationarity across touching set) The first variation of $V$ is given by
\begin{equation}
\label{E:minimal-touching}
    \delta V(X)= \int X \cdot \nu \; d \|\delta V\|_{sing} \qquad \forall X \in \mathfrak{X}_c(U).
\end{equation}
\end{enumerate}

\begin{remark}
    By (1) and Theorem \ref{T:boundary-rectifiability}, $\pi_* \|\partial V\|$ is an $(n-2)$-rectifiable measure on $U$.
\end{remark}
\begin{remark}
Note that if $V=|\Sigma|$ is a unit-density varifold for some smoothly embedded, two-sided, stable free boundary minimal hypersurface $\Sigma \subset \Omega$ with $\partial \Sigma \subset \partial \Omega$, the classical \emph{intrinsic} stability inequality for $\Sigma$ reads (recall the notation as in Example \ref{Ex:generalised-sff-smooth}):
\[ \int_\Sigma |B_\Sigma|^2 \psi^2 \; d\Hcal^{n-1} \leq \int_\Sigma |\nabla^\Sigma \psi|^2 \; d\Hcal^{n-1} - \int_\Sigma \Ric (N_\Sigma,N_\Sigma)  \psi^2 \; d\Hcal^{n-1} - \int_{\partial \Sigma} \sff_{\partial \Omega}(N_\Sigma,N_\Sigma) \psi^2 \; d \Hcal^{n-2} \]
where $B_\Sigma$ is the classical second fundamental form of $\Sigma$. Comparing this with \eqref{E:stability-inequality-boundary}, note that the term involving the extrinsic gradient $\nabla \psi$ is replaced by the intrinsic gradient $\nabla^\Sigma \psi$ and the boundary term differs by a factor of two.   
\end{remark}

\begin{remark}
\label{R:stationarity-across-touching}
Since $V$ is a free boundary stationary varifold with free boundary in $\partial \Omega \cap U$ and $\spt \|V\| \subset \overline{\Omega} \cap U$, combining \eqref{E:minimal-touching} with \eqref{E:1st-variation-fbv-stationary}, we conclude that
\begin{equation}
\label{E:minimal-touching-2}
    H_{\partial \Omega}(x)=0 \qquad \text{for $\|V\|$-a.e. $x \in \partial \Omega \cap U$.}
\end{equation}
In other words, the part that $V$ leans on to the boundary must be a minimal portion of $\partial \Omega$.
\end{remark}

%
%
\end{theorem}

\begin{proof}
First of all, we show that that the first variation of the diffused varifold $V^{u_{\ep}} \in \bV_{n-1}(U)$ (recall \eqref{E:varifold-def}) associated to a solution $u_{\ep}$ to \eqref{E:Allen-Cahn} is given by: for any $\tilde{X} \in \mathfrak{X}_c(U)$,
\begin{equation}
\label{E:diffused-V-first-variation}
		\delta V^{u_\ep} (\tilde{X}) =\int_{\Omega \cap U}\xi_\ep(u_\ep)\Div_g \tilde{X} \; \dvol_g+\int_{\partial \Omega \cap U} e_\ep(u_\ep) (\tilde{X} \cdot \nu) \; d\mathcal{H}^{n-1}.
	\end{equation}
To prove \eqref{E:diffused-V-first-variation}, we first recall the following Pohozaev identity (see \cite[Lemma 4.1]{LiPariseSarnataro2023}, note that we are using the normalized energy density as in \eqref{E:energy-density-discrepancy})
	\begin{equation}\label{E:Pohozaev-1} 
		\int_{\Omega \cap U} \left( h_0 e_\ep(u_\ep)\Div_g \tilde{X} - \ep \nabla \tilde{X} (\nabla u_\ep, \nabla u_\ep) \right) \; \dvol_g 
		=\int_{\partial \Omega \cap U} h_0 e_\ep(u_\ep) (\tilde{X} \cdot \nu) \; d\mathcal{H}^{n-1}.
	\end{equation}
Observe that $\Div_{S^{u_\ep}} \tilde{X} = \Div_g \tilde{X} - \nabla \tilde{X} (N_{u_\ep},N_{u_\ep})$ and $\ep |\nabla u_\ep|^2 = h_0 (e_\ep(u_\ep)+\xi_\ep(u_\ep))$, the integrand in the left hand side of \eqref{E:Pohozaev-1} can be written as
\[ h_0 e_\ep(u_\ep)\Div_g \tilde{X} - \ep \nabla \tilde{X} (\nabla u_\ep, \nabla u_\ep)= - h_0 \xi_\ep(u_\ep) \Div_g \tilde X + \ep |\nabla u_\ep|^2 \Div_{S^{u_\ep}} \tilde{X} \]
from which \eqref{E:diffused-V-first-variation} follows (recall \eqref{E:varifold-def}).

Next, we test \eqref{E:diffused-V-first-variation} with a vector field\footnote{Strictly speaking, in order to deal with the singularities of $N_\ep$ across $\{|\nabla u_\ep|=0\}$, we should first consider $\varphi\left(x,Id-\frac{\nabla u_\ep\otimes\nabla u_\ep}{s+|\nabla u_\ep|^2}\right)$ for $s>0$ and then let $s \to 0$. For more details on this technical aspect, see \cite{Tonegawa2005a}. For simplicity of presentation, we will pretend $|\nabla u|\neq 0$ everywhere.}
\[ \tilde{X}(x)=\varphi(x,S^{u_\ep}(x)) S^{u_\ep}(x)(X(x))\]
for any $C^1$ vector field $X:U \to T\tM$ and any $\varphi \in C^1_c(G(B_r(x_0)))$ where $x_0 \in U$ and $r<\frac{1}{4}\min\{\dist_g(x_0,\partial U), i_0\}$. A direct computation shows that
\[ \Div_{S^{u_\ep}} \tilde{X}=S^{u_\ep}(\nabla^{hor} \varphi) \cdot X + \varphi \Div_{S^{u_\ep}} X + B_{u_\ep}(\nabla^{ver} \varphi, S^{u_\ep}(X)) \cdot N_{u_\ep} - \varphi H_{u_\ep} \cdot X \]
where everything is evaluated at either $x$ or $(x,S^{u_\ep}(x))$; here $B_{u_\ep}(x)$ and $H_{u_\ep}(x)$ are respectively the second fundamental form (as a $(1,2)$-tensor) and the mean curvature vector (see Example \ref{Ex:generalised-sff-smooth}) of the level set $\{u_\ep=u_\ep(x)\}$ of $u_\ep$ through the point $x \in \overline{\Omega} \cap U$. Note that the free boundary condition $\frac{\partial u_\ep}{\partial \nu}=0$ means that
\[ S^{u_\ep}(x)(\nu(x))=\nu(x) \quad \text{for all $x \in \partial \Omega$.} \]
Plugging these back into \eqref{E:diffused-V-first-variation}, we obtain
\[ \int_{\partial \Omega \cap U} \varphi(x,S^{u_\ep}(x)) X(x) \cdot \nu(x) e_\ep(u_\ep)(x) \, d\Hcal^{n-1} + \int_{\Omega \cap U} \xi_\ep(u_\ep)(x) \Div_g \tilde{X}(x) \, \dvol_g  \]    
\begin{equation}
\label{E:ep-curvature-varifold}
    =\int_{\Omega \cap U} S^{u_\ep}(x)\left(\nabla^{hor} \varphi (x,S^{u_\ep}(x)) \right) \cdot X(x) + \varphi(x,S^{u_\ep}(x)) \Div_{S^{u_\ep}}X(x) \, dV^{u_\ep}
\end{equation}
\[ + \int_{\Omega \cap U} B_{u_\ep}\left( \nabla^{ver} \varphi(x,S^{u_\ep}(x)), S^{u_\ep}(x)X(x) \right) \cdot N_{u_\ep}(x) - \varphi(x,S^{u_\ep}(x)) H_{u_\ep}(x) \cdot X(x) \, dV^{u_\ep}. \]
We will pass $\ep=\ep_i \to 0$ in \eqref{E:ep-curvature-varifold} to deduce that $V=\lim_{i \to \infty} V^{\ep_i}$ is in fact a curvature varifold with boundary. Moreover, we only present below the proof in the case $g_i=g$ and $\Omega_i=\Omega$ for all $i$ and leave the trivial modifications to the readers. In order to do this, we shall prove that all terms in \eqref{E:ep-curvature-varifold} can be suitably bounded as \textit{measure-function pairs} (see Appendix \ref{ap:measure-function}, or \cite{HutchinsonIUMJ, Tonegawa2005a}) and appeal to the compactness result of measure-function pairs as in Proposition \ref{prop:measure-function compactness}. With this goal in mind, we define a tensor-valued $\mathbf{B}_{V^{u_\ep}} \in L^1_{loc}(G(U);V^{u_\ep})$ with values in $T^*\tM \otimes T^* \tM \otimes T \tM$ by
\[ \mathbf{B}_{V^{u_\ep}}(X,Y)= B_{u_\ep}(S(X),S(Y)) \]
and a vector-valued Radon measure $\partial V^{u_\ep}$ on $G(U)$ with values in $T\tM$ by (recall \eqref{E:mu-hat-partial})
\[ \partial V^{u_\ep}= S(\nu(x)) \, d\hat{\mu}^\partial_{u_\ep} \otimes \delta_{S^{u_\ep}} \]
which clearly fulfil \eqref{E:formal-properties-B} and \eqref{E:formal-properties-partial-V}. By the well-known formula for the second fundamental form of a level set (see e.g. \cite[Lemma 3.3]{Hiesmayr2017}), 
		\begin{equation}
        \label{E:sff vs enhanced sff}
			|\mathbf{B}_{V^{u_{\ep_i}}}|^2(x, S) \leq |B_{u_\ep}|^2(x) \leq \frac{|\Hess u_\ep|^2-\big|\nabla|\nabla u_\ep|\big|^2}{|\nabla u_\ep|^2}=|\Acal_{u_\ep}|^2(x),
		\end{equation}
		where $\Acal_{u_\ep}$ denotes the enhanced second fundamental form as defined in \eqref{E:enhanced-sff}.

In order to apply the compactness theorem (Proposition \ref{prop:measure-function compactness}), we need to show that $\mathbf{B}_{V^{u_\ep}}$ is locally uniformly bounded in $L^2(B_r(x_0);V^{u_\ep})$ independent of $\ep$. Set (recall \eqref{E:cutoff}) $\psi(x)=\chi(\dist_g(x,x_0))$ in the $\ep$-stability inequality \eqref{E:ep-stability}, we see that (recall that $\mu^\partial_{u_\ep} \leq 2 \hat{\mu}^\partial_{u_\ep}$ by definitions \eqref{E:mu-partial} and \eqref{E:mu-hat-partial})
\begin{equation}
\label{E:B^ep is L^2 bounded}
			\begin{aligned}
				\int_{G(B_{r/2}(x_0))} |\mathbf{B}_{V^{u_\ep}}|^2(x, S) \, dV^{u_\ep} & \leq \int_{B_{r/2}(x_0)} |\Acal_{u_\ep}|^2\; d\|V^{u_\ep}\|	\\
                & \leq \left(\frac{4}{r^2} + (n-1)k_0 \right) \|V^{u_\ep}\|(B_r(x_0)) + 2 \kappa_0  \hat{\mu}^\partial_{u_\ep}(B_r(x_0))\\
			\end{aligned} 
		\end{equation}
By Assumptions (A1) in \ref{As:stationary-AC}, we have a uniform bound on the first term of \eqref{E:B^ep is L^2 bounded} since 
\[ \|V^{u_\ep}\|(B_r(x_0)) \leq \frac{2}{h_0}E_\ep (u_\ep) \leq \frac{2 E_0}{h_0}.\]
On the other hand, by \cite[Lemma 3.2]{LiPariseSarnataro2023} (with $\Lambda_0=0$ in our case), we also have a uniform bound on the second term of \eqref{E:B^ep is L^2 bounded}:
\begin{equation}
\label{E:mu-partial-bound}
\hat{\mu}^\partial_{u_\ep}(B_r(x_0)) \leq C(n,k_0,\kappa_0,r) \frac{E_\ep(u_\ep)}{h_0} \leq C(n,k_0,\kappa_0,r) \frac{E_0}{h_0}.
\end{equation} 
Putting these back into \eqref{E:B^ep is L^2 bounded}, we have the following local uniform $L^2$-bounds:
\begin{equation}
    \label{E:B^ep is L^2 bounded-2}
\int_{G(B_{r/2}(x_0))} |\mathbf{B}_{V^{u_\ep}}|^2(x, S) \, dV^{u_\ep} \leq C(n,k_0,\kappa_0,r) \frac{E_0}{h_0}.
\end{equation} 

We consider now a sequence of positive numbers $\ep_i\to 0$ and a sequence of solutions $\{u_{\ep_i}\}_{i=1}^\infty$ to \eqref{E:Allen-Cahn} as in Theorem \ref{T:old-main}. By \eqref{E:B^ep is L^2 bounded-2}, we can apply Proposition \ref{prop:measure-function compactness} in Appendix \ref{ap:measure-function}\footnote{The result there is applicable since no metric structure is required (except the fact that $X=G(B_r(x_0))$ is a locally compact Hausdorff space) and hence is independent of the Riemannian metric.} to extract a subsequence such that 
\[ (\mathbf{B}_{V^{u_{\ep_i}}},V^{u_{\ep_i}}) \weakto (\mathbf{B}_V,V) \]
where $\mathbf{B}_V \in L^1_{loc}(G(B_r(x_0));V) $ is a tensor-valued map with values in $T^*\tM \otimes T^* \tM \otimes T \tM$, and 
\begin{equation} 
\int  \mathbf{B}_{V^{u_{\ep_i}}}(X,Y)\cdot Z \; dV^{u_{\ep_i}} \to \int \mathbf{B}_V (X,Y)\cdot Z \; dV
\end{equation}
for any continuous vector fields $X,Y,Z$ on $G(B_r(x_0))$ with values in $T\tM$. This clearly implies $\mathbf{B}_V$ satisfies \eqref{E:formal-properties-B}. On the other hand, the uniform bound \eqref{E:mu-partial-bound} implies that, after passing to a subsequence, 
\[ \partial V^{u_{\ep_i}} \weakto \partial V=\textbf{n}_V \|\partial V\|\]
for some vector-valued Radon measure $\partial V$ on $G(B_r(x_0))$ with values in $T\tM$. From the definition of $\partial V^{u_{\ep_i}}$, it is clear that $\partial V$ satisfies \eqref{E:formal-properties-partial-V}. Passing $\ep_i \to 0$ in \eqref{E:ep-curvature-varifold}, using that $\xi_{\ep_i}(u_{\ep_i}) \to 0$ in $L^1(\Omega \cap U;\dvol_g)$ by Theorem \ref{T:old-main} (1), we obtain 
\[ \int \varphi(x,S) X(x) \cdot \textbf{n}_V(x,S) \; d\|\partial V\|=\int S \left(\nabla^{hor} \varphi (x,S) \right) \cdot X(x) + \varphi(x,S) \Div_{S} X(x) \; dV \] 
\begin{equation}
\label{E:V-curvature-varifold}
    + \int \mathbf{B}_V \left( \nabla^{ver} \varphi(x,S), S(X(x)) \right) \cdot N_S - \varphi(x,S) \mathbf{H}_V(x,S) \cdot X(x) \; dV
\end{equation}
where $\mathbf{H}_V=\Tr \mathbf{B}_V \in L^1_{loc}(G(B_r(x_0));V)$ with values in $T\tM$. This shows that $V$ is a curvature varifold with boundary in any geodesic ball $B_r(x_0)$ according to Definition \ref{D:curvature-varifolds}. Using a locally finite open covering of $U$ by such balls, (1) follows by Remark \ref{R:curvature-varifold-local}.

We now proceed to prove the stability type inequality \eqref{E:stability-inequality-boundary} in (2). First of all, since $\partial V^{u_{\ep_i}} \weakto \partial V$, combined with the fact that $V$ is a free boundary varifold with $\spt\|V\| \subset \overline{\Omega} \cap U$ by Theorem \ref{T:old-main} (4) and using Proposition \ref{P:free-boundary-characterization} together with \eqref{E:relation-to-first-variation}, we have
\begin{equation}
\label{E:partial-V-convergence}
\nu \,  \hat{\mu}^\partial_{u_{\ep_i}}=\pi_*(\partial V^{u_{\ep_i}}) \weakto \pi_*(\partial V)=\nu  \|\delta V\|_{sing}.
\end{equation}
Therefore, $\hat{\mu}^\partial_{u_{\ep_i}} \weakto \|\delta V\|_{sing}$. By the lower semi-continuity property \eqref{E:lower-semicontinuity-measure-function} in Proposition \ref{prop:measure-function compactness}, \eqref{E:sff vs enhanced sff} and the $\ep$-stability inequality \eqref{E:ep-stability}, since $V \in \bIV_{n-1}(U)$ by Theorem \ref{T:old-main} (3), for any $\psi\in C^1(U)$ we have (recall again that $\mu^\partial_{u_\ep} \leq 2\hat{\mu}^\partial_{u_\ep}$)
\[			\int |\mathbf{B}_V(x,T_xV)|^2\psi^2(x)\;d\|V\|(x)  \leq \liminf_{i \to\infty} \int |B_{u_{\ep_i}}(x)|^2 \psi^2(x) \; d \mu_{u_{\ep_i}}(x) \]
\[ \leq \int |\nabla \psi|^2 \; d\|V\| + (n-1)k_0 \int \psi^2 \; d\|V\| + 2 \kappa_0 \int \psi^2 \; d\|\delta V\|_{sing} \]
which is the desired inequality \eqref{E:stability-inequality-boundary}.

Finally, we prove assertion (3). Let $\ep=\ep_i \to 0$ in \eqref{E:diffused-V-first-variation}, by the continuity of first variation with respect to varifold convergence, together with Theorem \ref{T:old-main} (1) and \eqref{E:partial-V-convergence}, we obtain for any $X \in \mathfrak{X}_c(U)$, 
\[ \delta V(X)= \lim_{i \to \infty} \delta V^{u_{\ep_i}}(X)= \int X \cdot \nu \; d \|\delta V\|_{sing}.\]
\end{proof}

\begin{remark}[Test functions supported inside the regular set]
	Let $W \subset\subset U \setminus \mathrm{sing}(V)$ be an open set such that $W \cap \mathrm{reg}(V) \neq \emptyset$, and let $\phi\in C^1_c(W)$. In this case, we can get a more precise stability inequality as follows. First, note that, by \eqref{E:varifold-def}, 
	\begin{equation}
		\int_U \Ric(x) (N_{u_{\ep_i}(x)}, N_{u_{\ep_i}}(x)) \phi^2(x) \, d\|V^{u_{\ep_i}}\|(x)=\int_{G(U)} \Ric(x)(N_S, N_S) \phi^2 (x)\, d V^{u_{\ep_i}}(x,S).
	\end{equation}
	Letting $\ep_i \to 0$ and denoting $\Sigma=\reg V$, we have
    \[ \lim_{i \to \infty} \int_U \Ric(x) (N_{u_{\ep_i}(x)}, N_{u_{\ep_i}}(x)) \phi^2(x) \, d\|V^{u_{\ep_i}}\|(x)= \int_U \Ric (N_\Sigma,N_\Sigma) \phi^2\, d\|V\|. \]
	Therefore, when $\phi\in C^1_c(W)$ with $W \cap \mathrm{reg}(V) \neq \emptyset$, the stability inequality \eqref{E:stability-inequality-boundary} becomes
	\begin{equation}
		\int_U \left(|B_\Sigma|^2+\Ric (N_\Sigma, N_\Sigma) \right)\phi^2 \, d\|V\| \leq \int_U |\nabla \phi|^2 \, d\|V\| + 2 \kappa_0 \int_U \phi^2 \, d\|\delta V\|_{sing}.
	\end{equation}
    However, it is unclear whether one can recover the boundary term as in the classical case even when $\phi$ is supported inside $\reg V$ because we do not know the convergence of the Radon measures $\mu^\partial_{u_{\ep_i}}$ instead of $\hat{\mu}^\partial_{u_{\ep_i}}$.
\end{remark}

\begin{remark}
    By the trace inequality in Lemma \ref{L:trace-inequality}, \eqref{E:stability-inequality-boundary} implies for any $\delta>0$
    \begin{equation}
    \label{E:stability-inequality-boundary-removed}
    \int |\mathbf{B}_V|^2\psi^2 \, d\|V\| \leq (1+\delta)\int |\nabla\psi|^2 \, d\|V\| + C(n,i_0,k_0,\kappa_0,\delta) \int \psi^2 \, d\|V\|
    \end{equation}
    which is expressed entirely in interior terms resembling the concept of $\delta$-stability as studied in \cite{Hong-Li-Wang2024}, for example.
\end{remark}

\subsection{Excluding Classical Tangent Cones} 
\label{SS:exclude-classical-cones}

As a first step towards the optimal regularity of limit-interface $V$ arising from stable solutions to the Allen-Cahn equation as in Theorem \ref{T:main-AC}, we will show in this section that one can rule out classical tangent cones both in the interior and at the boundary for a generic set of singular points in $\sing V$. Our method mostly follows the work of \cite{Tonegawa2012} but extending to the Riemannian setting with boundary. Besides our previous work \cite{LiPariseSarnataro2023}, the conclusions in Theorem \ref{T:main-AC} will also play a major role in the arguments below.

We first recall the notion of classical tangent cones based at an interior point.

\begin{definition}[Interior classical tangent cones]
    \label{D:interior-classical-cone}
    Let $\bC \in \bIV_{n-1}(\mathbb{R}^n)$ be a stationary integral varifold cone (i.e. $(\eta_{0,r})_\sharp \bC=\bC$ for all $r>0$). We say that $\bC$ is an \emph{interior classical tangent cone} if 
    \[ \spt \|\bC\| = \bigcup_{j=1}^N P_j \]
    for some $N \geq 3$, and each $P_j \subset \mathbb{R}^n$ is a distinct half-hyperplane meeting along a common $(n-2)$-dimensional subspace.
\end{definition}

\begin{remark}
    By the stationarity and integrality of $\bC$, in fact we have as varifolds
    \[ \bC=\sum_{j=1}^N \Theta_j |P_j| \quad \text{ for some constants $\Theta_j \in \mathbb{N}$.} \]
    Moreover, if we denote $\eta_j$ to be the outward unit co-normal of $\partial P_j$ with respect to the half-hyperplane $P_j$, then stationarity implies the balancing condition
    \[ \sum_{j=1}^N \Theta_j \eta_j=0\]
    is satisfied along the spine $S(\bC)=\cap_{j=1}^N P_j$ (recall Definition \ref{D:spine}).
\end{remark}

Next, we introduce an analogous notion of classical tangent cones based at a boundary point.

\begin{definition}[Boundary classical tangent cones]
    \label{D:boundary-classical-cone}
    Let $\bC \in \bIV_{n-1}(\mathbb{R}^n)$ be a free boundary stationary integral varifold cone (i.e. $(\eta_{0,r})_\sharp \bC=\bC$ for all $r>0$) with free boundary on $\partial \mathbb{R}^n_+$ and $\spt \|\bC\| \subset \mathbb{R}^n_+$. We say that $\bC$ is a \emph{boundary classical tangent cone} if 
    \[ \spt \|\bC\| = \bigcup_{j=1}^N P_j \cap \mathbb{R}^n_+ \]
    for some $N \geq 2$, and each $P_j \subset \mathbb{R}^n$ is a distinct half-hyperplane meeting along a common $(n-2)$-dimensional subspace. We further divide the case according to the dichotomy \eqref{E:dichotomy-cones} on the spine of the doubled cone $\overline{\bC}$ (as defined in Lemma \ref{L:reflection stationary}). We say that $\bC$ is of
    \begin{itemize}
        \item \emph{Type I}: if $S(\overline{\bC}) \subset \partial \mathbb{R}^n_+$, $N \geq 2$, and each $P_j \subset \mathbb{R}^n_+$ but $\spt \|\bC\| \neq \partial \mathbb{R}^n_+$;
        \item \emph{Type II}: if $S(\overline{\bC}) \perp \partial \mathbb{R}^n_+$ and $N \geq 3$.
    \end{itemize}
\end{definition}

\begin{remark}
    By the free boundary stationarity and integrality of $\bC$, in fact we have as varifolds
    \[ \bC=\sum_{j=1}^N \Theta_j |P_j| \quad \text{ for some constants $\Theta_j \in \mathbb{N}$} \]
    if $\bC$ is of type I (c.f. Example \ref{Ex:boundary-classical-cones}). Moreover, if we denote $\eta_j$ to be the outward unit co-normal of $\partial P_j$ with respect to the half-hyperplane $P_j$, then free boundary stationarity implies the free boundary balancing condition
    \[ \sum_{j=1}^N \Theta_j \eta_j \perp \partial \mathbb{R}^n_+\]
    is satisfied along the spine $S(\overline{\bC})=\cap_{j=1}^N P_j \subset \partial \mathbb{R}^n_+$. If $\bC$ is of type II, then the doubled cone $\overline{\bC}$ is in fact an interior classical tangent cone with $S(\overline{\bC}) \perp \partial \mathbb{R}^n_+$.
\end{remark}

The main result of this section is the following:

\begin{theorem}[Exclusion of classical tangent cones]
\label{T:exclude-classical-cones}
Let $\{\ep_i\}_{i=1}^\infty$ be a sequence of positive real numbers such that $\ep_i\to 0$, and let $\{u_{\ep_i}\}_{i=1}^\infty$ be a sequence of solutions to \eqref{E:Allen-Cahn} satisfying (A1)-(A2) in Assumptions \ref{As:stationary-AC}. Assume, furthermore, that each $u_{\ep_i}$ is stable, i.e. (A3) in Assumption \ref{As:stable-AC} is satisfied. Then, for any limit-interface $V \in \bIV_{n-1}(U)$ obtained from Theorem \ref{T:old-main},
\[ \text{there exists a Borel set $\mathcal{Z} \subset \spt \|V\|$ with $\Hcal^{n-2}(\mathcal{Z})=0$} \]
such that for any $x \in \spt\|V\| \setminus \mathcal{Z}$, there does not exist $\bC \in \VarTan(V,x)$ which is an interior classical tangent cone in $T_x \tM$ when $x \in \Omega$; or a boundary classical tangent cone in $T^+_x \tM$ when $x \in \partial \Omega$.
\end{theorem}

\begin{proof}
Our proof follows the main strategy of \cite[Proposition 3.2]{Tonegawa2012} which treats the interior case in the Euclidean setting. First of all, the $\ep$-stability inequality \eqref{E:ep-stability} implies that after passing to a subsequence, we have
\[ |\Acal_{u_{\ep_i}}|^2 \; d\mu_{u_{\ep_i}} \weakto \mu\]
as Radon measures on $U$. Now, we define 
\begin{equation}
    \mathcal{Z}:=\left\{ x \in \spt\|V\| \; : \; \limsup_{r \to 0} \frac{\mu(B_r(x))}{r^{n-3}} >0  \right\},
\end{equation}
which is a Borel set since $\mu$ is a Radon measure on $U$. Moreover, we have $\Hcal^{n-3+\gamma}(\mathcal{Z})=0$ for every $\gamma>0$ (c.f. \cite[Remark 3.1 and Theorem 3.2]{Simon1983}).

We now show that such $\mathcal{Z}$ satisfies the assertion in the theorem by a contradiction argument as in \cite{Tonegawa2012}. Suppose on the contrary that there exists $x \in \spt\|V\| \setminus \mathcal{Z}$ such that some tangent cone of $V$ at $x$ is a classical tangent cone. We will focus on the case $x \in \partial \Omega$ is a boundary point (as the interior case is similar and has been treated in the Euclidean setting in \cite{Tonegawa2012}). Therefore, we would obtain a tangent cone $\bC \in \VarTan(V,x)$ at $x \in \partial \Omega$ which is a boundary classical tangent cone in $T^+_x \tM$. There are two cases to consider according to whether $\bC$ is of type I or type II as described in Definition \ref{D:boundary-classical-cone}.

We first look at the case where $\bC$ is of type I, i.e. 
\[ T_x \partial \Omega \neq \spt \|\bC\| = \bigcup_{j=1}^N P_j  \]
for some $N \geq 2$, and each $P_j \subset T^+_x \tM$ is a distinct half-hyperplane meeting along the $(n-2)$-dimensional spine $S(\overline{\bC}) \subset T_x \partial \Omega$. After a rotation, we may assume that $S(\overline{\bC})=\{0\} \times \mathbb{R}^{n-2} \subset \mathbb{R}^2_+ \times \mathbb{R}^{n-2}$ and the half-hyperplanes $P_j$ can be expressed as
\[ P_j=\{t \textbf{p}_j \; : \; t \geq 0\} \times \mathbb{R}^{n-2}\]
for some distinct vectors $\textbf{p}_j \in \mathbb{R}^2_+$ with $|\textbf{p}_j|=1/2$.

\vspace{1ex}

\textit{Step 1: Passing the analysis back to the $\ep$-level.}

\vspace{1ex}

Since $\bC \in \VarTan(V,x)$ is a tangent cone at $x$, we have $\lim_{i \to \infty} (\eta_{0,r_i})_\sharp (\exp_x^{-1})_\sharp V=\bC$ for some $r_i \to 0$. On the other hand, $V=\lim_{i \to \infty} V^{u_{\ep_i}}$ for some $\ep_i \to 0$, where each $u_{\ep_i}$ is a stable solution to \eqref{E:Allen-Cahn} with $\ep=\ep_i$ satisfying (A1)-(A2) in Assumptions \ref{As:stationary-AC}. After passing to a subsequence, we can assume without loss of generality that
\begin{equation}
\label{E:proof-exclude-classical-cone-1}
\lim_{i \to \infty} (\eta_{0,r_i})_\sharp (\exp_x^{-1})_\sharp V^{u_{\ep_i}}=\bC \quad \text{ and } \quad \lim_{i \to \infty} \frac{\ep_i}{r_i}=0. 
\end{equation}
Setting 
\[ \tilde{\ep}_i:=\frac{\ep_i}{r_i} \to 0 \quad \text{ and } \quad \tilde{u}_{\tilde{\ep}_i}:= u_{\ep_i} \circ \exp_x \circ  \eta_{0,r_i^{-1}}.\]
If we let $\Phi_{x,r_i}:=\exp_x \circ  \eta_{0,r_i^{-1}} :B^\delta_{i_0/2}(0) \to B^g_{r_ii_0/2}(x)$ which is a $C^2$ diffeomorphism taking $\Omega_i \cap B^\delta_R(0)$ to $\Omega \cap B^g_{Rr_i}(x)$ for any fixed $R < \min\{\frac{i_0}{2}, \frac{1}{2\kappa_0} \}$ where $\Omega_i \subset T_x\tM$ are $C^2$ domains satisfying
\begin{equation}
    \label{E:domain-convergence}
    \Omega_i \cap B^\delta_R(0) \text{ converges locally in $C^2$ to } T^+_x \tM \cap B^\delta_R(0),
\end{equation}
Therefore, $\tilde{u}_{\tilde{\ep}_i} \in C^2(\overline{\Omega}_i \cap B^\delta_R(0))$ are stable solutions to the homogeneous Allen-Cahn equation
\begin{equation}
\label{E:u-tilde-AC}
-\tilde{\ep}_i^2 \Delta_{\tilde{g}_i} \tilde{u}_{\tilde{\ep}_i} +W'(\tilde{u}_{\ep_i})=0 \qquad \text{in $\Omega_i \cap B^\delta_R(0)$} 
\end{equation}
where $\tilde{g}_i:= r_i^{-2}\Phi^*_{x,r_i} g$ is a Riemannian metric on $B^\delta_R(0)$, together with the homogeneous Neumann boundary condition
\[ \frac{\partial \tilde{u}_{\tilde{\ep}_i}}{\partial \tilde{\nu}_i}=0 \qquad \text{ along $\partial \Omega_i \cap B^\delta_R(0)$}\]
where $\tilde{\nu}_i$ is the outward unit normal of $\partial \Omega_i$ with respect to the metric $\tilde{g}_i$. Note that the Riemannian metrics $\tilde{g}_i$ converge uniformly on $B^\delta_R(0)$ in $C^2$ to the flat metric $\delta$ on $T_x \tM$ and $\Omega_i$ converges in $C^2$-sense to $T_x^+\tM$ by \eqref{E:domain-convergence}. 

We now check that $\tilde{u}_{\tilde{\ep}_i}$ satisfies uniform bounds (A1)-(A2) with $U=B^\delta_R(0) \subset T_x \tM$ with respect to the metric $\tilde{g}_i$ in Assumptions \ref{As:stationary-AC}. It is clear from the definition that (A2) is satisfied by $\tilde{u}_{\tilde{\ep}_i}$ with the same $C_0$. On the other hand, the uniform energy bound (A1), with a different $\tilde{E}_0$, can be seen from the almost-monotonicity formula applied to $u_{\ep_i}$ (c.f. \cite[Proposition 4.15]{LiPariseSarnataro2023}). Therefore, Theorem \ref{T:old-main} can be applied to the varifolds $V^{\tilde{u}_{\tilde{\ep}_i}} \in \bV_{n-1}(B^\delta_R(0))$ defined by
\begin{equation}
\label{E:varifold-def-tilde}
	V^{\tilde{u}_{\tilde{\ep}_i}}(\phi) := \int_{\{y \in \Omega_i \cap B^\delta_R(0): |\nabla^{\tilde{g}_i} \tilde{u}_{\tilde{\ep}_i}(y)|\neq 0\}}\phi\left(y, S^{\tilde{u}_{\tilde{\ep}_i}}(y)\right)\ d\mu_{\tilde{u}_{\tilde{\ep}_i}}(y) \qquad \forall \phi\in C_c(G(B^\delta_R(0))).
\end{equation}
It is straightforward from \eqref{E:proof-exclude-classical-cone-1} and the definitions that
\begin{equation}
    \label{E:V-tilde-convergence}
 V^{\tilde{u}_{\tilde{\ep}_i}} \to \bC \quad \text{as varifolds in $\bV_{n-1}(B^\delta_R(0))$}.
\end{equation}
Recall that $x \in \spt \|V\| \setminus \mathcal{Z}$, after passing to a subsequence of $\{\ep_i\}$, we can further assume that
\[ \lim_{i\to \infty} \int_{B^\delta_2(0)} |\tilde{\Acal}_{\tilde{u}_{\tilde{\ep}_i}}|^2 \; d\tilde{\mu}_{\tilde{u}_{\tilde{\ep}_i}}=  \lim_{i \to \infty} \frac{1}{(2r_i)^{n-3}} \int_{B^g_{2r_i}(x)} |\Acal_{u_{\ep_i}}|^2 d\mu_{u_{\ep_i}}=0 \]
where the first integral is computed with respect to the metric $\tilde{g}_i$. 

Therefore, after suitably passing to subsequences, we can reduce our problem to the following situation (we are omitting the tildes for simplicity): let $U=B_2 \subset \mathbb{R}^n$ denote the Euclidean ball of radius $2$ centred at the origin, there exists a sequence of positive real numbers $\{\ep_i\}_{i=1}^\infty$ such that
\[ \lim_{i \to \infty} \ep_i=0\]
and a sequence $\{u_{\ep_i}\}_{i=1}^\infty$ of $C^2$ solutions to the homogeneous Allen-Cahn equation
\[ -\ep_i^2 \Delta_{g_i}u_{\ep_i} +W'(u_{\ep_i})=0 \qquad \text{ in $\Omega_i \cap B_2$}\]
with homogeneous Neumann boundary condition
\[ \frac{\partial u_{\ep_i}}{\partial \nu_i} = 0 \qquad \text{ along $\partial \Omega_i \cap B_2$}\]
where $\{g_i\}_{i=1}^\infty$ is a sequence of $C^2$ Riemannian metrics on $B_2$ and $\{\Omega_i\}_{i=1}^\infty$ is a sequence of Euclidean $C^2$ domains in $\mathbb{R}^n$ with outward unit normal $\nu_i$ (with respect to $g_i$) such that
\begin{equation}
    \label{E:ambient-convergence}
    \lim_{i \to \infty} g_i = \delta \qquad \text{ and } \qquad \lim_{i \to \infty} \Omega_i = \mathbb{R}^n_+
\end{equation}
in the $C^2$-sense. Moreover, $\{u_{\ep_i}\}_{i=1}^\infty$ satisfies the uniform bounds (A1)-(A2) in Assumptions \ref{As:stationary-AC}; $\{g_i\}_{i=1}^\infty$ and are stable in the sense of \eqref{E:def-stability} in (A3), Assumptions \ref{As:stable-AC}. Moreover, we have in addition
\begin{equation}
    \label{E:blow-up-argument}
    \lim_{i \to \infty} V^{u_{\ep_i}}=\bC \qquad \text{ and } \qquad \lim_{i\to \infty} \int_{B_2} |\Acal_{u_{\ep_i}}|^2 \; d\mu_{u_{\ep_i}}=0.
\end{equation}
Moreover, we set up our coordinates $(y,z) \in  \mathbb{R}^2_+ \times \mathbb{R}^{n-2}$ so that
\[ S(\overline{\bC})=\{(0,z) \; : \; z \in \mathbb{R}^{n-2}\}=\{0\} \times \mathbb{R}^{n-2}.\]



\vspace{1ex}

\textit{Step 2: Show that, in the limit sense as $\ep_i \to 0$, for $\mathcal{L}^{n-2}$-a.e. point on the spine $S(\overline{\bC})$, the $y$-cross sections of all the intermediate level sets of $u_{\ep_i}$ are uniformly regular.}

\vspace{1ex}

More precisely, we claim that there exists some $c=c(W)>0$ such that if we denote the set of \emph{irregular cross sections} by
\[ \mathcal{I}_{u_{\ep_i}}:=\left\{ z \in \mathbb{R}^{n-2} \; : \; 
\begin{array}{l}
            \textrm{$|z| < 1$ and for any $r>0$, there exists some $z' \in B_r(z)$ such that }
            \\
            \textrm{$| \nabla^{g_i} u_{\ep_i}(y',z')|_{g_i} < \frac{c}{\ep_i}$ for some $|y'|<1$ and $|u_{\ep_i}(y',z')| \leq \frac{1}{2}$}
        \end{array} \right\}, \]
then 
\begin{equation}
\label{E:step-2-proof}
\lim_{i \to \infty} \mathcal{L}^{n-2}(\mathcal{I}_{u_{\ep_i}})=0.
\end{equation}

\begin{remark}
    Note that by standard elliptic estimates (c.f. \cite[Lemma 4.7]{LiPariseSarnataro2023}), we have the uniform gradient estimate $\ep_i | \nabla^{g_i} u_{\ep_i}(y',z')|_{g_i} \leq C$ for some large enough $C>0$. Therefore, \eqref{E:step-2-proof} amounts to showing that in a limiting measure-theoretic sense, inside the transition layer $\{|u_{\ep_i}|\leq \frac{1}{2}\}$, the gradient of $u_{\ep_i}$ is indeed of order $\frac{1}{\ep_i}$.
\end{remark}

The proof first involves an estimate of the left hand side of \eqref{E:step-2-proof} by a combination of covering arguments, standard elliptic estimates and the two-dimensional Gagliardo–Nirenberg–Sobolev inequality. Then, by the vanishing of the discrepancy function in the $L^1$-limit (Theorem \ref{T:old-main}(1)) and the simplifying assumption \eqref{E:blow-up-argument} from Step 1, we can obtain \eqref{E:step-2-proof} verbatim from the proof of \cite[Lemma 4.1]{Tonegawa2012}.

\vspace{1ex}

\textit{Step 3: Show that, in the limit sense as $\ep_i \to 0$, for $\mathcal{L}^{n-2}$-a.e. point on the spine $S(\overline{\bC})$, we have at a fixed distance from the spine the cross sections of all the intermediate level sets of $u_{\ep}$ concentrate near every $P_j$.}

\vspace{1ex}

Recall that 
\[\spt \|\bC\| = \bigcup_{j=1}^N \{t \textbf{p}_j \; : \; t \geq 0\} \times \mathbb{R}^{n-2}\]
for some $N \geq 2$ and distinct vectors $\textbf{p}_j \in \mathbb{R}^2_+$ with $|\textbf{p}_j|=1/2$.  Fix $\sigma \in (0,\frac{1}{4})$ such that $\sigma < \frac{1}{2}\min_{j \neq k}|\textbf{p}_j - \textbf{p}_k|$ so that the balls $B_{2\sigma}(\mathbf{p}_j)$ are mutually disjoint. Define the \emph{non-concentration set} along the spine by
\[ \mathcal{N}_{u_{\ep_i}}:=\left\{ z \in \mathbb{R}^{n-2} \; : \; 
\begin{array}{l}
            \textrm{$|z| < 1$, and there exists some $t' \in [-\frac{1}{2},\frac{1}{2}]$ and $j'=1,\cdots,N$ }
            \\
            \textrm{such that $u_{\ep_i}(y,z) \neq t'$ for all $y$ with $|y-\textbf{p}_{j'}|<\sigma$}
        \end{array} \right\}, \]
then we claim that
\begin{equation}
\label{E:step-3-proof}
\lim_{i \to \infty} \mathcal{L}^{n-2}(\mathcal{N}_{u_{\ep_i}})=0.
\end{equation}

To prove \eqref{E:step-3-proof}, it suffices to establish $ \lim_{i \to \infty} \mathcal{L}^{n-2}(\mathcal{N}^j_{u_{\ep_i}})=0$ for each $j=1,\cdots,N$, where 
\[ \mathcal{N}^j_{u_{\ep_i}}:=\left\{ z \in \mathbb{R}^{n-2} \; : \; 
\begin{array}{l}
            \textrm{$|z| < 1$, and there exists some $t' \in [-\frac{1}{2},\frac{1}{2}]$ }
            \\
            \textrm{such that $u_{\ep_i}(y,z) \neq t'$ for all $y$ with $|y-\textbf{p}_j|<\sigma$}
        \end{array} \right\}. \]

For notational simplicity, we will fix $j$ and write the vector $\textbf{p}=\textbf{p}_j \in \mathbb{R}^2_+$ and the $(n-1)$-dimensional half-hyperplane $P=P_j$. There are two cases to consider: either $\textbf{p} \in \mathbb{R}^2_+ \setminus \partial \mathbb{R}^2_+$ or $\textbf{p} \in \partial \mathbb{R}^2_+$. We will focus only on the latter case as the former case is similar and resembles the interior theory treated in \cite[Lemma 4.2]{Tonegawa2012}. From now on, we take
\[\textbf{p}=\left( \frac{1}{2},0 \right) \in \mathbb{R}^2_+ \qquad \text{ and } \qquad P=\{(y_1,0,z)\; : \; y_1 \geq 0, z \in \mathbb{R}^{n-2}\}  \subset \mathbb{R}^n_+,\]
and we will denote, for any $z' \in \mathbb{R}^{n-2}$, the half-cylinder of size $\frac{\sigma}{2}$ centred at $(\textbf{p},z')$ by
\[ C_{\frac{\sigma}{2}} (\textbf{p},z'):=\left\{(y_1,y_2,z) \; : \; |y_1-\frac{1}{2}|^2+|z-z'|^2 < \left( \frac{\sigma}{2} \right)^2, \, 0 \leq y_2 < \frac{\sigma}{2} \right\} \subset \mathbb{R}^n_+. \]
Moreover, we define the set of exceptional points on $P$ centred at $(\textbf{p},z')$ as follows:
\begin{equation}
    \label{E:exceptional-set}
    S_i(\textbf{p},z'):=\left\{(y_1,0,z) \in P \cap C_{\frac{\sigma}{2}} (\textbf{p},z') \; : \; \exists t' \in \left[-\frac{1}{2},\frac{1}{2}\right] \text{ s.t. } u_{\ep_i}(y_1,y_2,z) \neq t' \; \forall y_2 \in \left[0,\frac{\sigma}{2}\right) \right\}.
\end{equation}

The proof of \eqref{E:step-3-proof} is a contradiction argument following \cite[Lemma 4.2]{Tonegawa2012}, utilizing the results in our previous work \cite[Section 6]{LiPariseSarnataro2023}. Suppose, on the contrary that \eqref{E:step-3-proof} does not hold, then there exists some $z' \in \mathbb{R}^{n-2}$ with $|z'|<1$ such that
\[ \limsup_{i \to \infty} \mathcal{L}^{n-2}(B_{\frac{\sigma}{4}}(z') \cap \mathcal{N}_{u_{\ep_i}})>0.\]
For any $z'' \in B_{\frac{\sigma}{4}}(z') \cap \mathcal{N}_{u_{\ep_i}}$, by the definition of $\mathcal{N}_{u_{\ep_i}}$ and $S_i(\textbf{p},z')$, we must have 
\[ (y_1,0,z'') \in S_i(\textbf{p},z') \quad \text{ provided that } \left|y_1-\frac{1}{2}\right|<\frac{\sigma}{4}.\]
Therefore, we have
\[\mathcal{L}^{n-1}(S_i(\textbf{p},z')) \geq \frac{\sigma}{2} \mathcal{L}^{n-2}(B_{\frac{\sigma}{4}}(z') \cap \mathcal{N}_{u_{\ep_i}})\]
and thus $\limsup_{i \to \infty} \mathcal{L}^{n-1}(S_i(\textbf{p},z')) >0$ by the choice of $z'$. This will immediately lead to a contradiction once we show that
\begin{equation}
\label{E:step-3-proof-main-estimate}
\lim_{i \to \infty} \mathcal{L}^{n-1}(S_i(\textbf{p},z')) =0 \quad \text{ for any $|z'|<1$.}
\end{equation}

We now prove \eqref{E:step-3-proof-main-estimate}. Fix any $|z'|<1$. By Theorem \ref{T:old-main}(2), we can assume that 
\begin{equation}
    \label{E:uniform-converge-to-one}
\text{$u_{\ep_i} \to 1$ locally uniformly on $B_\sigma(\textbf{p},z') \setminus P$.} 
\end{equation}
Take any $(y_1,0,z) \in S_i(\textbf{p},z')$, by definition and \eqref{E:uniform-converge-to-one}, we have for any given $b \in (0,\frac{1}{2})$ one of the following scenarios:
\begin{itemize}
    \item Scenario I: $u_{\ep_i}(y_1,y_2,z) > 1-b$ for all $0 \leq y_2 < \frac{\sigma}{2}$;
    \item Scenario II: $u_{\ep_i}(y_1,y_2,z) \leq 1-b$ for some $0 \leq y_2 < \frac{\sigma}{2}$, in which case we must have
    \begin{equation}
        \label{E:scenario-II-consequence}
    -\frac{1}{2} \leq \inf_{y_2 \in [0,\frac{\sigma}{2})} u_{\ep_i}(y_1,y_2,z) \leq 1-b. 
    \end{equation} 
\end{itemize}
We will estimate the $\mathcal{L}^{n-1}$-measure of those $(y_1,0,z) \in S_i(\textbf{p},z')$ belonging to each of the two scenarios. For our purpose, we denote
\[ S^I_i(\textbf{p},z';b):=\left\{ (y_1,0,z) \in S_i(\textbf{p},z') \; : \; u_{\ep_i}(y_1,y_2,z) > 1-b \text{ for all } 0 \leq y_2 < \frac{\sigma}{2} \right\}, \]
\[ S^{II}_i(\textbf{p},z';b):=\left\{ (y_1,0,z) \in S_i(\textbf{p},z') \; : \; u_{\ep_i}(y_1,y_2,z) \leq 1-b \text{ for some } 0 \leq y_2 < \frac{\sigma}{2} \right\}, \]
where $b \in (0,\frac{1}{2})$ is a parameter to be chosen small enough later. Clearly 
\begin{equation}
    \label{E:decompose-S_i}
S_i(\textbf{p},z')=S^I_i(\textbf{p},z';b) \cup S^{II}_i(\textbf{p},z';b) \qquad \text{for any $b \in (0,\frac{1}{2})$}.
\end{equation}
To estimate $S^I_i(\textbf{p},z';b)$ and $S^{II}_i(\textbf{p},z';b)$, one needs to analyse points around which the level sets of $u_{\ep_i}$ have relatively small (or large) scale-invariant ``tilt-excess''. For our convenience, we introduce the following notation (recall \eqref{E:energy-density-discrepancy}, \eqref{E:mu} and \eqref{E:u-normal} all taken with respect to the metric $g_i$ and domain $\Omega_i$): for any $x \in C_{\frac{\sigma}{2}} (\textbf{p},z')$ and $r \in (0,1)$, we denote the scale-invariant \emph{tilt-excess} of $u_{\ep_i}$ inside the ball $B_r(x)$ with respect to the half-hyperplane $P$ by the non-negative quantity
\begin{equation}
    \label{E:excess}
    \widehat{E}(u_{\ep_i};B_r(x)):= \frac{1}{r^{n-1}} \int_{B_r(x) \cap \Omega_i} |\xi_{\ep_i}(u_{\ep_i})| \; dv_{g_i} +  \big(1-g_i(N_{u_{\ep_i}}, \nabla^{g_i} d_{g_i}^P)^2 \big) \; d\mu_{u_{\ep_i}},
\end{equation}
where $d_{g_i}^P$ is the signed distance function from $P$ with respect to the metric $g_i$. 

We now fix some parameters using the key propositions in \cite[Section 6]{LiPariseSarnataro2023}. For any given $s>0$, by \cite[Proposition 6.3]{LiPariseSarnataro2023}, we can choose $b \in (0,\frac{1}{2})$ sufficiently small such that
\[ (*) \quad \int_{B_2 \cap \{|u_{\ep_i}| \geq 1-b\}} \frac{W(u_{\ep_i})}{\ep_i} \; dv_{g_i} \leq s \qquad \text{for all sufficiently large $i$}. \]
For any such $s$ and $b$, by \cite[Proposition 6.1]{LiPariseSarnataro2023} we can fix $\eta_1 \in (0,1)$ and $L>1$ such that\footnote{Recall that $g_i \to \delta$ and $\xi_{\ep_i}(u_{\ep_i}) \leq \ep_i^{-4/5}$ by Theorem \ref{T:old-main}. Therefore, (ii) of \cite[Proposition 6.1]{LiPariseSarnataro2023} are satisfied for all sufficiently large $i$. Note that there is a typo in (ii) of \cite[Proposition 6.1]{LiPariseSarnataro2023}, $|\xi_{\tilde{\ep}}|$ should be replaced by $\xi_{\tilde{\ep}}$.}
\[ \begin{array}{l}
\text{(**) For $x=(y_1,y_2,z) \in C_{\frac{\sigma}{2}} (\textbf{p},z')$ with $|u_{\ep_i}(x)| \leq 1-b$ and $\widehat{E}(u_{\ep_i};B_{4 \ep_i L}(x))\leq \eta_1$, and for all}\\
\text{sufficiently large $i$, we have $y_2>3\ep_iL$ and $u_{\ep_i}(y_1,y,z) \neq u_{\ep_i}(x)$ for all $0<|y-y_2|<3 \ep_i L$.}\\
\text{Moreover, inside $B_{\ep_i L}(x)$, $u_{\ep_i}$ is $C^1$-close to the unique one-dimensional finite-energy heteroclinic}\\
\text{solution $q_{\ep_i}(y_2)$ satisfying $q_{\ep_i}(y_2)=u_{\ep_i}(x)$ and}
\end{array}
\]
\begin{equation} 
\label{E:**}
1-s \leq \frac{1}{\omega_{n-1}(\ep_i L)^{n-1}} \int_{B_{\ep_i L}(x)} e_{\ep_i}(u_{\ep_i})\, dv_{g_i}  \leq 1+s. 
\end{equation}
Finally, for such $s$ and $b$, by \cite[Proposition 6.2]{LiPariseSarnataro2023} (with $R=\sigma$ and $N=\Theta(P):=\Theta_\bC(\textbf{p},z') \in \mathbb{N}$) we can fix $\eta_2 \in (0,1)$ such that\footnote{Note that the second assumption in (iii) of \cite[Proposition 6.2]{LiPariseSarnataro2023} is satisfied for all sufficiently large $i$ as $r^{-(n-1)}\mu_{\tilde{u}_{\tilde{\ep}_i}}(B_r(x_k))$ is of order $r_i$ by the blow-up process. Similarly, the constant $c_6$ in \cite[Lemma 4.12]{LiPariseSarnataro2023} is of order $r_i$ and the constant $c_7$ in \cite[Proposition 4.13]{LiPariseSarnataro2023} is of order $r_i^{1/8}$ in the blow-up process, thus all the assumptions in (ii) of \cite[Proposition 6.2]{LiPariseSarnataro2023} are satisfied for all sufficiently large $i$ as well.}
\[ \begin{array}{l}
\text{(***) For any collection $\{\hat{y}_k\}_{k=1}^{\Theta(P)+1}$ of distinct positive numbers with $\min_{k}|\hat{y}_{k+1}-\hat{y}_k|>3\ep_iL$}\\
\text{and $3\ep_iL<\hat{y}_1<\hat{y}_2<\cdots<\hat{y}_{\Theta(P)+1} < \eta_2 \sigma$ such that whenever $x_k:=(y_1,\hat{y}_k,z) \in C_{\frac{\sigma}{2}} (\textbf{p},z')$ with}\\
\text{$\widehat{E}(u_{\ep_i};B_r(x_k)) \leq \eta_2$ for all $\ep_iL \leq r \leq \sigma$ and all $k$, we have}\\
\end{array}
\]
\begin{equation} 
\label{E:***} 
\sum_{k=1}^{\Theta(P)+1} \frac{1}{\omega_{n-1} (\ep_i L)^{n-1}} \int_{B_{\ep_i L}(x_k)} e_{\ep_i}(u_{\ep_i})\; dv_{g_i} \leq s + \frac{1+s}{\omega_{n-1} \sigma^{n-1}} \int_{\cup_k B_\sigma(x_k)} e_{\ep_i}(u_{\ep_i})\; dv_{g_i}.  
\end{equation}
To summarize, for any given $s \in (0,1)$, we first fix $b=b(s) \in (0,\frac{1}{2})$ such that (*) is satisfied. Then we fix $\eta=\min\{\eta_1,\eta_2\} \in (0,1)$ and $L>1$ such that both (**) and (***) hold. We can now sub-divide the points in the intermediate level sets inside $C_{\frac{\sigma}{2}} (\textbf{p},z')$:
\[ C_{\frac{\sigma}{2}} (\textbf{p},z') \cap \{|u_{\ep_i}| \leq 1-b\}=\mathcal{G}^b_i(\textbf{p},z') \cup \mathcal{B}^b_i(\textbf{p},z'), \]
where $x \in \mathcal{G}^b_i(\textbf{p},z')$ belongs to the \emph{good set} if the scale-invariant excess satisfies
\[ \widehat{E}(u_{\ep_i};B_r(x)) \leq \eta \qquad \forall r \in [\ep_i L,\sigma],\]
otherwise $x \in \mathcal{B}^b_i(\textbf{p},z')$ belongs to the \emph{bad set}.

We first prove the following lemma which says that the bad set $\mathcal{B}^b_i(\textbf{p},z')$ has small size with respect to both $\mu_{u_{\ep_i}}$ and the $\mathcal{L}^{n-1}$-measure of its orthogonal projection onto $P$. This lemma lies at the heart of the remaining arguments in Step 3.

\begin{lemma}
    \label{L:bad-set-estimate}
    Let $\pi_P$ be the orthogonal projection map (with respect to the flat metric) onto $P$. Then,
    \[ \lim_{i \to \infty} \mu_{u_{\ep_i}}(\mathcal{B}^b_i(\textbf{p},z'))= \lim_{i \to \infty} \mathcal{L}^{n-1}(\pi_P(\mathcal{B}^b_i(\textbf{p},z')))=0. \]
\end{lemma}

\begin{proof}
    By definition the bad set $\mathcal{B}^b_i(\textbf{p},z')$ is covered by a collection of balls $\mathfrak{B}=\{B_{r_x}(x)\}_{x \in \mathcal{B}^b_i(\textbf{p},z')}$ where $r_x \in [ \ep_i L,\sigma]$ with $\widehat{E}(u_{\ep_i};B_{r_x}(x)) > \eta$. By Besicovitch Covering Theorem, we can choose a sub-collection $\cup_{\ell=1}^{N(n)} \mathfrak{B}_\ell$ covering $\mathcal{B}^b_i(\textbf{p},z')$ such that each $\mathfrak{B}_\ell$ contains mutually disjoint balls. Since $\mu_{u_{\ep_i}} \to \Theta(P)\mathcal{H}^{n-1} \mres P$ as Radon measures in $C_{2\sigma} (\textbf{p},z')$ by Theorem \ref{T:old-main} (1) and \eqref{E:blow-up-argument}, for each $B_{r_x}(x) \in \mathfrak{B}_\ell$, by the choice of $r_x$ such that $\widehat{E}(u_{\ep_i};B_{r_x}(x)) > \eta$, we have for all sufficiently large $i$ that
    \[ \mu_{u_{\ep_i}}(B_{r_x}(x)) \leq 2 \Theta(P) r_x^{n-1} \leq 2 \Theta(P) \eta^{-1} \int_{B_{r_x}(x) \cap \Omega_i} |\xi_{\ep_i}(u_{\ep_i})| \; dv_{g_i} +  \big(1-g_i(N_{u_{\ep_i}}, \nabla^{g_i} d_{g_i}^P)^2 \big) \; d\mu_{u_{\ep_i}}.  \]
    Summing over all the (disjoint!) balls in $\mathfrak{B}_\ell$, we have
    \[ \mu_{u_{\ep_i}}(\mathfrak{B}_\ell) \leq 2 \Theta(P) \eta^{-1} \int_{\mathfrak{B}_\ell \cap \Omega_i} |\xi_{\ep_i}(u_{\ep_i})| \; dv_{g_i} +  \big(1-g_i(N_{u_{\ep_i}}, \nabla^{g_i} d_{g_i}^P)^2 \big) \; d\mu_{u_{\ep_i}} \]
    which goes to zero as $i \to \infty$ by Theorem \ref{T:old-main} (1) and \eqref{E:blow-up-argument} again. This shows that 
    \[  \lim_{i \to \infty} \mu_{u_{\ep_i}}(\mathcal{B}^b_i(\textbf{p},z')) \leq N(n) \lim_{i \to \infty} \mu_{u_{\ep_i}}(\mathfrak{B}_\ell)=0. \]
    Similarly, by considering the collection of projected balls\footnote{Note that the projected balls of $\mathfrak{B}_\ell$ may not be disjoint in $P$, but the same argument still goes through since the balls in $\mathfrak{B}_\ell$ are disjoint before projection.} of $\cup_{\ell=1}^{N(n)} \mathfrak{B}_\ell$ by $\pi_P$ which covers the projected bad set $\pi_P(\mathcal{B}^b_i(\textbf{p},z'))$, we obtain $\lim_{i \to \infty} \mathcal{L}^{n-1}(\pi_P(\mathcal{B}^b_i(\textbf{p},z')))=0$.
\end{proof}

Next, we proceed to estimate in the next two lemmas the size of the exceptional set $S_i(\textbf{p},z')$ as defined in \eqref{E:exceptional-set} based on the decomposition \eqref{E:decompose-S_i}.

\begin{lemma}
\label{L:second-scenario}
$\lim_{i \to \infty} \mathcal{L}^{n-1}(S^{II}_i(\textbf{p},z';b))=0$.
\end{lemma}

\begin{proof}
    The assertion follows from Lemma \ref{L:bad-set-estimate} once we show that
    \begin{equation}
        \label{E:lemma-II}
        S^{II}_i(\textbf{p},z';b) \subset \pi_P(\mathcal{B}^b_i(\textbf{p},z')).
    \end{equation}
    To prove \eqref{E:lemma-II}, take any $(y_1,0,z) \in S^{II}_i(\textbf{p},z';b)$, by definition there exists some $y_2' \in [0,\frac{\sigma}{2})$ such that 
    \[ -\frac{1}{2} \leq u_{\ep_i}(y_1,y_2',z) \leq 1-b.\]
    Therefore, $(y_1,y_2',z) \in C_{\frac{\sigma}{2}} (\textbf{p},z') \cap \{|u_{\ep_i}| \leq 1-b\}$. We have two different cases to consider:
    \begin{itemize}
        \item $(y_1,y_2',z) \in \mathcal{G}^b_i(\textbf{p},z')$, in which case $x=(y_1,y_2',z)$ satisfies $\widehat{E}(u_{\ep_i};B_r(x)) \leq \eta$ for all $r \in [\ep_i L,\sigma]$ by definition. By (**), $y_2' \geq 3 \ep_iL$ and inside $B_{\ep_i L}(x)$, $u_{\ep_i}$ is $C^1$-close to the one-dimensional heteroclinic solution $q_{\ep_i}(y_2)$ with $q_{\ep_i}(y_2')=u_{\ep_i}(x)$. By enlarging $L$ depending on $W$ if necessary, we can arrange $\inf_{B_{\ep_i L}(x)} u_{\ep_i} < -\frac{1}{2}$, contradicting \eqref{E:scenario-II-consequence}. Therefore, we have excluded this case.
        \item $(y_1,y_2',z) \in \mathcal{B}^b_i(\textbf{p},z')$, hence $(y_1,0,z) \in \pi_P(\mathcal{B}^b_i(\textbf{p},z'))$ which is our desired conclusion.
    \end{itemize}
\end{proof}

\begin{lemma}
    \label{L:first-scenario}
    There exists $s_0=s_0(\Theta(P)) \in (0,1)$ small enough such that for any $s \in (0,s_0)$, we have
    \[ \limsup_{i \to \infty} \mathcal{L}^{n-1}(S^I_i(\textbf{p},z';b)) \leq C(n) s\]
\end{lemma}

\begin{proof}
    First of all, by definitions $S^I_i(\textbf{p},z';b)$ and $\pi_P(\mathcal{G}^b_i(\textbf{p},z')$ are disjoint subsets of $P \cap C_{\frac{\sigma}{2}} (\textbf{p},z')$. Therefore, we have
    \begin{equation} 
    \label{E:step-three-proof-a}
    \mathcal{L}^{n-1}(S^I_i(\textbf{p},z';b)) \leq  \omega_{n-1} \left(\frac{\sigma}{2}\right)^{n-1} - \mathcal{L}^{n-1}(\pi_P(\mathcal{G}^b_i(\textbf{p},z'))).
    \end{equation}
    Hence, we need to estimate the second term on the right hand side above. Fix any $|t| \leq 1-b$. Take any $(y_1,0,z) \in \pi_P(\mathcal{G}^b_i(\textbf{p},z'))$, and denote $\{\hat{y}_k\}_{k=1}^K$ to be distinct non-negative numbers such that $\hat{y}_k \leq \eta \sigma$, $x_k=(y_1,\hat{y}_k,z) \in \mathcal{G}^b_i(\textbf{p},z')$ and $u_{\ep_i}(x_k)=t$ for each $k=1,\cdots,K$. We claim that a-priori one must have $K \leq \Theta(P)$. Otherwise, we can take $K=\Theta(P)+1$ and then derive a contradiction below. Roughly speaking, the contradiction comes from the fact that around every $x_k$ in the good set $\mathcal{G}^b_i(\textbf{p},z')$, the smallness of excess at all scales above $O(\ep_i)$ implies good sheet separation and locally at $O(\ep_i)$-scale $u_{\ep_i}$ resembles the one-dimensional heteroclinic solution $q_{\ep_i}(y_2)$. Therefore, having $K \geq \Theta(P)+1$ would contradict the varifold convergence in \eqref{E:blow-up-argument}. 

    Suppose to the contrary that we can take $K=\Theta(P)+1$. Since each $x_k \in \mathcal{G}^b_i(\textbf{p},z')$, (**) implies immediately that all the assumptions in (***) are satisfied for $\{\hat{y}_k\}_{k=1}^{\Theta(P)+1}$. Combining \eqref{E:**} and \eqref{E:***}, taking $i \to \infty$, we obtain
    \[ (\Theta(P)+1)(1-s) \leq s + \frac{1+s}{\omega_{n-1}\sigma^{n-1}} \|\bC\|\left(\bigcup_{k=1}^{\Theta(P)+1} \overline{B_\sigma(x_k)} \right) \leq s+ (1+s) \Theta(P) \]
    which will result in a contradiction provided that $s < \frac{1}{2(\Theta(P)+1)}$.

    To finish the proof of the lemma, we perform a standard change of variable
    \[ w_{u_{\ep_i}} := \Phi \circ u_{\ep_i} \qquad \text{ where } \Phi(t):=\int_0^t \sqrt{2W(\tau)} \; d\tau.\]
    Note that $\nabla w_{\ep_i}=\sqrt{2W(u_{\ep_i})} \, \nabla u_{\ep_i}$ and (recall \eqref{E:energy-density-discrepancy})
    \[ \left| e_{\ep_i}(u_{\ep_i}) - \frac{|\nabla w_{u_{\ep_i}}|}{h_0} \right| \leq |\xi_{\ep_i}(u_{\ep_i})|.\]
    By Theorem \ref{T:old-main} (1), we have 
    \begin{equation} 
    \label{E:step-three-last-equation}
    \lim_{i \to \infty} \int_{C_{\frac{\sigma}{2}} (\textbf{p},z')} \frac{|\nabla w_{u_{\ep_i}}|}{h_0} \; dv_{g_i}= \|\bC\|(C_{\frac{\sigma}{2}} (\textbf{p},z'))= \Theta(P) \omega_{n-1} \left( \frac{\sigma}{2} \right)^{n-1}. 
    \end{equation}
    On the other hand, using (*), Lemma \ref{L:bad-set-estimate}, the co-area formula and \eqref{E:energy-constant}, 
    \begin{eqnarray*}
    \lim_{i \to \infty} \int_{C_{\frac{\sigma}{2}} (\textbf{p},z')} \frac{|\nabla w_{u_{\ep_i}}|}{h_0} \; dv_{g_i} & \leq & s + \liminf_{i \to \infty} \int_{\mathcal{G}^b_i(\textbf{p},z')} \frac{|\nabla w_{u_{\ep_i}}|}{h_0} \; dv_{g_i}  \\
    & \leq & s + \liminf_{i \to \infty} \int_{-1+b}^{1-b} \Hcal^{n-1}(\{u_{\ep_i}=t\} \cap \mathcal{G}^b_i(\textbf{p},z')) \frac{\sqrt{2W(t)}}{h_0} \; dt \\
    & \leq & s+ (1+s) \Theta(P) \liminf_{i \to \infty} \mathcal{L}^{n-1}(\pi_P(\mathcal{G}^b_i(\textbf{p},z'))).
    \end{eqnarray*}
    The last inequality is due to the observation (c.f. \cite[(4.11)]{Tonegawa2012}) that around any $x \in \mathcal{G}^b_i(\textbf{p},z')$ which lies on a regular level set of $u_{\ep_i}$, since $u_{\ep_i}$ is $C^1$-close to the one-dimensional heteroclinic solution $q_{\ep_i}(y_2)$ inside $B_{\ep_i L}(x)$, the level set $\{u_{\ep_i}=u_{\ep_i}(x)\}$ can be written as the graph of a function over $\pi_P(B_{\ep_i L}(x))$ with gradient approaching zero less than  as $i \to \infty$. Combining the above with \eqref{E:step-three-last-equation}, and taking limsup on both sides of \eqref{E:step-three-proof-a}, we obtain
    \[  \limsup_{i \to \infty} \mathcal{L}^{n-1}(S^I_i(\textbf{p},z';b)) \leq \left( \omega_{n-1}\left( \frac{\sigma}{2} \right)^{n-1} + \frac{1}{\Theta(P)} \right) s, \]
    which proves the assertion in the lemma.
\end{proof}

By Lemma \ref{L:second-scenario} and \ref{L:first-scenario}, together with \eqref{E:decompose-S_i}, we have
\[\lim_{i \to \infty} \mathcal{L}^{n-1}(S_i(\textbf{p},z')) \leq C(n)s \qquad \text{for all $s \in (0,s_0)$,} \]
hence we obtain the desired conclusion \eqref{E:step-3-proof-main-estimate}.        

\vspace{1ex}

\textit{Step 4: Controlling in an integral sense the total turning angle of the cross sections of the intermediate level sets of $u_{\ep}$ and the final contradiction argument.}

\vspace{1ex}

In what follows, we will use $B^k_r$ to denote the Euclidean ball of radius $r>0$ centred at the origin. Let $\pi_{S(\bC)}$ be the orthogonal projection map (with respect to the flat metric) onto the spine $S(\bC) \cong \mathbb{R}^{n-2}$. We first establish a lemma saying that generic level sets of $u_{\ep_i}$ satisfy some nice properties.

\begin{lemma}
    \label{L:nice-level-sets-are-generic}
    For $\mathcal{L}^1$-a.e. $t \in (-\frac{1}{2},\frac{1}{2})$, the following statements hold:
    \begin{itemize}
        \item[(1)] For every $i \in \mathbb{N}$, $t$ is a regular value for $u_{\ep_i}$ restricted on the open\footnote{Note that each $\mathcal{I}_{u_{\ep_i}}$ is a relatively closed subset of $B^{n-2}_1$ by definition.} set $B^2_1 \times (B^{n-2}_1\setminus \mathcal{I}_{u_{\ep_i}})$, thus the level set $\{u_{\ep_i}=t\}$ is a $C^3$ hypersurface in $\Omega_i$ with free boundary on $\partial \Omega_i$. 
        \item[(2)] Over the ``nice part'' $\mathfrak{N}_i \subset B^2_1 \times B^{n-2}_1 \subset B^n_2$ defined by
        \[ \mathfrak{N}_i:=\pi_{S(\bC)}^{-1}(B^{n-2}_1\setminus (\mathcal{I}_{u_{\ep_i}} \cup \mathcal{N}_{u_{\ep_i}})), \]
        the level hypersurfaces $\{u_{\ep_i}=t\}$ have finite area in the liminf sense, i.e.
        \begin{equation}
            \label{E:finite-area-level-set}
        \liminf_{i \to \infty} \Hcal^{n-1}_{g_i}(\{u_{\ep_i} = t \} \cap \mathfrak{N}_i) < \infty;
        \end{equation} 
        and their enhanced second fundamental form are converging to zero in the $L^2$-sense, i.e.
        \begin{equation}
            \label{E:2ff-converges-to-zero}
            \liminf_{i \to \infty} \int_{\{u_{\ep_i}=t\} \cap \mathfrak{N}_i} |\Acal_{u_{\ep_i}}|^2 \; d\Hcal^{n-1}_{g_i} = 0.
        \end{equation}
    \end{itemize}
\end{lemma}

\begin{proof}
(1) is a direct consequence of Sard's theorem. To prove (2), first note that by the definition of $\mathcal{I}_{u_{\ep_i}}$ (recall \eqref{E:step-2-proof}), for any $z \in B^{n-2}_1\setminus \mathcal{I}_{u_{\ep_i}}$, there exists $r_z>0$ such that
\begin{equation} 
\label{E:non-degeneracy}
|\nabla u_{\ep_i}| \geq \frac{c}{\ep_i} \qquad \text{ on $B^2_1 \times B^{n-2}_{r_z}(z) \cap \left\{|u_{\ep_i}|\leq \frac{1}{2}\right\} $.}
\end{equation}
Using co-area formula and \eqref{E:non-degeneracy}, 
\[ \int_{-\frac{1}{2}}^{\frac{1}{2}} \Hcal^{n-1}_{g_i}(\{u_{\ep_i} = t \} \cap \mathfrak{N}_i) \; dt = \int_{\left\{|u_{\ep_i}|\leq \frac{1}{2}\right\} \cap \mathfrak{N}_i} |\nabla u_{\ep_i}|\; dv_{g_i} \leq \frac{h_0}{c} \mu_{u_{\ep_i}}(B^n_2), \]
which implies \eqref{E:finite-area-level-set} by letting $i \to \infty$ using Fatou's lemma. For \eqref{E:2ff-converges-to-zero}, using again co-area formula and \eqref{E:non-degeneracy}, 
\[ \int_{-\frac{1}{2}}^{\frac{1}{2}} \int_{\{u_{\ep_i}=t\} \cap \mathfrak{N}_i} |\Acal_{u_{\ep_i}}|^2 \; d\Hcal^{n-1}_{g_i} \; dt= \int_{\left\{|u_{\ep_i}|\leq \frac{1}{2}\right\} \cap \mathfrak{N}_i} |\Acal_{u_{\ep_i}}|^2 |\nabla u_{\ep_i}| \; dv_{g_i} \leq \frac{h_0}{c} \int_{B^n_2}|\Acal_{u_{\ep_i}}|^2 \; d\mu_{u_{\ep_i}}.\]
The desired assertion follows by taking $i \to \infty$ and using Fatou's lemma together with \eqref{E:blow-up-argument}.
\end{proof}

Applying an elementary geometric lemma \cite[Lemma 4.3]{Tonegawa2012}\footnote{The proof there works verbatim in the Riemannian setting with boundary.}, Lemma \ref{L:nice-level-sets-are-generic} with another use of Fatou's lemma implies that for $\mathcal{L}^1$-a.e. $t \in (-\frac{1}{2},\frac{1}{2})$ such that (1) and (2) in Lemma \ref{L:nice-level-sets-are-generic} holds, 
\[ \liminf_{i \to \infty} \int_{\ell_z} |\kappa^i_{z}|\; ds_{g_i} = 0  \]
for $\mathcal{L}^{n-2}$-a.e. $z \in B^{n-2}_1\setminus (\mathcal{I}_{u_{\ep_i}} \cup \mathcal{N}_{u_{\ep_i}})$ which is also a regular value of $\pi_{S(\bC)}$ restricted on the level hypersurface $\{u_{\ep_i}=t\}$ in $B^2_1 \times (B^{n-2}_1\setminus \mathcal{I}_{u_{\ep_i}})$. Here, $\ell^i_z:=\pi_{S(\bC)}^{-1}(z) \cap \{u_{\ep_i}=t\}$ is a $C^3$ planar curve contained in $B^2_1 \times \{z\}$, and $\kappa^i_z$ is the geodesic curvature of the planar curve $\ell^i_z$ with respect to $g_i$. Since $\lim_{i \to \infty} \mathcal{L}^{n-2}(\mathcal{I}_{u_{\ep_i}} \cup \mathcal{N}_{u_{\ep_i}})= 0$ by \eqref{E:step-2-proof} and \eqref{E:step-3-proof}, we can choose a sequence of $z_i \in  B^{n-2}_1\setminus (\mathcal{I}_{u_{\ep_i}} \cup \mathcal{N}_{u_{\ep_i}})$ which is a regular value of $\pi_{S(\bC)}$ restricted on the level hypersurface $\{u_{\ep_i}=t\}$ in $B^2_1 \times (B^{n-2}_1\setminus \mathcal{I}_{u_{\ep_i}})$ and, possibly after passing to a subsequence, 
\begin{equation}
    \label{E:final-contradiction}
\lim_{i \to \infty} \int_{\ell^i_{z_i}} |\kappa^i_{z_i}|\; ds_{g_i} = 0. 
\end{equation} 

Finally, we describe the contradiction arguments. Note that each $\ell^i_{z_i}$ is a non-empty union of disjoint, connected, properly embedded planar curves in $\Omega_i \cap (B^2_1 \times \{z_i\})$ with (possibly empty) free boundary in $\partial \Omega_i \cap (B^2_1 \times \{z_i\})$. Moreover, by the definition of $\mathcal{N}_{u_{\ep_i}}$, $\ell^i_{z_i}$ must intersect every one of the disjoint disks $B^2_{\sigma}(\textbf{p}_j) \times \{z_i\}$ for all $j=1,\cdots,N$. Since $N \geq 2$ and $\spt\|\bC\| \neq \partial \mathbb{R}^n_+$, for all sufficiently large $i$ (recall \eqref{E:ambient-convergence}), we have one of the following cases to consider: (in what follows, we call a component ``unbounded'' if it is not contained in a compact subset of $B^2_1 \times \{z_i\}$)
\begin{itemize}
    \item there exists a closed (i.e. without boundary) component of $\ell^i_{z_i}$, which would result in $\int_{\ell^i_{z_i}} |\kappa^i_{z_i}|\; ds_{g_i} \geq 1.99\pi$.
    \item there exists a component of $\ell^i_{z_i}$ with two free boundary points at $\partial \Omega_i \cap (B^2_1 \times \{z_i\})$, which would result in $\int_{\ell^i_{z_i}} |\kappa^i_{z_i}|\; ds_{g_i} \geq 0.99\pi$.
    \item there exists a component of $\ell^i_{z_i}$ that intersects both $B^2_{\sigma}(\textbf{p}_{j_1}) \times \{z_i\}$ and $B^2_{\sigma}(\textbf{p}_{j_2}) \times \{z_i\}$ for some $j_1\neq j_2$ such that one of the $\textbf{p}_{j_1}$ or $\textbf{p}_{j_2}$ does not lie on $\partial \mathbb{R}^2_+$, by the Hausdorff convergence of the level sets in Theorem \ref{T:old-main} (2), we must have 
    \[ \int_{\ell^i_{z_i}} |\kappa^i_{z_i}|\; ds_{g_i} \geq c(\textbf{p}_{j_1},\textbf{p}_{j_2},\sigma)>0.\]
    \item there exists an unbounded component of $\ell^i_{z_i}$ that intersects only one $B^2_{\sigma}(\textbf{p}_{j}) \times \{z_i\}$ for some $\textbf{p}_j \neq (0,\frac{1}{2})$ which has exactly one free boundary point at $\partial \Omega_i \cap (B^2_1 \times \{z_i\})$, in this case we must have
    \[ \int_{\ell^i_{z_i}} |\kappa^i_{z_i}|\; ds_{g_i} \geq c(\textbf{p}_j,\sigma)>0.\]
    \item there exists an unbounded component of $\ell^i_{z_i}$ that intersects only one $B^2_{\sigma}(\textbf{p}_{j}) \times \{z_i\}$ without boundary points, which would again result in $\int_{\ell^i_{z_i}} |\kappa^i_{z_i}|\; ds_{g_i} \geq 0.99\pi$.
\end{itemize}
Therefore, in all of the above cases, we always have (recall that $\sigma$ only depends on $\bC$ as well)
\[ \int_{\ell^i_{z_i}} |\kappa^i_{z_i}|\; ds_{g_i} \geq c(\bC)>0\]
which clearly contradicts \eqref{E:final-contradiction}. This finishes the proof of excluding boundary classical tangent cones of Type I. The proof for Type II is very similar and the details are left to the reader.
\end{proof}

\appendix

\section{Monotonicity of reflected density}
\label{A:reflected-density}

In this appendix, we use the reflected ball arguments as in \cite{Tonegawa2003} to establish a monotonicity-type formula for the ``reflected'' rescaled energy up to the boundary. Besides giving an alternative approach to the monotonicity formula in \cite{LiPariseSarnataro2023}, the reflected ball arguments have the advantage that the balls can intersect the boundary without being centred at a boundary point. This yields up-to-the-boundary Hausdorff convergence of the intermediate level sets $\{|u_\varepsilon|\leq 1-b\}$ to the support of the limit-interface (see Theorem \ref{T:old-main} (2) and Remark \ref{R:up-to-boundary-Hausdorff-convergence}). 


Throughout this appendix, we assume that $u_\varepsilon\in C^3(\overline\Omega\cap U)$ solves the following inhomogeneous Allen-Cahn equation as in \cite{LiPariseSarnataro2023}:
\begin{equation}
\label{equation: EL}
\begin{cases}
-\varepsilon\Delta_g u_\varepsilon+\dfrac{W'(u_\varepsilon)}{\varepsilon}=\lambda_\varepsilon&\text{in }\Omega\cap U,\\[4pt]
\partial u_\varepsilon/\partial \nu=0&\text{on }\partial\Omega\cap U,
\end{cases}
\end{equation}
where $\varepsilon\in(0,1)$, with (A1)-(A2) in Assumptions \ref{As:stationary-AC} and
\[
|\lambda_\varepsilon|\leq\Lambda_0.
\]

\begin{remark}
We will mainly focus on the homogeneous case $\Lambda_0=0$. When $\Lambda_0\neq0$, we can consider as in \cite{LiPariseSarnataro2023} the following modified (non-negative) potential
\[
\widetilde W_\varepsilon(s)=W(s)-\varepsilon\lambda_\varepsilon s+\varepsilon C_0\Lambda_0.
\]
and the following modified normalized energy density and discrepancy function (recall \eqref{E:energy-density-discrepancy}):
\begin{equation}
\label{E:modified-densities}
\widetilde e_\varepsilon (u_\varepsilon):=e_\varepsilon (u_\varepsilon)+f_\varepsilon(u_\varepsilon),
\qquad
\widetilde\xi_\varepsilon (u_\varepsilon):=\xi_\varepsilon (u_\varepsilon)-f_\varepsilon (u_\varepsilon).
\end{equation}
where $f_\varepsilon(u_\varepsilon):=h_0^{-1}(-\lambda_\varepsilon u_\varepsilon+C_0\Lambda_0)$ satisfies $0\leq f_\varepsilon\leq2h_0^{-1}C_0\Lambda_0$. Thus the same computations below apply to $\widetilde e_\varepsilon$ and $\widetilde\xi_\varepsilon$ in place of $e_\varepsilon$ and $\xi_\varepsilon$.  Note that the associated varifold $V^{u_\ep}$ defined in \eqref{E:varifold-def} is independent of the potential term and hence the limit-interface $V$ would be the same even if we consider the modified quantities as above.
\end{remark}

In what follows, we let $V=\lim_{i \to \infty} V^{u_{\ep_i}}$ be a limit-interface as obtained in Theorem \ref{T:old-main} for some $\ep_i \to 0$. We will write $u_i=u_{\varepsilon_i}$ for simplicity. The goal of this appendix is to establish the following (c.f. Theorem \ref{T:old-main} (2) and Remark \ref{R:up-to-boundary-Hausdorff-convergence}):
\begin{proposition}
\label{P:reflected-Hausdorff}
For any $b\in(0,1)$, the subsets $\{|u_i|\leq1-b\}$ locally converge in the Hausdorff sense to $\spt\|V\|$ in $\overline\Omega\cap U$.
\end{proposition}

To proceed with the reflected balls argument, we recall the Fermi coordinates with respect to $\partial\Omega$. Fix $p\in\partial\Omega$, and let $(y_1,\ldots,y_{n-1})$ be geodesic normal coordinates on a neighbourhood $\mathcal O$ of $p$ in $\partial\Omega$. We parametrise an open subset of the two-sided tubular neighbourhood $U_{\frac{1}{2\kappa_0}}(\partial \Omega)$ by
\[
\Phi:\mathcal O\times \left(-\frac{1}{2\kappa_0},\frac{1}{2\kappa_0} \right)\longrightarrow\widetilde M,
\qquad \Phi(y,t):=\exp_y(-t\nu(y)).
\]
Here $t$ is geometrically the signed boundary distance function (recall \eqref{E:dist-estimates}), and $t>0$ corresponds to the interior of $\Omega$. In these coordinates, the metric takes the form $g=dt^2+g_t$, with $g_t=g_0+2t\sff_0+O(t^2),$ where $\sff_0$ is the second fundamental form with respect to the inward unit normal $\partial_t=\nabla d$. We define reflection across $\partial\Omega$ by
\[
\mathcal R(y,t):=(y,-t).
\]
Then $\mathcal R|_{\partial\Omega}=\mathrm{Id}$ and $\mathcal R^2=\mathrm{Id}$ (i.e. $\mathcal{R}$ is an involution). Its differential is exactly $\operatorname{diag}(I_{n-1},-1)$, but reflection is not generally an isometry. More precisely,
\begin{equation}
\label{E:reflection-distortion}
(D\mathcal R_x)^*D\mathcal R_x=\mathrm{Id}+O(|t|),
\end{equation}
where the adjoint is taken with respect to the metrics at $x$ and $\mathcal R(x)$.

Fix a point $a\in\overline\Omega \cap U_{\frac{1}{4\kappa_0}}(\partial \Omega)$ which will be our centre, and set
\[
r(x):=\dist_g(x,a),\qquad
\widetilde r(x):=r(\mathcal R(x)),\qquad
\widetilde B_\rho(a):=\Omega \cap U_{\frac{1}{4\kappa_0}}(\partial \Omega) \cap \{\widetilde r(x)<\rho\}
\]
where $\rho \in (0, \frac{1}{4\kappa_0})$ will be taken small enough so that all the functions and vector fields are well-defined and their cutoffs are compactly supported in $U$. Introduce
\[
\zeta(x):=\frac12\nabla r^2=r\nabla r =-\exp_x^{-1}(a),\qquad
A(x):=(D\mathcal R_x)^{-1},\qquad
Y(x):=A(x)\zeta(\mathcal R(x)).
\]
Fix $\varphi\in C_c^\infty([0,\infty))$ with $0\leq\varphi\leq1$, $\varphi'\leq0$, $\varphi\equiv1$ on $[0,1]$, and $\varphi\equiv0$ on $[2,\infty)$. For a small scale $\rho>0$, define
\[
X(x):=\varphi(r(x)/\rho)\zeta(x)+\varphi(\widetilde r(x)/\rho)Y(x)=X_1(x)+X_2(x).
\]
This vector field is compactly supported in $U$, and $|X|\leq C\rho$. Along $\partial\Omega$, the map $D\mathcal R$ is orthogonal reflection across $T\partial\Omega$. Since $r=\widetilde r$ there, the normal components of $X_1$ and $X_2$ cancel, so that $X\cdot\nu=0$.

As in \cite[Section 4]{LiPariseSarnataro2023}, standard computations give $\nabla\zeta=\mathrm{Id}+O(r)$ and $\operatorname{div}\zeta=n+O(r)$. Hence
\[
\operatorname{div}X_1=n\varphi(r/\rho)+\frac r\rho\varphi'(r/\rho)+O(r)\varphi(r/\rho),
\]
and, writing $\nabla X(v,v)=\langle\nabla_vX,v\rangle$, we obtain
\[
\nabla X_1(\nabla u,\nabla u)
=\varphi(r/\rho)|\nabla u|^2
+\frac r\rho\varphi'(r/\rho)(\partial_ru)^2
+O(r)\varphi(r/\rho)|\nabla u|^2.
\]
For the reflected term, differentiating $Y=A(\zeta\circ\mathcal R)$ gives
\[
\nabla Y=(\nabla A)(\zeta\circ\mathcal R)
+A\circ(\nabla\zeta)_{\mathcal R(x)}\circ D\mathcal R_x
=\mathrm{Id}+O(\widetilde r+t).
\]
Here $\nabla A=O(1)$ on the fixed tubular neighbourhood, while the second term is $\mathrm{Id}+O(\widetilde r)$. Moreover, from $d\widetilde r(Y)=\widetilde r$ we deduce
\[
\operatorname{div}X_2=n\varphi(\widetilde r/\rho)
+\frac{\widetilde r}{\rho}\varphi'(\widetilde r/\rho)
+O(\widetilde r+t)\varphi(\widetilde r/\rho).
\]
Set $\eta=Y/\widetilde r$ away from $\widetilde r=0$. The cutoff term in $\nabla X_2(\nabla u,\nabla u)$ is
\[
\frac{\widetilde r}{\rho}\varphi'(\widetilde r/\rho)
\langle\nabla\widetilde r,\nabla u\rangle\langle\eta,\nabla u\rangle.
\]
By \eqref{E:reflection-distortion},
\[
\nabla\widetilde r=(D\mathcal R_x)^*D\mathcal R_x\eta=\eta+O(t).
\]
Consequently, replacing this product by $\langle\eta,\nabla u\rangle^2$ introduces an error bounded by
\[
Ct\frac{\widetilde r}{\rho}|\varphi'(\widetilde r/\rho)|\,|\nabla u|^2.
\]

Write
\[
S_\rho:=\varphi(r/\rho)+\varphi(\widetilde r/\rho),
\qquad \text{and} \qquad T_\rho:=-\frac{\widetilde r}{\rho}\varphi'(\widetilde r/\rho)\geq0.
\]
On the support of the first cutoff, $r\leq2\rho$; on the support of the reflected cutoff, $0\leq t\leq\widetilde r\leq2\rho$. Thus there exist a function $q_1$ and a tensor $Q_2$, depending on $\rho$ and $\varphi$, such that
\begin{align*}
 \operatorname{div}X&=nS_\rho+\frac r\rho\varphi'(r/\rho)
+\frac{\widetilde r}{\rho}\varphi'(\widetilde r/\rho)+q_1,\\
\nabla X(\nabla u,\nabla u)& =S_\rho|\nabla u|^2+\frac r\rho\varphi'(r/\rho)(\partial_ru)^2+\frac{\widetilde r}{\rho}\varphi'(\widetilde r/\rho)
\langle\eta,\nabla u\rangle^2+Q_2(\nabla u,\nabla u),
\end{align*}
with the cutoff-independent bounds
\begin{equation}
\label{E:reflection-errors}
|q_1|\leq C\rho S_\rho,
\qquad |Q_2|\leq C\rho(S_\rho+T_\rho).
\end{equation}

Recall now the Pohozaev identity from \cite[Section 4]{LiPariseSarnataro2023} (note that we have a different normalization constant in \eqref{E:energy-density-discrepancy}). Since $X\cdot\nu=0$ and $\frac{\partial u}{\partial \nu}=0$, for $\lambda_\varepsilon=0$ it reads
\[
\int_{\Omega\cap U}
\bigl(e_\varepsilon\operatorname{div}X-\varepsilon h_0^{-1} \nabla X(\nabla u,\nabla u)\bigr)\, \dvol_g=0.
\]
Using $ne_\varepsilon-\varepsilon h_0^{-1}|\nabla u|^2=(n-1)e_\varepsilon-\xi_\varepsilon$, we obtain
\begin{align*}
0&=\int_\Omega\bigl((n-1)e_\varepsilon-\xi_\varepsilon\bigr)S_\rho\,\dvol_g+\int_\Omega e_\varepsilon\left(\frac r\rho\varphi'(r/\rho)
+\frac{\widetilde r}{\rho}\varphi'(\widetilde r/\rho)\right)\,\dvol_g\\
&\quad-h_0^{-1}\int_\Omega\varepsilon \left(\frac r\rho\varphi'(r/\rho)(\partial_ru)^2 +\frac{\widetilde r}{\rho}\varphi'(\widetilde r/\rho)\langle\eta,\nabla u\rangle^2\right)\,\dvol_g+ h_0^{-1}\int_\Omega\bigl(q_1e_\varepsilon-\varepsilon Q_2(\nabla u,\nabla u)\bigr)\,\dvol_g.
\end{align*}
All integrands here and below are supported in $U$. Define
\begin{align*}
I(\rho)&:=\int_\Omega e_\varepsilon S_\rho\,\dvol_g, \qquad J_1(\rho):=\int_\Omega(-\xi_\varepsilon)S_\rho\,\dvol_g,\\
J_2(\rho)&:=h_0^{-1}\int_\Omega\varepsilon\left[\varphi(r/\rho)(\partial_ru)^2
+\varphi(\widetilde r/\rho)\langle\eta,\nabla u\rangle^2\right]\,\dvol_g,\\
J_3(\rho)&:=h_0^{-1}\int_\Omega\bigl(q_1e_\varepsilon-\varepsilon Q_2(\nabla u,\nabla u)\bigr)\,\dvol_g.
\end{align*}
Using $\rho\,\partial_\rho\varphi(s/\rho)=-(s/\rho)\varphi'(s/\rho)$, the Pohozaev identity becomes
\begin{equation}
\label{equation: itermediate PI}
(n-1)I(\rho)-\rho I'(\rho)=-J_1(\rho)-\rho J_2'(\rho)-J_3(\rho).
\end{equation}
This is the argument of \cite{Tonegawa2003}; the error term $J_3$ accounts for curvature and the fact that $\mathcal R$ is not an isometry. Since $\varphi'\leq0$, we have $I'\geq0$ and $J_2'\geq0$. Moreover, $|J_1|\leq I$, and \eqref{E:reflection-errors} gives $|J_3(\rho)|\leq C\rho I(\rho)+C\rho^2I'(\rho).$ Indeed, $\varepsilon h_0^{-1}|\nabla u|^2\leq2e_\varepsilon$ and $\int_\Omega e_\varepsilon T_\rho\,\dvol_g\leq\rho I'(\rho)$. It follows from \eqref{equation: itermediate PI} that $(1+C\rho)\rho I'(\rho)\geq(n-1)I(\rho)+J_1(\rho)-C\rho I(\rho).$
Dividing by $1+C\rho$ and using $|J_1|\leq I$, we may enlarge $C$ to obtain $\rho I'(\rho)\geq(n-1)I(\rho)+J_1(\rho)-C\rho I(\rho).$ Consequently,
\begin{equation}
\label{E:smooth-reflected-mono}
\frac{d}{d\rho}\bigl(e^{C\rho}\rho^{1-n}I(\rho)\bigr)
\geq e^{C\rho}\rho^{-n}J_1(\rho). 
\end{equation}

Choose such cutoffs $\varphi_j$ decreasing pointwise to the characteristic function $\mathbf1_{[0,1]}$. By dominated convergence,
\begin{align*}
I_j(\rho)&\rightarrow\mathcal E_\varepsilon(\rho,a)
:=\int_{B_\rho(a)\cap\Omega}e_\varepsilon\,\dvol_g
+\int_{\widetilde B_\rho(a)}e_\varepsilon\,\dvol_g,\\
J_{1,j}(\rho)&\rightarrow D_\varepsilon(\rho,a)
:=\int_{B_\rho(a)\cap\Omega}(-\xi_\varepsilon)\,\dvol_g
+\int_{\widetilde B_\rho(a)}(-\xi_\varepsilon)\,\dvol_g.
\end{align*}
We first integrate \eqref{E:smooth-reflected-mono} over a compact interval of admissible positive radii and then pass to the limit. This yields the integrated inequality for $\mathcal E_\varepsilon$ and $D_\varepsilon$. Since the ball integrals are locally absolutely continuous in the radius, the corresponding differential inequality holds for almost every radius.

We have thus proved the following monotonicity-type formula (the interior case follows from \cite[Section 4.1]{LiPariseSarnataro2023}). 
\begin{proposition}[Reflected monotonicity formula]
\label{P:reflected-monotonicity}
There exists $C=C(n,k_0,\kappa_0,\Lambda_0)\geq0$ such that the following holds: for every $a\in \overline\Omega \cap U_{\frac{1}{4\kappa_0}}(\partial \Omega)$ and almost every $\rho \in (0,\frac{1}{4\kappa_0})$,
\begin{equation}
\label{E:reflected-mono}
\frac{d}{d\rho}\left(e^{C\rho}\rho^{1-n}\mathcal E_\varepsilon(\rho,a)\right)
\geq e^{C\rho}\rho^{-n}D_\varepsilon(\rho,a).
\end{equation}
On the other hand, for any $a \in \Omega \cap U_{\frac{1}{4\kappa_0}}(\partial \Omega)$ and almost every $0<\rho<\min\{\frac{1}{4\kappa_0},\dist_g(a,\partial \Omega)\}$, we have
\begin{equation}
\label{E:interior-mono-1}
\frac{d}{d\rho}\left(e^{C\rho}\rho^{1-n}
\int_{B_\rho(a)}e_\varepsilon\,\dvol_g\right)
\geq e^{C\rho}\rho^{-n}\int_{B_\rho(a)}(-\xi_\varepsilon)\,\dvol_g.
\end{equation}

\end{proposition}

\begin{remark}
The monotonicity formulas above are stated for the homogeneous case, i.e. $\Lambda_0=0$. When $\Lambda_0 \neq 0$, they hold with $e_\varepsilon$ and $\xi_\varepsilon$ replaced by $\widetilde e_\varepsilon$ and $\widetilde\xi_\varepsilon$ from \eqref{E:modified-densities}.
\end{remark}

With Proposition \ref{P:reflected-monotonicity}, we can now proceed to the proof of Proposition \ref{P:reflected-Hausdorff}. 

\begin{proof}[Proof of Proposition \ref{P:reflected-Hausdorff}]
Let $\mu_i=\mu_{u_{\ep_i}}=\|V^{u_{\ep_i}}\|$ be the Radon measures as defined in \eqref{E:mu}. Note that $\mu_i \to \|V\|$ as Radon measures on $U$. By \cite[Theorem 3.3(2)]{LiPariseSarnataro2023}, we already have the required Hausdorff convergence away from the boundary $\partial \Omega$. Therefore, we now focus on the additional arguments needed to conclude the Hausdorff convergence up to the boundary using Proposition \ref{P:reflected-monotonicity}. 

Fix $b\in(0,1)$ and work on a relatively compact subset of $U$, shrinking the admissible radii if needed. The local gradient estimate and the discrepancy and energy estimates in \cite[Section 4.2 and Propositions 4.10, 4.13, 4.15]{LiPariseSarnataro2023} give
\begin{equation}
\label{E:estimates-for-Hausdorff}
\varepsilon|\nabla u_\varepsilon|\leq C,\qquad
\xi_\varepsilon^+\leq C\varepsilon^{-4/5},\qquad
\int_{B_r(a)\cap\Omega}e_\varepsilon\,\dvol_g\leq Cr^{n-1}
\end{equation}
for $\varepsilon\leq r\leq r_0$ and all centres under consideration. To obtain the last estimate for a ball meeting the boundary whose centre is interior, include it in $B_{2r}(\pi(a))\cap\Omega$; balls disjoint from the boundary are covered by the interior estimate. The integral discrepancy estimate then yields
\begin{equation}
\label{E:integrated-discrepancy-Hausdorff}
r^{-n}\int_{B_r(a)\cap\Omega}\xi_\varepsilon^+\, \dvol_g\leq Cr^{-7/8}
\qquad \text{for any $c\varepsilon^{4/5}\leq r\leq r_0$}.
\end{equation}
Here $c>0$ can be any fixed constant, cf. \cite[Proposition 4.13]{LiPariseSarnataro2023}.

For the reflected sets, $t\leq\widetilde r$ and $\dist_g(x,\mathcal R(x))=2t$ imply
\begin{equation}
\label{E:reflected-ball-comparison}
\widetilde B_r(a)\subset B_{3r}(a)\cap\Omega.
\end{equation}
Thus the same upper bounds apply to the reflected energy and positive discrepancy, after changing the constants.

Suppose now that $|u_\varepsilon(a)|\leq1-b$. The gradient bound in \eqref{E:estimates-for-Hausdorff}, positivity of $W$ away from its wells, and the local volume lower bound give $\varepsilon^{1-n}\mathcal E_\varepsilon(\varepsilon,a)\geq\theta_b>0.$ Apply Proposition \ref{P:reflected-monotonicity} from $\varepsilon$ to $s_\varepsilon=c\varepsilon^{4/5}$. Choosing $c$ sufficiently small preserves a positive lower bound at $s_\varepsilon$. From $s_\varepsilon$ to a fixed small radius, use \eqref{E:integrated-discrepancy-Hausdorff}. Its integral is bounded by $Cr^{1/8}$, so, after decreasing $r_0$, we obtain $\mathcal E_\varepsilon(r,a)\geq\theta r^{n-1}$ for $\varepsilon\leq r\leq r_0.$ Finally, \eqref{E:reflected-ball-comparison} gives the lower bound for the original energy measure:
\begin{equation}
\label{E:transition-density-lower}
\mu_i(B_r(a))\geq\theta_0r^{n-1}
\quad\text{if }|u_i(a)|\leq1-b,\qquad
3\varepsilon_i\leq r\leq r_0.
\end{equation}
The constants are locally uniform in $a$ and independent of $i$ for all sufficiently large $i$.

Consequently, if $a_i\in\{|u_i|\leq1-b\}$ and $a_i\to x\in\partial\Omega\cap U$, then $x\in\spt\mu$. Indeed, for each fixed sufficiently small $r>0$ and all large $i$, $B_{r/4}(a_i)\subset\overline B_{r/2}(x)$. By \eqref{E:transition-density-lower} and weak convergence,
\[
\mu(\overline B_{r/2}(x))\geq\limsup_{i\to\infty}\mu_i(\overline B_{r/2}(x))
\geq\theta_0(r/4)^{n-1}>0.
\]

For completeness, we also justify the converse approximation for every $b\in(0,1)$. Suppose that $x\in\spt\mu$ and, along a subsequence, a connected neighbourhood $B\cap\overline\Omega$ of $x$ contains no point with $|u_i|\leq1-b$. Passing to a further subsequence, assume $u_i>1-b$ there. The energy bound implies $u_i\to1$ in measure and in $L^1$ on this neighbourhood. Choose a $C^1$ function $h$ on $[1-b,C_0]$, with bounded derivative, satisfying
\[
h(1)=0,\qquad 2W h'+W'h=2W.
\]
Such a function extends regularly across $1$: writing $W(s)=(s-1)^2w(s)$, with $w>0$, it is given by
\[
h(s)=\frac{\displaystyle\int_1^s(q-1)\sqrt{w(q)}\,dq}
{(s-1)\sqrt{w(s)}},\qquad h(1)=0.
\]
For a nonnegative smooth cutoff $\psi$ supported in $B$, test \eqref{equation: EL} with $\psi h(u_i)$ to obtain
\[
\int_\Omega\psi\left(\frac{2W(u_i)}{\varepsilon_i}
+2h_0h'(u_i)\xi_{\varepsilon_i}(u_i)\right)\,\dvol_g
=\lambda_i\int_\Omega\psi h(u_i)\,\dvol_g
-\varepsilon_i\int_\Omega h(u_i)\nabla u_i\cdot\nabla\psi\,\dvol_g.
\]
The right-hand side tends to zero: its first term does so by $h(u_i)\to0$ in $L^1$, and its second is $O(\sqrt{\varepsilon_i})$ by the energy bound. By \cite[Theorem 5.4]{LiPariseSarnataro2023}, $\xi_{\varepsilon_i}(u_i)\to0$ in $L^1_{\mathrm{loc}}(\overline\Omega\cap U)$. It follows that $\int\psi W(u_i)/\varepsilon_i\to0$, and hence $\int\psi e_{\varepsilon_i}(u_i)\to0$. Taking $\psi=1$ near $x$ contradicts $x\in\spt\mu$. The negative phase is treated in the same way, using a function $h$ vanishing at $-1$. This proves the converse approximation. Together with the preceding argument and local compactness, it proves Proposition \ref{P:reflected-Hausdorff}.
\end{proof}

\section{Revisiting rectifiability and integrality}
\label{app:rectifiability-integrality}

In this appendix, we give an alternative proof of the rectifiability and integrality for \emph{stable} limit-interfaces of the Allen--Cahn equation with Neumann boundary conditions, following the approach of De Rosa--Pigati \cite[Sections~6--7]{DRP25}. We retain the vanishing-discrepancy theorem of \cite{LiPariseSarnataro2023}, but do not use its rectifiability or integrality conclusions. The argument therefore revisits these two steps in light of stability, rather than providing an independent proof of the entire convergence result in Theorem \ref{T:old-main} in the setting of general critical points. Set-up and notation are as in the previous sections (c.f. Section \ref{S:Allen-Cahn}). We start with the following lemma. 

\begin{lemma}\label{lem:app-first-variation}
There is a constant $C$, independent of $i$, such that
\[
 \int_{\partial \Omega \cap U}e_{\ep_i}(u_{\ep_i})\,d \mathcal{H}^{n - 1} \le C E_0, \qquad \int |B_{u_{\ep_i}}|^2 \, d\mu_{u_{\ep_i}} \le C E_0, \qquad \|\delta\widetilde V_i\|(\widetilde M)\le C E_0.
\]
where $\widetilde{V}^{u_{\ep_i}} \in \bV_{n-1}(U)$ is defined by (c.f. Remark \ref{R:varifold-definition})
\[ \widetilde{V}^{u_{\ep_i}}(\phi)=\int_{\{\nabla u_{\ep_i} \neq 0\}} \sqrt{2W(u_{\ep_i})} |\nabla u_{\ep_i}| \phi(x,S^{u_{\ep_i}}) \, dv_g. \]
\end{lemma}

\begin{proof}
The stress-energy tensor $T_i:=e_{\ep_i}(u_{\ep_i}) g-\varepsilon_i\,du_i\otimes du_i$ satisfies $\operatorname{div}_gT_i=0,$ and $T_i(\nu,\cdot)=e_{\ep_i}(u_{\ep_i})\,g(\nu,\cdot)$ on $\partial \Omega \cap U$. Consequently, for every $X \in \mathfrak{X}_c(U)$,
\begin{equation}\label{eq:app-stress}
 \int_{\Omega \cap U} T_i:\nabla X\, dv_g = \int_{\partial \Omega \cap U} e_{\ep_i}(u_{\ep_i}) \langle X,\nu \rangle\,d\mathcal{H}^{n - 1}.
\end{equation}
Choosing a fixed $X$ with $X=\nu$ on $\partial \Omega \cap U$, since $|T_i|\le C e_{\ep_i}(u_{\ep_i})$, there holds (cf. \cite[Lemma~3.2]{LiPariseSarnataro2023})
\begin{equation}\label{eq:app-boundary-energy}
 \int_{\partial \Omega \cap U} e_{\ep_i}(u_{\ep_i}) \, d \mathcal{H}^{n - 1} = \int_{\Omega \cap U} T_i:\nabla X \, dv_g \le C E_0.
\end{equation}

We next give the isotropic boundary counterpart of \cite[Theorem~6.4]{DRP25}. Set $w_i=|\nabla u_{\ep_i}|$. Arguing as in the proof of Proposition \ref{P:ep-stability}, yields
\begin{align}\label{eq:app-bochner}
 \int_{\Omega \cap U} \varepsilon_i \bigl(|\nabla^2u_{\ep_i}|^2-|\nabla w_i|^2\bigr)\, dv_g \le C E_0, 
\end{align}
where the boundary term is controlled by \eqref{eq:app-boundary-energy}; note that no sign assumption on the second fundamental form of $\partial \Omega$ is required. On $\{w_i>0\}$,
\[
 |\nabla^2u_{\ep_i}|^2-|\nabla w_i|^2 = w_i^2|B_{u_{\ep_i}}|^2+|\nabla^T w_i|^2,
\]
where $\nabla^T$ denotes the tangential gradient along the level sets of $u_{\ep_i}$, which proves the curvature estimate.

For almost every $s$, the regular level set $\{u_{\ep_i}=s\}$ meets $\partial \Omega$ orthogonally, with outward conormal $\nu$, or it does not meet it at all, in which case part of the argument below simplifies. Its first variation is bounded as follows
\begin{align*}
 h_0&\|\delta \widetilde{V}^{u_{\ep_i}}\|(U) \\
 &\le \sqrt{n - 1} \left(\int_{\Omega \cap U} \varepsilon_i|B_{u_{\ep_i}}|^2 w_i^2\, dv_g \right)^{1/2} \left(\int_{\Omega \cap U} \frac{2W(u_{\ep_i})}{\varepsilon_i}\, dv_g \right)^{1/2} + \int_{\partial \Omega \cap U} \sqrt{2W(u_{\ep_i})}\,w_i\,d\mathcal{H}^{n - 1} \\
 & \le C E_0.
\end{align*}
This is the analogue of \cite[Corollary~6.6]{DRP25}, with the additional boundary contribution. In particular, the estimate concerns the full ambient first variation, not only vector fields tangent to the boundary. 
\end{proof}

\begin{proposition}\label{prop:app-rectifiability}
The varifold $V$ is rectifiable and stationary with free boundary:
\[
 \delta V(X)=0 \quad\text{for every }X\in \mathfrak{X}^{tan}_c(U).
\]
\end{proposition}

\begin{proof}
We adapt the positive-density argument of \cite[Theorem~6.7]{DRP25}. First, the boundary version of \cite[Lemma~6.2]{DRP25} gives, for each sufficiently small $b>0$, constants $R,c_b>0$ such that
\begin{equation}\label{eq:app-microscopic-density}
 |u_{\ep_i}(x)|\le 1- b \quad\Longrightarrow\quad \|\widetilde{V}^{u_{\ep_i}}\|(B_{R\varepsilon_i}(x))
 \ge c_b\varepsilon_i^{n - 1}
\end{equation}
for all large $i$. To see this, suppose that the normalised masses on the left tend to zero and rescale by $\varepsilon_i$ about the corresponding points. Elliptic estimates give a limit on $B_R(0)$ intersected with a full space or a half-space containing the origin. This limit is then a constant. Its value lies in $[-1+b,1-b]$, and the equation forces it to equal $\gamma$ (the unique local non-degenerate maxima of $W$). Stability would then imply
\[
 -W''(\gamma)\int\zeta^2\le\int|\nabla\zeta|^2
\]
for every test function supported away from the curved boundary, with no restriction on its trace on the flat boundary. A radial cutoff contradicts this inequality when $R$ is sufficiently large.

Multiplying \eqref{E:Allen-Cahn} with $\gamma - u_{\ep_i}$ gives
\[
 \int_{\Omega \cap U} \frac{(\gamma-u_{\ep_i})W'(u_{\ep_i})}{\varepsilon_i}\,dv_g =\int_{\Omega \cap U} \varepsilon_i |\nabla u_{\ep_i}|^2 \, dv_g \le 2E_0.
\]
The integrand on the left is nonnegative, and $W(s)\le Cb(\gamma-s)W'(s)$ for $|s|>1-b$. Using $h_0 e_{\ep_i} (u_{\ep_i})=2W(u_{\ep_i})/\varepsilon_i+ h_0 \xi_{\ep_i}(u_{\ep_i})$ therefore gives
\begin{equation}\label{eq:app-well-energy}
\int_{\{|u_{\ep_i}|>1-b\}} e_{\ep_i}(u_{\ep_i}) \, dv_g \le Cb+\int_{\Omega \cap U} |\xi_{\ep_i}(u_{\ep_i})|\, dv_g.
\end{equation}
Compare \cite[equation~(6.12)]{DRP25}.

Fix a sufficiently small $r_0>0$ and define
\[
 G_{i,b}:=\left\{x\in \Omega \cap U:
 \begin{array}{l}
 |u_{\ep_i}(x)|\le1-b,\\
 \|\delta\widetilde{V}^{u_{\ep_i}}\|(B_r(x)) \le b^{-1}\|\widetilde{V}^{u_{\ep_i}}\|(B_r(x)) \quad\text{for every }0<r<r_0
 \end{array}\right\}.
\]
Besicovitch's covering theorem, Lemma~\ref{lem:app-first-variation}, and \eqref{eq:app-well-energy} imply
\begin{equation}\label{eq:app-good-mass}
 \|\widetilde{V}^{u_{\ep_i}}\|(M\setminus G_{i,b}) \leq Cb + o_i(1).
\end{equation}
This is the good-set construction in \cite[equation~(6.11) and the proof of Theorem~6.7]{DRP25}. At a point of $G_{i,b}$, the first-variation bound in its definition and the varifold monotonicity inequality propagate \eqref{eq:app-microscopic-density} to
\begin{equation}\label{eq:app-propagated-density}
 \|\widetilde{V}^{u_{\ep_i}}\|(B_r(x))\ge c_b r^{n - 1} \qquad(R\varepsilon_i\le r<r_0).
\end{equation}
The exponential factor absorbs both $b^{-1}$ and the ambient geometric error. Since we have bounded the full first variation in $U$, this applies also when the ball intersects $\partial \Omega$.

For fixed $b$, take a weak limit $\sigma_b$ of $\|\widetilde{V}^{u_{\ep_i}}\| \mres G_{i,b}$. Then $\sigma_b\le\mu$ and $(\mu-\sigma_b)(U)\le Cb$. Every point of $\operatorname{spt}\sigma_b$ is a limit of points in $G_{i,b}$, so \eqref{eq:app-propagated-density}, applied first on slightly smaller balls, gives $\mu(B_r(x))\ge c_b r^{n - 1}$ there. Letting $b$ tend to zero along a countable sequence proves positive upper $(n - 1)$-density at $\mu$-almost every point. Lemma~\ref{lem:app-first-variation} and Allard's rectifiability theorem \cite[Section~5]{Allard1972} now show that $V$ is rectifiable.

Finally, $h_0 e_{\ep_i} (u_{\ep_i}) - \sqrt{2W(u_{\ep_i})} \vert \nabla u_{\ep_i} \vert \leq h_0 \vert \xi_{\ep_i}(u_{\ep_i}) \vert$ combined with $|\varepsilon_i|\nabla u_{\ep_i}|^2- \sqrt{2W(u_{\ep_i})} \vert \nabla u_{\ep_i} \vert|\le 2 h_0 |\xi_{\ep_i}(u_{\ep_i})|$ imply that $h_0^{-1}T_i\,dv_g$ converges. For $X \in \mathfrak{X}^{tan}_c(U)$, the right-hand side of \eqref{eq:app-stress} vanishes. Passing to the limit proves $\delta V(X)=0$.
\end{proof}

\begin{proposition}\label{prop:app-integrality}
The varifold $V$ is integral, including its restriction to $\partial \Omega$ (compare with \cite[Proposition~6.1]{LiPariseSarnataro2023}). 
\end{proposition}

\begin{proof}
We follow  \cite[Lemmas~7.7--7.8 and Theorem~7.10]{DRP25}. The fact that $\xi_{\ep_i}(u_{\ep_i}) \rightarrow 0$ in $L^1$ simplifies the proof. 

Fix a point $x$ at which $V$ has density $\theta\in(0,\infty)$ and a unique tangent varifold $\theta|P|$. At $\mu$-almost every such point in $\partial \Omega \cap U$, arguing as in \cite[Proposition~6.1]{LiPariseSarnataro2023} it is enough to consider $P=T_x\partial \Omega$. Using normal coordinates in the interior and Fermi coordinates at the boundary, choose a diagonal blow-up sequence with $\varepsilon_i/r_i\to0$. The resulting metrics $g_i$ converge smoothly to the Euclidean metric, and the domains are either a full space or the half-space $\{t\ge0\}$. Write $z=(y,t)\in\mathbb R^{n - 1}\times\mathbb R$, with $P=\{t=0\}$, and set
\[
 I_r:=(-r,r)\quad\text{or}\quad I_r:=[0,r), \qquad Q_r:=B_r^{n - 1}\times I_r,
\]
respectively. In the boundary case $\partial_tu_{\ep_i}=0$ at $t=0$. From now on, write
\[
 \bar e_i:=\frac{\varepsilon_i}{2}|Du_{\ep_i}|^2 + \frac{W(u_{\ep_i})}{\varepsilon_i}, \qquad \bar\xi_i:=\frac{\varepsilon_i}{2}|Du_{\ep_i}|^2 - \frac{W(u_{\ep_i})}{\varepsilon_i}
\]
for the Euclidean energy and discrepancy in these coordinates. The metric errors vanish, and varifold convergence together with vanishing of the discrepancy in $L^1$ yields
\begin{equation}\label{eq:app-flat-limit}
 h_0^{-1}\bar e_i\,dz\rightharpoonup \theta\,\mathcal H^{n - 1} \mres P, \qquad \int_{Q_3}\bigl(|\bar\xi_i|+ \varepsilon_i|D_yu_{\ep_i}|^2\bigr)\,dz\longrightarrow0.
\end{equation}
The second assertion is the isotropic counterpart of \cite[equation~(7.12)]{DRP25}, with the discrepancy included.

Choose $0\le\chi\in C_c^\infty((-1,1))$ with $\chi=1$ on $[-1/2,1/2]$, and define
\[
 f_i(y):=\frac1{h_0}\int_{I_1}\chi(t)\bar e_i(y,t)\,dt.
\]
As in \cite[Lemma~7.7]{DRP25}, test the stress-energy identity with $X=\chi(t)\varphi(y)\partial_{y_a}$. The resulting distributional identity on $B_2^{n - 1}$ has the form
\begin{equation}\label{eq:app-projected-stress}
 Df_i=\operatorname{div}R_i+r_i,
 \qquad \|R_i\|_{L^1(B_2^{n - 1})}+\|r_i\|_{L^1(B_2^{n - 1})}\rightarrow0.
\end{equation}
In the Euclidean metric, the terms are explicit
\[
 R_i^{ab}=\frac1{h_0}\int_{I_1}\chi\varepsilon_i
                  \partial_{y_a}u_i\partial_{y_b}u_i\,dt,
 \qquad
 r_i^a=-\frac1{h_0}\int_{I_1}\chi'\varepsilon_i
                  \partial_{y_a}u_i\partial_tu_i\,dt.
\]
Their $L^1$ norms tend to zero by \eqref{eq:app-flat-limit} and Cauchy--Schwarz. The additional tensor and zeroth-order terms are bounded by $C\|g_i-\delta\|_{C^1}\int_{Q_3}\bar e_i$. Since $f_i\ge0$ and $f_i\,dy\rightharpoonup\theta\,dy$, Allard's strong constancy lemma \cite[Theorem~1.(4)]{Allard86} gives
\begin{equation}\label{eq:app-strong-constancy}
 f_i\rightarrow\theta\quad\text{in }L^1(B_1^{n - 1}).
\end{equation}
There holds
\begin{equation}\label{eq:app-local-wells}
 \int_{Q_2\cap\{|u_i|>1-b\}}\bar e_i\,dz
 \le Cb+o_i(1).
\end{equation}
Put
\[
 d_i(y):=\int_{I_2}\bigl(|\bar\xi_i|+
                    \varepsilon_i|D_yu_{\ep_i}|^2\bigr)(y,t)\,dt.
\]
The weak-$(1,1)$ estimate for the local Hardy--Littlewood maximal operator $\mathcal M$ and \eqref{eq:app-flat-limit} show that $\mathcal M d_i\to0$ in measure on $B_1^{n - 1}$. For fixed $b>0$, we can therefore choose Lebesgue points $y_i\in B_1^{n - 1}$ such that
\begin{align}\label{eq:app-good-slice}
 f_i(y_i)&\rightarrow\theta,
 &\mathcal M d_i(y_i)&\rightarrow0,\notag\\
 \int_{I_2\cap\{|t|\ge1/4\}}\bar e_i(y_i,t)\,dt
 &\rightarrow0,
 &\int_{I_2\cap\{|u_{\ep_i}(y_i,\cdot)|>1-b\}}
       \bar e_i(y_i,t)\,dt&\le Cb+o_i(1).
\end{align}
For the third condition use the concentration of the limit measure on $P$. For the last one use \eqref{eq:app-local-wells} and Markov's inequality, increasing $C$ so that its exceptional set occupies at most half of $B_1^{n - 1}$. The exceptional sets for the other conditions have vanishing measure. This makes explicit the selection used in \cite[proof of Theorem~7.10]{DRP25}. In particular,
\begin{equation}\label{eq:app-total-slice}
 \int_{I_1}\bar e_i(y_i,t)\,dt\rightarrow h_0\theta,
 \qquad \text{and} \qquad d_i(y_i)\rightarrow0.
\end{equation}

Consider any sequence $t_i\in I_1$ such that
$|u_{\ep_i}(y_i,t_i)|\le1-b$, and rescale about $(y_i,t_i)$ by $\varepsilon_i$. For each fixed $R$, the definition of the maximal operator gives
\begin{equation}\label{eq:app-microscopic-defect}
 \varepsilon_i^{-(n - 1)}\int_{B_{R\varepsilon_i}(y_i,t_i)\cap Q_2}
 \bigl(|\bar\xi_i|+\varepsilon_i|D_yu_{\ep_i}|^2\bigr)\,dz
 \le C R^{n - 1}\mathcal M d_i(y_i)\rightarrow0.
\end{equation} 
Compactness therefore gives a one-dimensional limit $q=q(t)$ satisfying (cf. \cite[Lemma~7.8 and Remark~7.9]{DRP25})
\begin{equation}\label{eq:app-profile}
 q''=W'(q),\qquad \frac12(q')^2=W(q),\qquad |q(0)|\le1-b.
\end{equation}

In the boundary case, if $t_i/\varepsilon_i$ remained bounded, the limit would be defined on a half-line $[-b,\infty)$ and satisfy $q'(-b)=0$. Equation~\eqref{eq:app-profile} would force $q(-b)=\pm1$, and uniqueness for the ordinary differential equation would make $q$ constant, a contradiction. Moreover, another contradiction argument using \eqref{eq:app-profile} gives
\begin{equation}\label{eq:app-transversality}
 \varepsilon_i|\partial_tu_{\ep_i}(y_i,t)|\ge c_b>0
 \quad\text{whenever }t\in I_1,
 \quad |u_{\ep_i}(y_i,t)|\le1-b,
\end{equation}
for all large $i$.

The uniform estimate $\varepsilon_i|Du_{\ep_i}|\le C$ on $Q_2$ and the third condition in \eqref{eq:app-good-slice} exclude transition points near the artificial endpoints of $I_1$. Together with the preceding boundary argument, this shows that every component of
\[
 \{t\in I_1:|u_{\ep_i}(y_i,t)|<1-b\}
\]
has both endpoints in the interior of $I_1$. By \eqref{eq:app-transversality}, the restriction of $u_i$ to each component is strictly monotone, and its endpoint values are $-1+b$ and $1-b$, in either order.

We can now prove integrality. Let $N_i$ be the number of these components and set
\[
 h_b:=\int_{-1+b}^{1-b}\sqrt{2W(s)}\,ds.
\]
The elementary inequality
\begin{equation}\label{eq:app-slice-comparison}
 0\le\bar e_i-\sqrt{2W(u_{\ep_i})}\,|\partial_t u_{\ep_i}|
 \le|\bar\xi_i|+\varepsilon_i|D_yu_{\ep_i}|^2
\end{equation}
shows, by change of variables on each monotone component and \eqref{eq:app-total-slice}, that
\[
 \int_{I_1\cap\{|u_{\ep_i}(y_i,\cdot)|<1-b\}}\bar e_i(y_i,t)\,dt =N_i h_b+o_i(1).
\]
In particular, $N_i$ is bounded independently of $i$ and of sufficiently small $b$. The remaining energy is bounded by the last condition in \eqref{eq:app-good-slice}. Since $h_b=h_0+O(b^2)$, we obtain $ \operatorname{dist}(\theta,\mathbb N)\le Cb.$ Letting $b \downarrow0$ proves $\theta\in\mathbb N$ and completes the argument.
\end{proof}

\section{Measure-function convergence}
\label{ap:measure-function}

In this appendix we shall summarise some basic notions and results from the theory of \textit{measure-function pairs}, which are used in the proof of Theorem \ref{T:main-AC}. See \cite{HutchinsonIUMJ, Tonegawa2005a} or \cite{Hiesmayr2017} for more details.

\begin{definition}
Let $X$ be a locally compact Hausdorff space and $N\in \mathbb{N}$. Let $\mu$ be a Radon measure on $X$ and $f:X\to\R^N$ a Borel-measurable vector-valued function which is in $L^1_{loc}(X;\mu)$. We shall call such a pair $(\mu,f)$ a \textit{measure-function pair} over $X$.
\end{definition}

\begin{definition}
Let $(\mu_i, f_i)_{i\in\N}$ be a sequence of measure-function pairs over $X$. We say $(\mu_i, f_i)$ \textit{converges} to a measure function pair $(\mu,f)$ over $X$ in the \textit{weak sense} (or, more simply, \textit{weakly}) if the following conditions hold:
\begin{enumerate}[(i)]
    \item $\mu_i \to \mu$ as Radon measures on $X$; and
    \item for any $\phi \in C_c(X,\R^N)$, we have
    \begin{equation}
        \int_X\langle f_i, \phi\rangle\,d\mu_i\to\int_X\langle f,\phi\rangle\, d\mu.
    \end{equation}
\end{enumerate}
We shall denote this convergence by $(\mu_i, f_i)\rightharpoonup(\mu, f)$.
\end{definition}

The main result we shall need in the proof of Theorem \ref{T:main-AC} is the following compactness result.
\begin{proposition}[\cite{HutchinsonIUMJ}]
\label{prop:measure-function compactness}
Let $(\mu_i, f_i)_{i\in\N}$ be a sequence of measure-function pairs over $X$. Suppose that
\begin{equation}
    \sup_i \mu_i(X)<+\infty,
\end{equation}
and
\begin{equation}
    \sup_i\int_X|f_i|^2\,d\mu_i<+\infty.
\end{equation}
Then, there is a subsequence $(\mu_{i'}, f_{i'})$ and a measure-function pair $(\mu, f)$ over $X$ such that
\begin{equation}
    (\mu_{i'}, f_{i'})\rightharpoonup(\mu,f).
\end{equation}
Moreover, the weak limit $(\mu, f)$ satisfies
\begin{equation}
\label{E:lower-semicontinuity-measure-function}
    \int_X|f|^2\, d\mu\leq\liminf_{i}\int_X|f_i|^2\, d\mu_i.
\end{equation}
\end{proposition}

\bibliographystyle{amsalpha}
\bibliography{references}

@article {Simon2023,
    AUTHOR = {Simon, Leon},
     TITLE = {Stable minimal hypersurfaces in {$\Bbb R^{N+1+\ell}$} with
              singular set an arbitrary closed {$K\subset\{0\}\times\Bbb
              R^{\ell}$}},
   JOURNAL = {Ann. of Math. (2)},
  FJOURNAL = {Annals of Mathematics. Second Series},
    VOLUME = {197},
      YEAR = {2023},
    NUMBER = {3},
     PAGES = {1205--1234},
      ISSN = {0003-486X,1939-8980},
   MRCLASS = {49Q05 (28A75 35A30 53A10 53C23)},
  MRNUMBER = {4564263},
MRREVIEWER = {Shi-Zhong\ Du},
       DOI = {10.4007/annals.2023.197.3.4},
       URL = {https://doi-org.easyaccess2.lib.cuhk.edu.hk/10.4007/annals.2023.197.3.4},
}

@article {Almgren1976,
    AUTHOR = {Almgren, Jr., F. J.},
     TITLE = {Existence and regularity almost everywhere of solutions to
              elliptic variational problems with constraints},
   JOURNAL = {Mem. Amer. Math. Soc.},
  FJOURNAL = {Memoirs of the American Mathematical Society},
    VOLUME = {4},
      YEAR = {1976},
    NUMBER = {165},
     PAGES = {viii+199},
      ISSN = {0065-9266,1947-6221},
   MRCLASS = {49F22 (58E15)},
  MRNUMBER = {420406},
MRREVIEWER = {Jean\ E.\ Taylor},
       DOI = {10.1090/memo/0165},
       URL = {https://doi-org.easyaccess1.lib.cuhk.edu.hk/10.1090/memo/0165},
}

@article{DDPS2026,
      title={Second order estimates on transition layers for the Allen-Cahn equation with homogeneous Neumann boundary conditions: Part I - Boundary-orthogonal transition layers}, 
      author={Dey, Akashdeep and Du, Wenkui and Parise, Davide and Sarnataro, Lorenzo},
      year={2026},
      journal={arXiv},
      pages={2609.32829},
      eprint={2609.32829},
      archivePrefix={arXiv},
      primaryClass={math.AP}
}

@article{CFFFS2026,
      title={Stable solutions of the Allen-Cahn equation in dimension three are one-dimensional}, 
      author={Chan, Hardy and Fern`{a}ndez-Real, Xavier and Figalli, Alessio and Florit-Simon, Enric and Serra, Joaqium},
      year={2026},
      journal={arXiv},
      pages={2609.30194},
      eprint={2609.30194},
      archivePrefix={arXiv},
      primaryClass={math.AP}
}

@article{Hong-Li-Wang2026,
      title={The stable Bernstein theorem in $\mathbb{R}^7$}, 
      author={Hong, Han and Li, Haizhong and Wang, Gaoming},
      year={2026},
      journal={arXiv},
      pages={2609.15720},
      eprint={2609.15720},
      archivePrefix={arXiv},
      primaryClass={math.DG}
}

@article{LLWWWW2026,
      title={Stable De Giorgi conjecture of the Allen-Cahn equation in $\mathbb{R}^3$}, 
      author={Liu, Yong and Luo, Tianci and Wang, Kelei and Wei, Juncheng and Wei, Yong and Wu, Ke},
      year={2026},
      journal={arXiv},
      pages={2609.21680},
      eprint={2609.21680},
      archivePrefix={arXiv},
      primaryClass={math.AP}
}

@article {Jacobowitz72,
    AUTHOR = {Jacobowitz, Howard},
     TITLE = {Implicit function theorems and isometric embeddings},
   JOURNAL = {Ann. of Math. (2)},
  FJOURNAL = {Annals of Mathematics. Second Series},
    VOLUME = {95},
      YEAR = {1972},
     PAGES = {191--225},
      ISSN = {0003-486X},
   MRCLASS = {53C40 (58C15)},
  MRNUMBER = {307127},
MRREVIEWER = {C.\ S.\ Houh},
       DOI = {10.2307/1970796},
       URL = {https://doi-org.easyaccess2.lib.cuhk.edu.hk/10.2307/1970796},
}

@incollection {Allard86,
    AUTHOR = {Allard, William K.},
     TITLE = {An integrality theorem and a regularity theorem for surfaces
              whose first variation with respect to a parametric elliptic
              integrand is controlled},
 BOOKTITLE = {Geometric measure theory and the calculus of variations
              ({A}rcata, {C}alif., 1984)},
    SERIES = {Proc. Sympos. Pure Math.},
    VOLUME = {44},
     PAGES = {1--28},
 PUBLISHER = {Amer. Math. Soc., Providence, RI},
      YEAR = {1986},
      ISBN = {0-8218-1470-2},
   MRCLASS = {49F20 (28A75)},
  MRNUMBER = {840267},
MRREVIEWER = {Marcello\ Ragni},
       DOI = {10.1090/pspum/044/840267},
       URL = {https://doi-org.easyaccess2.lib.cuhk.edu.hk/10.1090/pspum/044/840267},
}

@article {Naber-Valtorta2020,
    AUTHOR = {Naber, Aaron and Valtorta, Daniele},
     TITLE = {The singular structure and regularity of stationary varifolds},
   JOURNAL = {J. Eur. Math. Soc. (JEMS)},
  FJOURNAL = {Journal of the European Mathematical Society (JEMS)},
    VOLUME = {22},
      YEAR = {2020},
    NUMBER = {10},
     PAGES = {3305--3382},
      ISSN = {1435-9855,1435-9863},
   MRCLASS = {58E12 (35J60 35J93 49Q20)},
  MRNUMBER = {4153109},
MRREVIEWER = {Pei\ Biao\ Zhao},
       DOI = {10.4171/jems/987},
       URL = {https://doi-org.easyaccess1.lib.cuhk.edu.hk/10.4171/jems/987},
}

@article{DRP25,
      title={Anisotropic min-max via phase transitions}, 
      author={De Rosa, Antonio and Pigati, Alessandro},
      year={2025},
      journal={arXiv},
      pages={2512.07039},
      eprint={2512.07039},
      archivePrefix={arXiv},
      primaryClass={math.DG}
}

@unpublished{Bellettini-Li-Parise-Sarnataro,
      title={Regularity and compactness theory of stable free-boundary minimal hypersurfaces}, 
      author={Bellettini, Costante and Li, Martin Man-chun and Parise, Davide and Sarnataro, Lorenzo},
      year={2026},
      note={preprint},
}

@article {Sharp2017,
    AUTHOR = {Sharp, Ben},
     TITLE = {Compactness of minimal hypersurfaces with bounded index},
   JOURNAL = {J. Differential Geom.},
  FJOURNAL = {Journal of Differential Geometry},
    VOLUME = {106},
      YEAR = {2017},
    NUMBER = {2},
     PAGES = {317--339},
      ISSN = {0022-040X,1945-743X},
   MRCLASS = {53C42 (53C21)},
  MRNUMBER = {3662994},
MRREVIEWER = {Luciano\ Mari},
       DOI = {10.4310/jdg/1497405628},
       URL = {https://doi-org.easyaccess2.lib.cuhk.edu.hk/10.4310/jdg/1497405628},
}

@article {Menne2016b,
    AUTHOR = {Menne, Ulrich},
     TITLE = {Sobolev functions on varifolds},
   JOURNAL = {Proc. Lond. Math. Soc. (3)},
  FJOURNAL = {Proceedings of the London Mathematical Society. Third Series},
    VOLUME = {113},
      YEAR = {2016},
    NUMBER = {6},
     PAGES = {725--774},
      ISSN = {0024-6115,1460-244X},
   MRCLASS = {46E35 (49Q15)},
  MRNUMBER = {3626845},
MRREVIEWER = {Juha\ K.\ Kinnunen},
       DOI = {10.1112/plms/pdw023},
       URL = {https://doi-org.easyaccess2.lib.cuhk.edu.hk/10.1112/plms/pdw023},
}

@article {Menne2016,
    AUTHOR = {Menne, Ulrich},
     TITLE = {Weakly differentiable functions on varifolds},
   JOURNAL = {Indiana Univ. Math. J.},
  FJOURNAL = {Indiana University Mathematics Journal},
    VOLUME = {65},
      YEAR = {2016},
    NUMBER = {3},
     PAGES = {977--1088},
      ISSN = {0022-2518,1943-5258},
   MRCLASS = {49Q15 (49Q20)},
  MRNUMBER = {3528825},
MRREVIEWER = {Andrew\ Bucki},
       DOI = {10.1512/iumj.2016.65.5829},
       URL = {https://doi-org.easyaccess1.lib.cuhk.edu.hk/10.1512/iumj.2016.65.5829},
}

@article {Li2015,
    AUTHOR = {Li, Martin Man-chun},
     TITLE = {A general existence theorem for embedded minimal surfaces with
              free boundary},
   JOURNAL = {Comm. Pure Appl. Math.},
  FJOURNAL = {Communications on Pure and Applied Mathematics},
    VOLUME = {68},
      YEAR = {2015},
    NUMBER = {2},
     PAGES = {286--331},
      ISSN = {0010-3640,1097-0312},
   MRCLASS = {53C42 (53A10)},
  MRNUMBER = {3298664},
MRREVIEWER = {Rafael\ L\'opez},
       DOI = {10.1002/cpa.21513},
       URL = {https://doi-org.easyaccess2.lib.cuhk.edu.hk/10.1002/cpa.21513},
}

@article {Wang-Wei2019,
    AUTHOR = {Wang, Kelei and Wei, Juncheng},
     TITLE = {Finite {M}orse index implies finite ends},
   JOURNAL = {Comm. Pure Appl. Math.},
  FJOURNAL = {Communications on Pure and Applied Mathematics},
    VOLUME = {72},
      YEAR = {2019},
    NUMBER = {5},
     PAGES = {1044--1119},
      ISSN = {0010-3640,1097-0312},
   MRCLASS = {35J91 (53A10 58E05)},
  MRNUMBER = {3935478},
MRREVIEWER = {Jos\'e\ Carmona Tapia},
       DOI = {10.1002/cpa.21812},
       URL = {https://doi-org.easyaccess2.lib.cuhk.edu.hk/10.1002/cpa.21812},
}

@article {Figalli-Serra2020,
    AUTHOR = {Figalli, Alessio and Serra, Joaquim},
     TITLE = {On stable solutions for boundary reactions: a {D}e
              {G}iorgi-type result in dimension {$4+1$}},
   JOURNAL = {Invent. Math.},
  FJOURNAL = {Inventiones Mathematicae},
    VOLUME = {219},
      YEAR = {2020},
    NUMBER = {1},
     PAGES = {153--177},
      ISSN = {0020-9910,1432-1297},
   MRCLASS = {35R11},
  MRNUMBER = {4050103},
MRREVIEWER = {Vincenzo\ Ambrosio},
       DOI = {10.1007/s00222-019-00904-2},
       URL = {https://doi-org.easyaccess1.lib.cuhk.edu.hk/10.1007/s00222-019-00904-2},
}

@article {Cabre-Figalli-RosOton-Serra2020,
    AUTHOR = {Cabr\'e, Xavier and Figalli, Alessio and Ros-Oton, Xavier and
              Serra, Joaquim},
     TITLE = {Stable solutions to semilinear elliptic equations are smooth
              up to dimension 9},
   JOURNAL = {Acta Math.},
  FJOURNAL = {Acta Mathematica},
    VOLUME = {224},
      YEAR = {2020},
    NUMBER = {2},
     PAGES = {187--252},
      ISSN = {0001-5962,1871-2509},
   MRCLASS = {35J61 (35B65 35J91)},
  MRNUMBER = {4117051},
MRREVIEWER = {Tom\'as\ Sanz-Perela},
       DOI = {10.4310/acta.2020.v224.n2.a1},
       URL = {https://doi-org.easyaccess1.lib.cuhk.edu.hk/10.4310/acta.2020.v224.n2.a1},
}

@article {Cabre2026,
    AUTHOR = {Cabr\'e, Xavier},
     TITLE = {H\"older regularity of stable solutions to elliptic equations
              up to $\mathbb{R}^9$: full quantitative proofs},
   JOURNAL = {Bull. Amer. Math. Soc. (N.S.)},
  FJOURNAL = {American Mathematical Society. Bulletin. New Series},
    VOLUME = {63},
      YEAR = {2026},
    NUMBER = {1},
     PAGES = {1--57},
      ISSN = {0273-0979,1088-9485},
   MRCLASS = {35J61 (35-02 35B45 35B65 35J15)},
  MRNUMBER = {5009056},
       DOI = {10.1090/bull/1873},
       URL = {https://doi-org.easyaccess1.lib.cuhk.edu.hk/10.1090/bull/1873},
}

@article {Chan2026,
    AUTHOR = {Chan, Hardy and Fern\'{a}ndez-Real, Xavier and Figalli, Alessio and Serra, Joaquim},
     TITLE = {Global stable solutions to the free boundary {A}llen-{C}ahn and {B}ernoulli problems in 3{D} are one-dimensional},
   JOURNAL = {J. Amer. Math. Soc.},
  FJOURNAL = {Journal of the American Mathematical Society},
    VOLUME = {39},
      YEAR = {2026},
    NUMBER = {4},
     PAGES = {1027--1151},
      ISSN = {0894-0347,1088-6834},
       DOI = {10.1090/jams/1074},
       URL = {https://pubs.ams.org/JAMS/2026-39-04/S0894-0347-2026-01074-0},
}

@article{Florit-Simon2026,
      title={Phase transitions with bounded index: {P}arallels to {D}e Giorgi's conjecture}, 
      author={Florit-Simon, Enric},
      year={2026},
      journal={arXiv},
      pages={2602.03136},
      eprint={2602.03136},
      archivePrefix={arXiv},
      primaryClass={math.AP}
}

@article{Florit-Simon-Serra2025,
      title={On stable solutions to the Allen-Cahn equation with bounded energy density in $\mathbb{R}^4$}, 
      author={Florit-Simon, Enric and Serra, Joaquim},
      year={2025},
      journal={arXiv},
      pages={2509.02739},
      eprint={2509.02739},
      archivePrefix={arXiv},
      primaryClass={math.AP}
}

@article{Mazet2024,
      title={Stable minimal hypersurfaces in $\mathbb{R}^6$}, 
      author={Mazet, Laurent},
      year={2024},
      journal={arXiv},
      pages={2405.14676},
      eprint={2405.14676},
      archivePrefix={arXiv},
      primaryClass={math.DG}
}

@article {Chodosh-Li2024,
    AUTHOR = {Chodosh, Otis and Li, Chao},
     TITLE = {Stable minimal hypersurfaces in $\mathbb{R}^4$},
   JOURNAL = {Acta Math.},
  FJOURNAL = {Acta Mathematica},
    VOLUME = {233},
      YEAR = {2024},
    NUMBER = {1},
     PAGES = {1--31},
      ISSN = {0001-5962,1871-2509},
   MRCLASS = {53A10 (53C42)},
  MRNUMBER = {4816633},
MRREVIEWER = {Jinyu\ Guo},
       DOI = {10.4310/acta.2024.v233.n1.a1},
       URL = {https://doi-org.easyaccess1.lib.cuhk.edu.hk/10.4310/acta.2024.v233.n1.a1},
}

@article {Chodosh-Li-Minter-Stryker2026,
    AUTHOR = {Chodosh, Otis and Li, Chao and Minter, Paul and Stryker,
              Douglas},
     TITLE = {Stable minimal hypersurfaces in $\mathbb{R}^5$},
   JOURNAL = {Ann. of Math. (2)},
  FJOURNAL = {Annals of Mathematics. Second Series},
    VOLUME = {204},
      YEAR = {2026},
    NUMBER = {1},
     PAGES = {423--453},
      ISSN = {0003-486X,1939-8980},
   MRCLASS = {53C42 (35J50 53A07 53A10)},
  MRNUMBER = {5094441},
       DOI = {10.4007/annals.2026.204.1.5},
       URL = {https://doi-org.easyaccess1.lib.cuhk.edu.hk/10.4007/annals.2026.204.1.5},
}

@article {Modica-Mortola-1977,
    AUTHOR = {Modica, Luciano and Mortola, Stefano},
     TITLE = {Un esempio di {$\Gamma \sp{-}$}-convergenza},
   JOURNAL = {Boll. Un. Mat. Ital. B (5)},
  FJOURNAL = {Unione Matematica Italiana. Bollettino. B. Serie V},
    VOLUME = {14},
      YEAR = {1977},
    NUMBER = {1},
     PAGES = {285--299},
   MRCLASS = {49A50},
  MRNUMBER = {445362},
MRREVIEWER = {Francesco\ Barbieri},
}

@article{Hong-Li-Wang2024,
      title={On $\delta$-{S}table {M}inimal {H}ypersurfaces in $\mathbb{R}^{n+1}$}, 
      author={Hong, Han and Li, Haizhong and Wang, Gaoming},
      year={2024},
      journal={arXiv},
      pages={2407.03222},
      eprint={2407.03222},
      archivePrefix={arXiv},
      primaryClass={math.DG}
}

@article{Koltai-Volkner,
      title={Structure theory of {B}{V} functions and finite perimeter sets on {R}iemannian manifolds}, 
      author={Koltai, P\'{e}ter and V\"{o}lkner, Kathrin},
      year={2026},
      journal={arXiv},
      pages={2607.21031v1},
      eprint={2607.21031v1},
      archivePrefix={arXiv},
      primaryClass={math.DG}
}

@incollection {Allard-Almgren1976,
 AUTHOR = {Allard, W. K. and Almgren, Jr., F. J.},
     TITLE = {The structure of stationary one dimensional varifolds with
              positive density},
   JOURNAL = {Invent. Math.},
  FJOURNAL = {Inventiones Mathematicae},
    VOLUME = {34},
      YEAR = {1976},
    NUMBER = {2},
     PAGES = {83--97},
      ISSN = {0020-9910,1432-1297},
   MRCLASS = {49F20 (58E15)},
  MRNUMBER = {425741},
MRREVIEWER = {J.\ E.\ Brothers},
       DOI = {10.1007/BF01425476},
       URL = {https://doi.org/10.1007/BF01425476},
}

@book {Brakke,
    AUTHOR = {Brakke, Kenneth A.},
     TITLE = {The motion of a surface by its mean curvature},
    SERIES = {Mathematical Notes},
    VOLUME = {20},
 PUBLISHER = {Princeton University Press, Princeton, NJ},
      YEAR = {1978},
     PAGES = {i+252},
      ISBN = {0-691-08204-9},
   MRCLASS = {49F22 (35K99 49F20 58D25)},
  MRNUMBER = {485012},
MRREVIEWER = {Jean\ E.\ Taylor},
}

@article {Koike-Yamamoto2018,
    AUTHOR = {Koike, Naoyuki and Yamamoto, Hikaru},
     TITLE = {Gauss maps of the {R}icci-mean curvature flow},
   JOURNAL = {Geom. Dedicata},
  FJOURNAL = {Geometriae Dedicata},
    VOLUME = {194},
      YEAR = {2018},
     PAGES = {169--185},
      ISSN = {0046-5755,1572-9168},
   MRCLASS = {53C42 (53C44)},
  MRNUMBER = {3798066},
MRREVIEWER = {Yi\ Li},
       DOI = {10.1007/s10711-017-0271-8},
       URL = {https://doi-org.easyaccess1.lib.cuhk.edu.hk/10.1007/s10711-017-0271-8},
}

@incollection {Allard-Almgren1973,
    AUTHOR = {Allard, W. K. and Almgren, Jr., F. J.},
     TITLE = {An introduction to regularity theory for parametric elliptic
              variational problems},
 BOOKTITLE = {Partial differential equations ({P}roc. {S}ympos. {P}ure
              {M}ath., {V}ol. {XXIII}, {U}niv. {C}alifornia, {B}erkeley,
              {C}alif., 1971)},
    SERIES = {Proc. Sympos. Pure Math.},
    VOLUME = {Vol. XXIII},
     PAGES = {231--260},
 PUBLISHER = {Amer. Math. Soc., Providence, RI},
      YEAR = {1973},
   MRCLASS = {49F22},
  MRNUMBER = {385688},
MRREVIEWER = {Klaus\ Steffen},
}

@article{Nardulli,
      title={Allard's interior $\epsilon$-{R}egularity {T}heorem in {A}lexandrov spaces}, 
      author={Agnoletto, Marcos and Hoyos Julio C. Correa and da Silva, M\'{a}rcio Fabiano and Nardulli, Stefano},
      year={2025},
      journal={arXiv},
      pages={2504.10758v3},
      eprint={2504.10758v3},
      archivePrefix={arXiv},
      primaryClass={math.DG}
}

@article {HutchinsonIUMJ,
    AUTHOR = {Hutchinson, John E.},
     TITLE = {Second fundamental form for varifolds and the existence of
              surfaces minimising curvature},
   JOURNAL = {Indiana Univ. Math. J.},
  FJOURNAL = {Indiana University Mathematics Journal},
    VOLUME = {35},
      YEAR = {1986},
    NUMBER = {1},
     PAGES = {45--71},
      ISSN = {0022-2518,1943-5258},
   MRCLASS = {49F22 (58E15)},
  MRNUMBER = {825628},
MRREVIEWER = {Harold\ Parks},
       DOI = {10.1512/iumj.1986.35.35003},
       URL = {https://doi.org/10.1512/iumj.1986.35.35003},
}

@article {LiPariseSarnataro2023,
    AUTHOR = {Li, Martin Man-chun and Parise, Davide and Sarnataro, Lorenzo},
     TITLE = {Boundary behavior of limit-interfaces for the {A}llen-{C}ahn
              equation on {R}iemannian manifolds with {N}eumann boundary
              condition},
   JOURNAL = {Arch. Ration. Mech. Anal.},
  FJOURNAL = {Archive for Rational Mechanics and Analysis},
    VOLUME = {248},
      YEAR = {2024},
    NUMBER = {6},
     PAGES = {Paper No. 124, 43},
      ISSN = {0003-9527,1432-0673},
   MRCLASS = {35R01 (35J25 35J67 49Q15)},
  MRNUMBER = {4831349},
MRREVIEWER = {Haitao\ Wan},
       DOI = {10.1007/s00205-024-02070-z},
       URL = {https://doi.org/10.1007/s00205-024-02070-z},
}

@article {Mantegazza1996,
    AUTHOR = {Mantegazza, Carlo},
     TITLE = {Curvature varifolds with boundary},
   JOURNAL = {J. Differential Geom.},
  FJOURNAL = {Journal of Differential Geometry},
    VOLUME = {43},
      YEAR = {1996},
    NUMBER = {4},
     PAGES = {807--843},
      ISSN = {0022-040X,1945-743X},
   MRCLASS = {49Q20 (58E99)},
  MRNUMBER = {1412686},
MRREVIEWER = {Klaus\ Steffen},
       URL = {http://projecteuclid.org/euclid.jdg/1214458533},
}

@article{Masi2020,
	author = {De Masi, Luigi},
	doi = {10.1512/iumj.2021.70.9401},
	fjournal = {Indiana University Mathematics Journal},
	issn = {0022-2518},
	journal = {Indiana Univ. Math. J.},
	mrclass = {49Q15 (53C42)},
	mrnumber = {4359920},
	mrreviewer = {Mario Santilli},
	number = {6},
	pages = {2603--2651},
	title = {Rectifiability of the free boundary for varifolds},
	url = {https://doi.org/10.1512/iumj.2021.70.9401},
	volume = {70},
	year = {2021}}

@article{Chodosh2023,
	author = {Chodosh, Otis and Mantoulidis, Christos},
	doi = {10.1007/s10240-023-00141-7},
	fjournal = {Publications Math\'{e}matiques. Institut de Hautes \'{E}tudes Scientifiques},
	issn = {0073-8301},
	journal = {Publ. Math. Inst. Hautes \'{E}tudes Sci.},
	mrclass = {53C42 (49Q15 49Q20 58J50)},
	mrnumber = {4588597},
	pages = {245--342},
	title = {The p-widths of a surface},
	url = {https://doi.org/10.1007/s10240-023-00141-7},
	volume = {137},
	year = {2023}}

@article{Chan2018,
	author = {Chan, Hardy and Wei, Juncheng},
	doi = {10.1007/s11425-017-9307-4},
	fjournal = {Science China. Mathematics},
	issn = {1674-7283},
	journal = {Sci. China Math.},
	mrclass = {35J91 (35B08)},
	mrnumber = {3864761},
	number = {11},
	pages = {1925--1946},
	title = {On {D}e {G}iorgi's conjecture: recent progress and open problems},
	url = {https://doi.org/10.1007/s11425-017-9307-4},
	volume = {61},
	year = {2018}}

@article{Scharrer2022,
	author = {Scharrer, Christian},
	doi = {10.1007/s10455-021-09822-0},
	fjournal = {Annals of Global Analysis and Geometry},
	issn = {0232-704X,1572-9060},
	journal = {Ann. Global Anal. Geom.},
	mrclass = {49Q15 (28A75 53C21)},
	mrnumber = {4423121},
	mrreviewer = {Daniele\ Semola},
	number = {4},
	pages = {691--719},
	title = {Some geometric inequalities for varifolds on {R}iemannian manifolds based on monotonicity identities},
	url = {https://doi.org/10.1007/s10455-021-09822-0},
	volume = {61},
	year = {2022}}

@article{Li2021,
	author = {Li, Martin Man-Chun and Zhou, Xin},
	doi = {10.4310/jdg/1625860624},
	issn = {0022-040X},
	journal = {Journal of Differential Geometry},
	mrnumber = {4285846},
	number = {3},
	pages = {487--553},
	title = {Min-max theory for free boundary minimal hypersurfaces {I}---{R}egularity theory},
	volume = {118},
	year = {2021}}

@article{Allard1972,
	author = {Allard, William K.},
	doi = {10.2307/1970868},
	issn = {0003-486X},
	journal = {Annals of Mathematics. Second Series},
	mrnumber = {307015},
	pages = {417--491},
	title = {On the first variation of a varifold},
	volume = {95},
	year = {1972}}

@article{Mantoulidis2021,
	author = {Mantoulidis, Christos},
	doi = {10.4310/jdg/1609902018},
	issn = {0022-040X},
	journal = {Journal of Differential Geometry},
	mrnumber = {4195753},
	number = {1},
	pages = {93--135},
	title = {Allen-{C}ahn min-max on surfaces},
	volume = {117},
	year = {2021}}

@article{Chodosh2020,
	author = {Chodosh, Otis and Mantoulidis, Christos},
	doi = {10.4007/annals.2020.191.1.4},
	issn = {0003-486X},
	journal = {Annals of Mathematics. Second Series},
	mrnumber = {4045964},
	number = {1},
	pages = {213--328},
	title = {Minimal surfaces and the {A}llen-{C}ahn equation on 3-manifolds: index, multiplicity, and curvature estimates},
	volume = {191},
	year = {2020}}

@article{Grueter1986,
	author = {Gr\"{u}ter, Michael and Jost, J\"{u}rgen},
	issn = {0391-173X},
	journal = {Annali della Scuola Normale Superiore di Pisa. Classe di Scienze. Serie IV},
	mrnumber = {863638},
	number = {1},
	pages = {129--169},
	title = {Allard type regularity results for varifolds with free boundaries},
	volume = {13},
	year = {1986}}

@article{Guaraco2018,
	author = {Guaraco, Marco A. M.},
	doi = {10.4310/jdg/1513998031},
	issn = {0022-040X},
	journal = {Journal of Differential Geometry},
	mrnumber = {3743704},
	number = {1},
	pages = {91--133},
	title = {Min-max for phase transitions and the existence of embedded minimal hypersurfaces},
	volume = {108},
	year = {2018}}

@article{Hutchinson2000,
	author = {Hutchinson, John E. and Tonegawa, Yoshihiro},
	doi = {10.1007/PL00013453},
	issn = {0944-2669},
	journal = {Calculus of Variations and Partial Differential Equations},
	mrnumber = {1803974},
	number = {1},
	pages = {49--84},
	title = {Convergence of phase interfaces in the van der {W}aals-{C}ahn-{H}illiard theory},
	volume = {10},
	year = {2000}}

@article{Tonegawa2012,
	author = {Tonegawa, Yoshihiro and Wickramasekera, Neshan},
	doi = {10.1515/crelle.2011.134},
	issn = {0075-4102},
	journal = {Journal f\"{u}r die Reine und Angewandte Mathematik. [Crelle's Journal]},
	mrnumber = {2948876},
	pages = {191--210},
	title = {Stable phase interfaces in the van der {W}aals--{C}ahn--{H}illiard theory},
	volume = {668},
	year = {2012}}

@article{Tonegawa2005a,
	author = {Tonegawa, Yoshihiro},
	issn = {1019-8385},
	journal = {Communications in Analysis and Geometry},
	mrnumber = {2154826},
	number = {2},
	pages = {439--459},
	title = {On stable critical points for a singular perturbation problem},
	volume = {13},
	year = {2005}}

@article{Tonegawa2003,
	author = {Tonegawa, Yoshihiro},
	doi = {10.1512/iumj.2003.52.2351},
	issn = {0022-2518},
	journal = {Indiana University Mathematics Journal},
	mrnumber = {1970021},
	number = {1},
	pages = {69--83},
	title = {Domain dependent monotonicity formula for a singular perturbation problem},
	volume = {52},
	year = {2003}}

@book{Simon1983,
	author = {Simon, Leon},
	isbn = {0-86784-429-9},
	mrnumber = {756417},
	pages = {vii+272},
	publisher = {Australian National University, Centre for Mathematical Analysis, Canberra},
	series = {Proceedings of the Centre for Mathematical Analysis, Australian National University},
	title = {Lectures on geometric measure theory},
	volume = {3},
	year = {1983}}

@article{Padilla1998,
	author = {Padilla, Pablo and Tonegawa, Yoshihiro},
	doi = {10.1002/(SICI)1097-0312(199806)51:6<551::AID-CPA1>3.0.CO;2-6},
	issn = {0010-3640},
	journal = {Communications on Pure and Applied Mathematics},
	mrnumber = {1611144},
	number = {6},
	pages = {551--579},
	title = {On the convergence of stable phase transitions},
	volume = {51},
	year = {1998}}

@article{Modica1987a,
	author = {Modica, Luciano},
	doi = {10.1007/BF00251230},
	issn = {0003-9527},
	journal = {Archive for Rational Mechanics and Analysis},
	mrnumber = {866718},
	number = {2},
	pages = {123--142},
	title = {The gradient theory of phase transitions and the minimal interface criterion},
	volume = {98},
	year = {1987}}

@article{Malchiodi2007,
	author = {Malchiodi, A. and Wei, Juncheng},
	doi = {10.1007/s11784-007-0016-7},
	issn = {1661-7738},
	journal = {Journal of Fixed Point Theory and Applications},
	mrnumber = {2336614},
	number = {2},
	pages = {305--336},
	title = {Boundary interface for the {A}llen-{C}ahn equation},
	volume = {1},
	year = {2007}}

@article{Malchiodi2007a,
	author = {Malchiodi, Andrea and Ni, Wei-Ming and Wei, Juncheng},
	doi = {10.2140/pjm.2007.229.447},
	issn = {0030-8730},
	journal = {Pacific Journal of Mathematics},
	mrnumber = {2276519},
	number = {2},
	pages = {447--468},
	title = {Boundary-clustered interfaces for the {A}llen-{C}ahn equation},
	volume = {229},
	year = {2007}}

@article{Kohn1989,
	author = {Kohn, Robert V. and Sternberg, Peter},
	doi = {10.1017/S0308210500025026},
	issn = {0308-2105},
	journal = {Proceedings of the Royal Society of Edinburgh. Section A. Mathematics},
	mrnumber = {985990},
	number = {1-2},
	pages = {69--84},
	title = {Local minimisers and singular perturbations},
	volume = {111},
	year = {1989}}

@article{Edelen2020,
	author = {Edelen, Nick},
	doi = {10.1515/crelle-2017-0053},
	issn = {0075-4102},
	journal = {Journal f\"{u}r die Reine und Angewandte Mathematik. [Crelle's Journal]},
	mrnumber = {4048443},
	pages = {95--137},
	title = {The free-boundary {B}rakke flow},
	volume = {758},
	year = {2020}}

@article{Hiesmayr2017,
  title={Spectrum and index of two-sided {A}llen–-{C}ahn minimal hypersurfaces},
  author={Fritz Hiesmayr},
  journal={Communications in Partial Differential Equations},
  year={2017},
  volume={43},
  pages={1541--1565},
}

@article {SchoenSimon81,
    AUTHOR = {Schoen, Richard and Simon, Leon},
     TITLE = {Regularity of stable minimal hypersurfaces},
   JOURNAL = {Comm. Pure Appl. Math.},
  FJOURNAL = {Communications on Pure and Applied Mathematics},
    VOLUME = {34},
      YEAR = {1981},
    NUMBER = {6},
     PAGES = {741--797},
      ISSN = {0010-3640,1097-0312},
   MRCLASS = {49F22 (53C42 58E15)},
  MRNUMBER = {634285},
MRREVIEWER = {F.\ J.\ Almgren, Jr.},
       DOI = {10.1002/cpa.3160340603},
       URL = {https://doi.org/10.1002/cpa.3160340603},
}

@article {Federer1970,
    AUTHOR = {Federer, Herbert},
     TITLE = {The singular sets of area minimizing rectifiable currents with
              codimension one and of area minimizing flat chains modulo two
              with arbitrary codimension},
   JOURNAL = {Bull. Amer. Math. Soc.},
  FJOURNAL = {Bulletin of the American Mathematical Society},
    VOLUME = {76},
      YEAR = {1970},
     PAGES = {767--771},
      ISSN = {0002-9904},
   MRCLASS = {28.80 (26.00)},
  MRNUMBER = {260981},
MRREVIEWER = {J.\ E.\ Brothers},
       DOI = {10.1090/S0002-9904-1970-12542-3},
       URL = {https://doi.org/10.1090/S0002-9904-1970-12542-3},
}

@book {Almgren2000,
    AUTHOR = {Almgren, Jr., Frederick J.},
     TITLE = {Almgren's big regularity paper},
    SERIES = {World Scientific Monograph Series in Mathematics},
    VOLUME = {1},
      NOTE = {$Q$-valued functions minimizing Dirichlet's integral and the
              regularity of area-minimizing rectifiable currents up to
              codimension 2,
              With a preface by Jean E.\ Taylor and Vladimir Scheffer},
 PUBLISHER = {World Scientific Publishing Co., Inc., River Edge, NJ},
      YEAR = {2000},
     PAGES = {xvi+955},
      ISBN = {981-02-4108-9},
   MRCLASS = {49-02 (35J20 49N60 49Q20 58E12)},
  MRNUMBER = {1777737},
MRREVIEWER = {Brian\ Cabell\ White},
}

@article {White1997,
    AUTHOR = {White, Brian},
     TITLE = {Stratification of minimal surfaces, mean curvature flows, and
              harmonic maps},
   JOURNAL = {J. Reine Angew. Math.},
  FJOURNAL = {Journal f\"ur die Reine und Angewandte Mathematik. [Crelle's
              Journal]},
    VOLUME = {488},
      YEAR = {1997},
     PAGES = {1--35},
      ISSN = {0075-4102,1435-5345},
   MRCLASS = {49Q05 (49Q20 53C42 58E20)},
  MRNUMBER = {1465365},
MRREVIEWER = {Harold\ Parks},
       DOI = {10.1515/crll.1997.488.1},
       URL = {https://doi.org/10.1515/crll.1997.488.1},
}

@book {Simon1996,
    AUTHOR = {Simon, Leon},
     TITLE = {Theorems on regularity and singularity of energy minimizing
              maps},
    SERIES = {Lectures in Mathematics ETH Z\"urich},
      NOTE = {Based on lecture notes by Norbert Hungerb\"uhler},
 PUBLISHER = {Birkh\"auser Verlag, Basel},
      YEAR = {1996},
     PAGES = {viii+152},
      ISBN = {3-7643-5397-X},
   MRCLASS = {58E20 (35J60 49N60 58G03)},
  MRNUMBER = {1399562},
MRREVIEWER = {Nathan\ Smale},
       DOI = {10.1007/978-3-0348-9193-6},
       URL = {https://doi.org/10.1007/978-3-0348-9193-6},
}

@article {Bellettini2025,
    AUTHOR = {Bellettini, Costante},
     TITLE = {Extensions of {S}choen-{S}imon-{Y}au and {S}choen-{S}imon
              theorems via iteration \`a{} la {D}e {G}iorgi},
   JOURNAL = {Invent. Math.},
  FJOURNAL = {Inventiones Mathematicae},
    VOLUME = {240},
      YEAR = {2025},
    NUMBER = {1},
     PAGES = {1--34},
      ISSN = {0020-9910,1432-1297},
   MRCLASS = {49Q20 (58E12)},
  MRNUMBER = {4871955},
       DOI = {10.1007/s00222-025-01317-0},
       URL = {https://doi.org/10.1007/s00222-025-01317-0},
}

@article {Wick14,
    AUTHOR = {Wickramasekera, Neshan},
     TITLE = {A general regularity theory for stable codimension 1 integral
              varifolds},
   JOURNAL = {Ann. of Math. (2)},
  FJOURNAL = {Annals of Mathematics. Second Series},
    VOLUME = {179},
      YEAR = {2014},
    NUMBER = {3},
     PAGES = {843--1007},
      ISSN = {0003-486X,1939-8980},
   MRCLASS = {49Q20 (49N60 58E05 58K99)},
  MRNUMBER = {3171756},
MRREVIEWER = {S\l awomir\ Kolasi\'nski},
       DOI = {10.4007/annals.2014.179.3.2},
       URL = {https://doi.org/10.4007/annals.2014.179.3.2},
}

\end{document}